\documentclass[a4paper,reqno]{amsart}
\usepackage[english]{babel}
\usepackage{hyperref}

\usepackage{graphicx}
\usepackage{amssymb,amsthm,amsmath}
\usepackage{mathrsfs}

\usepackage{enumitem}
\setlist[enumerate]{leftmargin=*}

\usepackage{tikz}
\usetikzlibrary{patterns}

\numberwithin{equation}{section}

\newtheorem{thm}{Theorem}[section]
\newtheorem{dfn}[thm]{Definition}
\newtheorem{lem}[thm]{Lemma}
\newtheorem{rem}[thm]{Remark}
\newtheorem{cor}[thm]{Corollary}
\newtheorem{prop}[thm]{Proposition}
\newtheorem*{thm*}{Theorem}

\newcommand{\RR}{\mathbb{R}}
\newcommand{\CC}{\mathbb{C}}
\newcommand{\NN}{\mathbb{N}}
\newcommand{\ZZ}{\mathbb{Z}}

\newcommand{\Four}{\mathcal{F}} 

\newcommand{\Sch}{\mathscr{S}}  

\newcommand{\loc}{\mathrm{loc}}     
\newcommand{\rad}{\mathrm{rad}}     
\newcommand{\HS}{\mathrm{HS}}       

\newcommand{\thr}{\varsigma_+}  

\newcommand{\lie}{\mathfrak}    

\newcommand{\Dyad}{\mathbb{D}}   

\newcommand{\Adm}{\mathrm{Adm}}   

\newcommand{\HypF}{{}_{2}F_{1}}   

\newcommand{\fstW}{\rho}          
\newcommand{\fstZ}{\zeta}         
\newcommand{\sndW}{\psi}          

\newcommand{\skwJ}{\mathscr{J}}   

\newcommand{\SqT}{\mathscr{T}}   

\newcommand{\fDil}{D}            

\newcommand{\BRm}{\mathrm{m}}          
\newcommand{\tBRm}{\mathrm{\tilde m}}  
\newcommand{\BRn}{\mathrm{n}}          

\newcommand{\rI}{\mathrm{I}}
\newcommand{\rII}{\mathrm{II}}

\newcommand{\opL}{\mathcal{L}}     
\newcommand{\ctDer}{\mathscr{U}}   

\DeclareMathOperator{\hdeg}{hdeg}   

\newcommand{\chr}{\mathbf{1}}   

\newcommand{\hdim}{Q}           
\newcommand{\tdim}{D}           
\newcommand{\wdim}{Q_*}         

\newcommand{\defeq}{\mathrel{:=}}

\renewcommand{\Re}{\operatorname{Re}}  
\renewcommand{\Im}{\operatorname{Im}}  

\newcommand{\id}{\operatorname{Id}}   

\DeclareMathOperator{\supp}{supp}     
\DeclareMathOperator{\tr}{tr}         

\newcommand{\maxf}{\ast}              
\newcommand{\lmaxf}{\bullet}          

\newcommand{\adj}{\dagger}            

\newcommand{\tc}{\,:\,}               

\title[Bochner--Riesz means on Heisenberg-type groups]{Almost everywhere convergence of Bochner--Riesz means on Heisenberg-type groups}

\author{Adam D. Horwich}
\address[A.D. Horwich]{School of Mathematics \\ University of Birmingham \\ Edgbaston \\ Birmingham \\ B15 2TT \\ United Kingdom}
\email{adh192@alumni.bham.ac.uk}

\author{Alessio Martini}
\address[A. Martini]{School of Mathematics \\ University of Birmingham \\ Edgbaston \\ Birmingham \\ B15 2TT \\ United Kingdom}
\email{a.martini@bham.ac.uk}

\thanks{A.D.H.\ was supported by a studentship from EPSRC (Award Reference 1649508). A.M.\ was partially supported by the EPSRC Grant ``Sub-Elliptic Harmonic Analysis'' (EP/P002447/1).}



\begin{document}

\begin{abstract}
We prove an almost everywhere convergence result for Bochner--Riesz means of $L^p$ functions on Heisenberg-type groups, yielding the existence of a $p>2$ for which convergence holds for means of arbitrarily small order. The proof hinges on a reduction of weighted $L^2$ estimates for the maximal Bochner--Riesz operator to corresponding estimates for the non-maximal operator, and a `dual Sobolev trace lemma', whose proof is based on refined estimates for Jacobi polynomials.
\end{abstract}

\maketitle

\section{Introduction}\label{s:intro}

The study of Bochner--Riesz means is a classical topic in harmonic analysis. Recall that the Bochner--Riesz means of order $\lambda \geq 0$ of any function $f\in L^2(\RR^d)$ are defined by 
\begin{equation}\label{eq:bochnerriesz}
T_r^\lambda f \defeq (1-r\opL)_+^\lambda f,
\end{equation}
where $\opL = \Delta \defeq -\sum_{j=1}^d \partial_j^2$ is the Euclidean Laplacian and $r \in \RR^+ \defeq (0,\infty)$. The associated maximal Bochner--Riesz operator is then given by 
\begin{equation}\label{eq:maximalbochnerriesz}
T_\maxf^\lambda f \defeq \sup_{r>0} |(1-r\opL)_+^\lambda f|.
\end{equation}
The problem of under what conditions and in which sense one may ensure that $T^\lambda_r f$ converges to $f$ as $r \to 0^+$ is a key
part of the investigation of summability methods for the Fourier inversion formula, with connections to many other fundamental problems in harmonic analysis and PDE (see, e.g., \cite{35,Grafakos,LuYan,15,36}).

A question of particular interest is the range of $\lambda\geq 0$ and $p \in [1,\infty]$ for which $T_r^\lambda$ and $T_\maxf^\lambda$ are bounded on $L^p(\RR^d)$; the believed best bound on this is known as the Bochner--Riesz conjecture (respectively, maximal Bochner--Riesz conjecture). It is conjectured that, for $\lambda > 0$, the operator $T_r^\lambda$ is bounded on $L^p(\RR^d)$ if and only if
\begin{equation}\label{eq:brconj}
    \frac{d-1}{d} \left(\frac{1}{2} - \frac{\lambda}{d-1}\right) < \frac{1}{p} < \frac{d+1}{d} \left(\frac{1}{2} + \frac{\lambda}{d+1}\right),
\end{equation}
and, for $p \geq 2$, the same $L^p$ boundedness range is conjectured for $T_\maxf^\lambda$.
A number of partial results in this direction have been obtained, including recent breakthroughs (see \cite{45,40,39,Christ_BochnerRiesz,GuthHickmanIliopoulou,42,43,TaoMaximal} and references therein), but the full conjectures remain open.

A weaker property than $L^p$ boundedness of $T_*^\lambda$ is the almost everywhere convergence of $T_r^\lambda f$ to $f$ as $r\to 0^+$ for all $f\in L^p$. While the maximal Bochner--Riesz conjecture remains open, almost everywhere convergence has been proved \cite{CRV} in the range \eqref{eq:brconj} for $p \geq 2$ (see also \cite{Annoni,LeeSeeger} for more recent endpoint results).

\begin{thm*}[Carbery, Rubio de Francia and Vega]
Let $\opL$ be the Laplacian on $\RR^d$. Let $\lambda>0$ and $2\leq p \leq \infty$ be such that
\[
\frac{d-1}{d}\left(\frac{1}{2}-\frac{\lambda}{d-1}\right)<\frac{1}{p}\leq\frac{1}{2}.
\]
Then $T_r^\lambda f$ converges to $f$ almost everywhere as $r\to 0^+$ for all $f\in L^p(\RR^d)$. 
\end{thm*}

As the Laplacian on $\RR^d$ is a positive self-adjoint operator, it has a spectral resolution that may be used to define the Bochner--Riesz operators \eqref{eq:bochnerriesz} and \eqref{eq:maximalbochnerriesz}. As such, we may extend the notion of Bochner--Riesz operators to other positive self-adjoint operators $\opL$ on $L^2(X)$ for some measure space $X$. This corresponds to investigating `Fourier summability' for more general eigenfunction expansions than the one determined by the Euclidean Laplacian.

Here we are concerned with (homogeneous left-invariant) sub-Laplacians $\opL$ on stratified Lie groups. The current understanding of the optimal ranges for $L^p$ boundedness and almost everywhere convergence of Bochner--Riesz means is rather limited in this context, compared to the Euclidean case.
A particularly significant result 
is that of Gorges and M\"uller \cite{GM}, that extends the result of Carbery, Rubio de Francia and Vega to the setting of Heisenberg groups $H_m$.

\begin{thm*}[Gorges and M\"uller]
Let $\opL$ be the sub-Laplacian on the Heisenberg group $H_m$. Let $\hdim=2m+2$ and $\tdim=2m+1$.
Let $\lambda>0$ and $2\leq p \leq \infty$ be such that
\begin{equation}\label{eq:gmcond}
    \frac{\hdim-1}{\hdim}\left(\frac{1}{2}-\frac{\lambda}{\tdim-1}\right)<\frac{1}{p}\leq\frac{1}{2}.
\end{equation}
Then $T_r^\lambda f$ converges almost everywhere to $f$ as $r\to 0^+$ for all $f\in L^p(H_m)$.
\end{thm*}

We remark that the quantities represented by $\hdim$ and $\tdim$, namely the homogeneous and topological dimension of the group $H_m$ respectively, make sense for any stratified Lie group (see Section \ref{s:anaongroups} below for details)  and are both equal to $d$ for $\RR^d$.

The above theorem should be compared with the following general result by Mauceri and Meda \cite[Corollary 2.8]{MauceriMeda}, which is valid for any stratified group and concerns $L^p$ boundedness of the maximal Bochner--Riesz operator on such groups (see also \cite{HJ,Mauceri,MauceriSymposia,M_rieszmeans}).

\begin{thm*}[Mauceri and Meda]
Let $\opL$ be a sub-Laplacian on a stratified group $G$ of homogeneous dimension $Q$. Let $\lambda>0$ and $2\leq p \leq \infty$ be such that
\begin{equation}\label{eq:mercmedacond}
    \frac{1}{2} - \frac{\lambda}{Q-1} < \frac{1}{p} \leq \frac{1}{2}.
\end{equation}
Then the maximal operator $T_*^\lambda$ extends to a bounded operator on $L^p(G)$. In particular, $T_r^\lambda f$ converges almost everywhere to $f$ as $r\to 0^+$ for all $f \in L^p(G)$.
\end{thm*}

The condition \eqref{eq:mercmedacond} is more restrictive than \eqref{eq:gmcond}; however, Mauceri and Meda's result applies to a larger class of groups and gives in the range \eqref{eq:mercmedacond} a stronger property than almost everywhere convergence. 
A natural question is to what extent it is possible to obtain almost everywhere convergence beyond the range \eqref{eq:mercmedacond} for groups other than the $H_m$.
A particularly elusive problem is obtaining a range with the same `trapezoidal' shape as \eqref{eq:gmcond}, that is, such that a $p>2$ exists for which all $\lambda > 0$ are admissible (see Figure \ref{fig:mainfig}); apart from the pioneering work by Gorges and M\"uller, we are not aware of results of this kind for nonelliptic sub-Laplacians $\opL$, even outside the context of stratified groups.

\begin{figure}[ht]
\centering
\begin{tikzpicture}[scale=0.55]
\coordinate (O) at (0,0);
\coordinate (OX) at (25*1/2+1,0);
\coordinate (OT) at (0,7);
\coordinate (L1Q) at (0,4*3/2);
\coordinate (L1D) at (0,4*1);
\coordinate (H) at (25*1/2,0);
\coordinate (HT) at (25*1/2,7);
\coordinate (GM) at (25*3/8,0);
\coordinate (HM) at (25*0.41,0);
\draw[white, pattern=dots, pattern color=green] (GM)-- (H)-- (HT) -- (OT)-- (L1D) -- cycle;
\draw[white, pattern=north east lines, pattern color=red] (HM) -- (H) -- (HT) -- (OT) -- (L1D) --cycle;
\draw[fill=blue, opacity=0.2] (H) -- (HT) -- (OT) -- (L1Q) -- (H);
\draw[black, ->] (O) -- (OX) node[anchor=north] {$\frac{1}{p}$};
\draw[black, ->] (O) -- (OT) node[anchor=east] {$\lambda$}; 
\draw[black] (H)-- (HT);
\filldraw[black] (O) circle (1pt) node[anchor=north] {$0$}; 
\filldraw[black] (H) circle (2pt) node[anchor=north] {$\frac{1}{2}$}; 
\filldraw[black] (GM) circle (2pt) node[anchor=55] {$\frac{\hdim-1}{2\hdim}$}; 
\filldraw[black] (HM) circle (2pt) node[anchor=125] {$\frac{\wdim-1}{2\wdim}$}; 
\filldraw[black] (L1Q) circle (2pt) node[anchor=east] {$\frac{\hdim-1}{2}$}; 
\filldraw[black] (L1D) circle (2pt) node[anchor=east] {$\frac{\tdim-1}{2}$}; 
\draw[blue, dashed] (H) -- (L1Q); 
\draw[red, dashed] (HM) -- (L1D); 
\draw[green,dashed] (GM) -- (L1D);
\draw[green, pattern=dots, pattern color=green] (13.5,3) circle (5pt) node[anchor=west,label={[align=center,color=green, opacity=1]right:Gorges \& M\"uller}] {};
\draw[red, pattern=north east lines, pattern color=red] (13.5,3+.7) circle (5pt) node[anchor=west,label={[align=center,color=red, opacity=1]right:Theorem \ref{thm:mainthm}}] {};
\draw[blue, fill=blue, opacity=0.2] (13.5,3+2*.7) circle (5pt) node[anchor=west,label={[align=center,color=blue, opacity=1]right: Mauceri \& Meda}] {};
\end{tikzpicture}

\caption[Almost everywhere convergence of Bochner--Riesz means]{{Range of almost everywhere convergence of Bochner--Riesz means on H-type groups given by Theorem \ref{thm:mainthm}. The diagram also depicts the results by Gorges and M\"uller (valid for Heisenberg groups only) and Mauceri and Meda.}}
\label{fig:mainfig}
\end{figure}
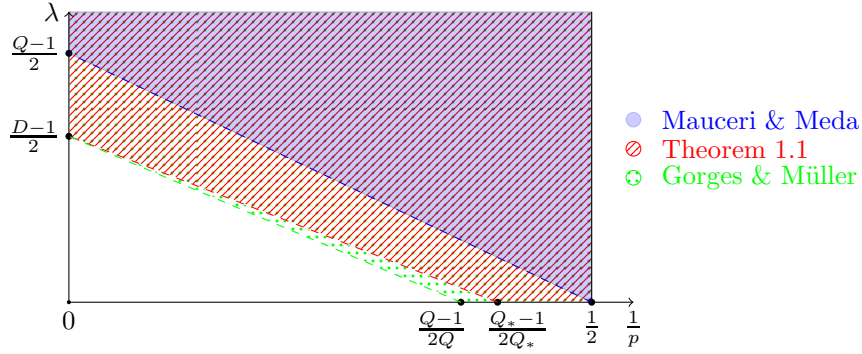

Here we succeed in proving almost everywhere convergence in a `trapezoidal' range 
in the setting of Heisenberg-type (henceforth H-type) groups. This is a class of $2$-step stratified Lie groups that includes the Heisenberg groups $H_m$, as well as groups with higher-dimensional centre \cite{Kaplan}.
Our main result reads as follows.

\begin{thm}\label{thm:mainthm}
Let $\opL$ be the sub-Laplacian on an H-type group $G$ of homogeneous dimension $\hdim$ and topological dimension $\tdim$, and set $\wdim = 2\hdim-\tdim$. Let $\lambda>0$ and $2\leq p \leq \infty$ be such that
\begin{equation}\label{eq:mainthm_range}
\frac{\wdim-1}{\wdim}\left(\frac{1}{2}-\frac{\lambda}{\tdim-1}\right)<\frac{1}{p}\leq\frac{1}{2}.
\end{equation}
Then $T_r^\lambda f$ converges almost everywhere to $f$ as $r\to 0^+$ for all $f\in L^p(G)$.
\end{thm}

Observe that, if $2m$ and $n$ are the dimensions of the first and second layers of the H-type group $G$, then $\tdim=2m+n$, $\hdim = 2m+2n$ and $\wdim = 2m+3n$.
The range \eqref{eq:mainthm_range} is smaller than \eqref{eq:gmcond}; however Gorges and M\"uller's result only applies for $n=1$, while Theorem \ref{thm:mainthm} applies for arbitrary $n\geq 1$.

As in other works on the subject, the proof of our almost everywhere convergence result is obtained by considering $L^p$ to $L^2_{\loc}$ boundedness of the maximal Bochner--Riesz operator. As a matter of fact, it is enough to consider the `local' maximal Bochner--Riesz operator defined by
\begin{equation}
    T^\lambda_\lmaxf f \defeq \sup_{0<r<1}|T^\lambda_r f|.
\end{equation}
Indeed, if $\|\chr_K T^\lambda_\lmaxf\|_{L^p\to L^2}<\infty$ for all compact sets $K\subseteq G$ (here $\chr_K$ denotes the characteristic function of $K$), then Sobolev embeddings for sub-Laplacians \cite{10} and a standard three-$\epsilon$ argument imply the almost everywhere convergence of $T_r^\lambda f$ to $f$ as $r\to 0^+$ for all $f \in L^p(G)$.

As usual in this context, we consider a dyadic decomposition of the Bochner--Riesz multiplier: for $\zeta>0$ and $\Dyad_0 \defeq \{2^{-k} \tc k\in\NN_0\}$, we may write
\begin{equation}\label{eq:multiplierdecompose}
(1-\zeta)^\lambda_+ = \sum_{\delta\in\Dyad_0} \delta^{\lambda} \BRm_{\delta}(\zeta),
\end{equation} 
where, for all $j\in\NN_0$ and $\delta\in\Dyad_0$, the function $\BRm_{\delta} \in C^\infty_c(\RR)$ is real-valued and satisfies
\begin{equation}\label{eq:mdeltabnd}
\|\BRm_{\delta}^{(j)}\|_\infty\lesssim_j \delta^{-j} \text{ and } \supp(\BRm_{\delta})\subseteq \begin{cases}
[1-\delta,1] &\text{ if $\delta < 1$,}\\
[-1,1] &\text{ if $\delta=1$.}
\end{cases}
\end{equation}
Note that the functions $\BRm_\delta$ in \eqref{eq:multiplierdecompose} depend on $\lambda$, but satisfy \eqref{eq:mdeltabnd} with implicit constants independent of $\lambda$; hence, with a slight abuse of notation,
we suppress the 
dependence on $\lambda$ of the functions $\BRm_\delta$ from their notation.

Let us define the maximal operators corresponding to the dyadic decomposition:
\begin{equation}\label{eq:mdecomposemax}
   M^\maxf_\delta f \defeq \sup_{r>0} |\BRm_\delta(r\opL) f|, \qquad  M^\lmaxf_\delta f \defeq \sup_{0<r<1}|\BRm_\delta(r\opL) f|
\end{equation}
In view of \eqref{eq:multiplierdecompose}, for any given $p \in [2,\infty]$ and $\lambda_0 \in \RR$, the $L^p$ to $L^2_\loc$ boundedness of $T^\lambda_\lmaxf$ for all $\lambda > \lambda_0$ would follow
from an estimate of the form
\begin{equation}\label{eq:410}
    \|\chr_K M_\delta^\lmaxf\|_{L^p \to L^2}
		\lesssim \delta^{-\lambda_0}
\end{equation}
for all $\delta \in \Dyad_0$ and all compact sets $K \subseteq G$, where the implicit constant depends only on those in \eqref{eq:mdeltabnd} and on the compact set $K$, but not on $\delta$. In order to prove Theorem \ref{thm:mainthm}, it is then enough to prove \eqref{eq:410} for all pairs $(1/p,\lambda_0)$ lying in the `infinite trapezoid' depicted in Figure \ref{fig:mainfig}.

As a matter of fact, thanks to interpolation \cite{CZ_interpolation}, it suffices to consider just the vertices of the trapezoid, i.e., the estimates
\begin{align}
\label{eq:est_infinity2} \|\chr_K M_\delta^\lmaxf \|_{L^\infty \to L^2} &\lessapprox \delta^{-(\tdim-1)/2} \\
\label{eq:est_estp2} \|\chr_K M_\delta^\lmaxf \|_{L^{2\wdim/(\wdim-1)} \to L^2} &\lessapprox 1, \\
\label{eq:est_22} \|\chr_K M_\delta^\lmaxf \|_{L^{2} \to L^2} &\lessapprox 1, 
\end{align}
where
$\lessapprox$ stands for $\lesssim C(\epsilon) \,\delta^{-\epsilon}$ for all
arbitrarily small $\epsilon>0$.

Among these, the estimates \eqref{eq:est_infinity2} and \eqref{eq:est_22} actually follow from stronger $L^p$ estimates for the `global' maximal operator $M_\delta^\maxf$,
which can be obtained in a relatively straightforward way using available estimates for functions of a sub-Laplacian. More precisely, for a general stratified group $G$ and sub-Laplacian $\opL$, one can prove the estimates
\begin{equation}\label{eq:general_maximal_ests}
\|M_\delta^\maxf \|_{L^\infty \to L^\infty} \lessapprox \delta^{1/2-\thr(\opL)}, \qquad
\|M_\delta^\maxf \|_{L^{2} \to L^2} \lesssim 1, 
\end{equation}
where $\thr(\opL)$ is the ``Mihlin--H\"ormander threshold'' for $\opL$ defined as in \cite{MM_gafa}; it is known that $\tdim/2 \leq \thr(\opL) \leq \hdim/2$ for arbitrary stratified groups and sub-Laplacians \cite{Christ_multipliers,MMG_lowerbound,MauceriMeda}, that $\thr(\opL) < \hdim/2$ for all $2$-step stratified groups \cite{MM_gafa}, and that $\thr(\opL) = \tdim/2$ for several classes of $2$-step stratified groups, including the H-type groups \cite{Hebisch_multipliers,M_HeisenbergReiter,MM_newclasses,MM_N32,MS_heisenberg}.
In view of \eqref{eq:multiplierdecompose}, the estimates \eqref{eq:general_maximal_ests} immediately lead to the following improvement of the result by Mauceri and Meda.

\begin{thm}\label{thm:mainMM}
Let $\opL$ be a sub-Laplacian on a stratified group $G$. Let $\lambda>0$ and $2\leq p \leq \infty$ be such that
\[
    \frac{1}{2}-\frac{\lambda}{2\thr(\opL)-1}<\frac{1}{p}\leq\frac{1}{2}.
\]
Then the maximal operator $T_*^\lambda$ extends to a bounded operator on $L^p(G)$. In particular, $T_r^\lambda f$ converges almost everywhere to $f$ as $r\to 0^+$ for all $f \in L^p(G)$.
\end{thm}

The estimate \eqref{eq:est_estp2}, instead, requires a more delicate analysis, which we develop for an H-type group $G$. By H\"older's inequality, it is readily seen that \eqref{eq:est_estp2} follows from the estimate
\begin{equation}\label{eq:weightedmaxopest_inter}
\|\chr_K M_\delta^\lmaxf \|_{L^{2}(\omega_*) \to L^2} \lessapprox 1, 
\end{equation}
where $\omega_*(z,u) = (1+|z|)^{-2m/\wdim} (1+|u|)^{-n/\wdim}$ in the usual exponential coordinates (here $z$ and $u$ correspond to the first and second layer of $G$ respectively);
in turn \eqref{eq:weightedmaxopest_inter} can be easily deduced by interpolating the weighted estimates
\begin{equation}\label{eq:weightedmaxopest}
\| M_\delta^\lmaxf \|_{L^{2}((1+|\cdot|)^{-a}) \to L^2((1+|\cdot|)^{-a})} \lessapprox 1, \qquad \| M_\delta^\lmaxf \|_{L^{2}((1+\fstW)^{-b}) \to L^2((1+\fstW)^{-b})} \lessapprox 1,
\end{equation}
where $|\cdot|$ is a homogeneous norm on $G$, $\fstW(z,u) = |z|$, $a= 2/3$ and $b=1$.

As it turns out, the estimates \eqref{eq:weightedmaxopest} reduce, roughly speaking, to the corresponding estimates for the `nonmaximal' operator:
\begin{equation}\label{eq:weightedopest}
\| \BRm_\delta(\opL) \|_{L^{2}((1+|\cdot|)^{-a}) \to L^2((1+|\cdot|)^{-a})} \lessapprox 1, \qquad \| \BRm_\delta(\opL) \|_{L^{2}((1+\fstW)^{-b}) \to L^2((1+\fstW)^{-b})} \lessapprox 1,
\end{equation}
More precisely, 
in Section \ref{s:chapsqfunarg} below we prove that
for a certain class of weights $w$ on an H-type group $G$, the following estimate holds:
\begin{equation}\label{eq:max_nonmax_redux}
\|M_\delta^\lmaxf\|^2_{L^2(w)\to L^2(w)} \lesssim \sup_{s \in (0,1)} \|\BRm_\delta(s\opL)\|_{L^2(w)\to L^2(w)} \sup_{s \in (0,1)} \|\tBRm_\delta(s\opL)\|_{L^2(w)\to L^2(w)}
\end{equation}
for all $\delta \in \Dyad \defeq \Dyad_0 \setminus \{1\}$,
where the implicit constant may depend on $w$, and $\tBRm_\delta(\zeta) \defeq \delta \zeta \BRm_\delta'(\zeta)$; note that the $\tBRm_\delta$ satisfy the same conditions \eqref{eq:mdeltabnd} as the $\BRm_\delta$. The `maximal-to-nonmaximal' reduction estimate \eqref{eq:max_nonmax_redux} actually applies to the weights in \eqref{eq:weightedopest} only if $a \in 4\NN_0$ and $b \in 2\NN_0$; however a more sophisticated `interpolation' argument, presented in Section \ref{s:chapsqfunarg_two}, allows us to work around this restriction and consider fractional powers as well.
While the idea of reducing estimates for the maximal operator to those for the nonmaximal operator is implicit in both the works of Carbery, Rubio de Francia and Vega \cite{CRV} and Gorges and M\"uller \cite{GM}, an explicit estimate such as \eqref{eq:max_nonmax_redux} does not seem to appear in either work, and 
may be of independent interest (cf.\ also \cite{MateuOrobitgVerdera}).

We are now down to proving the weighted estimates \eqref{eq:weightedopest}.
Through this paper,
this will be reduced to 
proving suitable `dual Sobolev trace inequalities', stated as Theorems \ref{thm:tracelemfinal} and \ref{thm:trace5} below.
To briefly explain the idea, in addition to the sub-Laplacian $\opL$, let us fix an orthonormal basis $U_1,\ldots,U_n$ of the second layer $\lie{g}_2$ of the Lie algebra of the H-type group $G$. The operators $\opL$ and $U_j/i$ all commute and admit a joint functional calculus \cite{2}. We define the pseudo-differential operator 
\begin{equation}\label{eq:difflambda}
\ctDer \defeq (-(U_1^2+\ldots+U_n^2))^{{1}/{2}}
\end{equation}
and the spectral cut-off operators $M_{\delta,j}$ by
\[
{M_{\delta,j}} \defeq \begin{cases}
\chr_{[1-\delta,1]}(\opL) \, \chr_{[2^j,2^{j+1})}(2\pi\opL/\ctDer) &\text{for } j=1,\dots,J_\delta-1,\\
\chr_{[1-\delta,1]}(\opL) \, \chr_{[2^{J_\delta},\infty)}(2\pi\opL/\ctDer) &\text{for } j=J_\delta,
\end{cases}
\]
where $\delta \in \Dyad$, and $J_\delta \in \NN$ is such that $2^{J_\delta} \simeq \delta^{-1}$
(see Figure \ref{fig:figdecompose}).
We wish to prove, for all $\delta \in \Dyad$ and $j=1,\dots,J_\delta$,
the estimates
\begin{align}
\label{eq:introest1}
    \|M_{\delta,j}f\|^2_2 &\lessapprox (2^{-j}\delta)^{a/2}\|f\|^2_{L^2((1+|\cdot|)^{a})}, \\
\label{eq:introest2}
    \|M_{\delta,j}f\|^2_2 &\lessapprox (2^{-j})^{b}\|f\|^2_{L^2((1+\fstW)^{b})},
\end{align}
where $a=2/3$ and $b=1$ as before.
Theorems \ref{thm:tracelemfinal} and \ref{thm:trace5} are minor technical modifications of these inequalities.

\begin{figure}
{\begin{tikzpicture}[scale=0.5]
\filldraw[gray, opacity=0.2] (0,5) rectangle (10,5.75);
\draw[thick, ->] (0,0) -- (10,0) node[anchor=north] {$\widehat{\ctDer}$};
\draw[thick, ->] (0,0) -- (0,8) node[anchor=east] {$\widehat{\opL}$};
\filldraw[black] (0,0) circle (1pt) node[anchor=north] {$0$};
\draw[red, thick] (0,0) -- (0.1, 8);
\draw[red, thick] (0,0) -- (0.2, 8);
\draw[red, thick] (0,0) -- (0.3, 8);
\draw[red, thick] (0,0) -- (0.4, 8);
\draw[red, thick] (0,0) -- (0.6, 8);
\draw[red, thick] (0,0) -- (0.9, 8);
\draw[red, thick] (0,0) -- (1.4, 8);
\draw[red, thick] (0,0) -- (2.5, 8);
\draw[red, thick] (0,0) -- (4, 8);
\draw[red, thick] (0,0) -- (6, 8);
\draw[red, thick] (0,0) -- (9, 8);
\draw[black] (10,5.75)--(0,5.75)--(0,5)--(10,5);
\draw[black] (1.03,5)--(1.21,5.75);
\draw[black] (2,5)--(2.29,5.75);
\draw[black] (3.6,5)--(4.15,5.75);
\draw[black] (6,5)--(6.9,5.75);
\filldraw[white, opacity=0] (9.4,5.4) circle (1pt) node[anchor=west,label={[align=center,color=black, opacity=1]right:\small $\chr_{[1-\delta,1]}(\opL)$}] {};
\filldraw[white, opacity=0] (0.31,6.15) circle (1pt) node[anchor=west,label={[align=center,color=black, opacity=1]below:\small$R_{J_\delta}$}] {};
\filldraw[white, opacity=0] (5,6.1) circle (1pt) node[anchor=west,label={[align=center,color=black, opacity=1]below:\small$R_1$}] {};
\filldraw[white, opacity=0] (2.85,6.1) circle (1pt) node[anchor=west,label={[align=center,color=black, opacity=1]below:\small$R_2$}] {};
\filldraw[white, opacity=0] (1.35,5.8) circle (1pt) node[anchor=west,label={[align=center,color=black, opacity=1]below:\small$\ldots$}] {};
\end{tikzpicture}}
  \caption{Joint spectrum of $\opL$ and $\ctDer$ and spectral cut-offs $M_{\delta,j} = \chr_{[1-\delta,1]}(\opL) \, R_{j}$, where $R_j = \chr_{[2^j,2^{j+1})}(2\pi\opL/\ctDer)$ for $j<J_\delta$.}
  \label{fig:figdecompose}
\end{figure}

In the case of Heisenberg groups, a stronger version of the estimate \eqref{eq:introest1}, where $a=1$,
appears in Gorges and M\"uller's paper \cite[Lemmas 7 and 8]{GM}, arising as a replacement for the following Euclidean `dual Sobolev trace inequality'
\begin{equation}\label{eq:crvest}
    \|\chr_{[1-\delta,1]}(\Delta)f\|^2_2\lessapprox\delta\|f\|^2_{L^2(1+|\cdot|)};
\end{equation}
note that \eqref{eq:crvest} is an immediate consequence of the Sobolev trace lemma applied in frequency space, where the norm in the left-hand side of \eqref{eq:crvest} turns into the $L^2$ norm of a function on an annulus of thickness $\delta$, while the norm in the right-hand side becomes the $L^2$ Sobolev norm of order $1/2$ of the function.
The method of Carbery, Rubio de Francia and Vega in the Euclidean case hinges on an estimate such as this \cite[Lemma 3]{CRV}.

In the case of Heisenberg(-type) groups, the proof of the `trace lemmas' \eqref{eq:introest1} and \eqref{eq:introest2} is significantly more complicated than that of \eqref{eq:crvest} in the Euclidean case. Among other things, the group Fourier transform on a noncommutative stratified Lie group has substantially different features from the Euclidean Fourier transform, and describing the effect  on the `Fourier side' of multiplication by a power of the homogeneous norm $|\cdot|$ is not as straightforward as in the Euclidean case, where it can be interpreted as (fractional) differentiation or integration.

The method used by Gorges and M\"uller to prove \eqref{eq:introest1} involves considering negative fractional powers of a difference-differential operator on the Fourier-dual space to the Heisenberg group $H_m$, which corresponds on the group side to the multiplication operator by the function $|z|^2-4iu$. Here we are adopting the usual exponential coordinates $(z,u) \in \CC^m \times \RR$ for the Heisenberg group $H_m$, and we remark that $||z|^2-4iu|^{1/2} = (|z|^4 + 16|u|^2)^{1/4}$ is a homogeneous norm on $H_m$. As it turns out, if one restricts to `radial' functions on $H_m$ (i.e., functions depending only on $|z|$ and $u$), then simple explicit formulas for the Schwartz kernel of these fractional powers can be found \cite[Theorem 11]{GM}, and an application of Schur's Test would readily lead to the estimate \eqref{eq:introest1} for radial functions $f$; a more delicate argument based on complex interpolation allows Gorges and M\"uller to dispense with the radiality constraint and obtain \eqref{eq:introest1} with $a=1$ in full generality.

In the case of Heisenberg-type groups, additional obstacles appear. Here, loosely speaking, $(z,u) \in \CC^m \times \RR^n$, where $n$ may be larger than $1$ (indeed $n>1$ is the case of interest), and the expression $|z|^2-4iu$ no longer makes sense. 
One could consider the function
$|z|^4+16|u|^2$ as a replacement; however, while relatively explicit formulas may be found for the Schwartz kernel of negative fractional powers corresponding to $|z|^4+16|u|^2$, when $n>1$ these formulas become significantly harder to handle.

For this reason, here 
we instead consider the `fractional integration operators' on the Fourier-dual space corresponding to multiplication on the group side by negative powers of $|z|$ and $|u|$, i.e., `pure' first and second layer `weights'. While the resulting formulas remain substantially more complicated than those dealt with by Gorges and M\"uller in the Heisenberg group case, we nevertheless manage to estimate them and deduce \eqref{eq:introest1} with $a=2/3$, as well as \eqref{eq:introest2} with $b=1$. In particular, the formulas for the Schwartz kernels of the fractional powers corresponding to the second-layer weight $|u|$ involve Jacobi polynomials, and our results are ultimately based on the combination of a number of classical and more recent estimates on Jacobi polynomials \cite{Dunster,Haagerup,Krasikov,KKT}.

It is a natural question whether the stronger estimate \eqref{eq:introest1} with $a=1$ can be proved for general H-type groups; this would imply the almost everywhere convergence result in wider range \eqref{eq:gmcond}. There is some evidence that this may actually be possible: indeed, we can prove \eqref{eq:introest1} with $a=1$ in a restricted range of $j$, namely for $j \leq 3J_\delta/4$, and also for $j=J_\delta$. We remark that the case $j \leq 3J_\delta/4$ is dealt with by using pure second-layer weights, while the case $j=J_\delta$ follows by considering pure first-layer weights; this suggests that the missing range $3J_\delta/4 < j < J_\delta$ could perhaps be recovered by exploiting `mixed' weights jointly depending on $z$ and $u$.

Another related question is whether the estimates and machinery developed in this paper can be used to prove a `localisation principle' for Bochner--Riesz means on H-type groups, in the spirit of Carbery and Soria's results in the Euclidean setting \cite{CarberySoria}; recent investigation in this direction in the context of Heisenberg groups can be found in \cite{GargJotsaroop}.

A further question is whether the almost everywhere convergence result from Theorem \ref{thm:mainthm} can be `upgraded' to an $L^p$ boundedness result for the maximal operator $T_*^\lambda$, going beyond the range given by Theorem \ref{thm:mainMM}.
As a matter of fact, in analogy with the Euclidean case \cite{Christ_BochnerRiesz}, it is possible \cite{CLSY} to deduce in great generality $L^p$ boundedness results for the maximal Bochner--Riesz operators associated with a `Laplace-like' operator $\opL$ from the validity of $L^q \to L^2$ restriction estimates of Tomas--Stein type for $\opL$, provided $2 \leq p < q'$. In the case of the Heisenberg groups, however, no nontrivial restriction estimates of this kind hold for the sub-Laplacian $\opL$ \cite{M_restriction}; for H-type groups with higher-dimensional centre, some estimates of Tomas--Stein type do hold \cite{CC_restriction}, but the corresponding $L^p$ boundedness results for $T^\lambda_\maxf$ given by \cite{CLSY} are strictly included in those given by Theorem \ref{thm:mainMM}. This seems to indicate once more that the investigation of Bochner--Riesz means for sub-Laplacians requires substantially new ideas and methods compared to the Euclidean case.

In these respects, it is worth pointing out that, in the Euclidean case, relatively explicit formulas and asymptotics for the convolution kernels of the Bochner--Riesz operators are available, from which one can derive the necessity of the condition \eqref{eq:brconj} for $L^p$ boundedness (see, e.g., \cite{herz} and \cite[\S XI.6.19]{15}) and confirm the sharpness of the result of Carbery, Rubio de Francia and Vega on almost everywhere convergence (cf.\ \cite[pp.\ 320-321]{CarberySoria}). However, already in the case of the Heisenberg groups, very little appears to be known in terms of necessary conditions for almost everywhere convergence or $L^p$ boundedness of Bochner--Riesz means for sub-Laplacians. The techniques introduced in the recent work \cite{MMG_lowerbound} allow one to relate the functional calculus for a sub-Laplacian on a manifold with that for the Laplacian on a Euclidean space of the same topological dimension $D$, and can be used to deduce that, at least for what concerns $L^p$ boundedness, for a sub-Laplacian one cannot go beyond the range \eqref{eq:brconj} with $d=D$. In light of this, one may also ask whether the quantities $Q$ and $Q_*$ in \eqref{eq:gmcond} and \eqref{eq:mainthm_range} can be replaced by $D$. However at this stage we do not know whether these or other improvements are possible, or instead, as in the case of restriction estimates, ``non-Euclidean'' obstructions may subsist.

\subsection*{Structure of the paper}
In Section \ref{s:anaongroups}, we recall basic definitions and results about stratified groups, H-type groups and sub-Laplacians thereon.
 Among other things, we introduce a number of weights we will be working with and see how they interact with convolution (so-called Leibniz rules) and the group Fourier transform.

Section \ref{s:basicresult} shows how the estimates \eqref{eq:general_maximal_ests} can be proved for an arbitrary sub-Laplacian $\opL$ on a stratified group $G$, leading to the proof of Theorem \ref{thm:mainMM}.

In Sections \ref{s:chapsqfunarg} and \ref{s:chapsqfunarg_two}, we restrict to the case of Heisenberg-type groups and we discuss the aforementioned `maximal-to-nonmaximal' reduction, showing in particular that \eqref{eq:weightedmaxopest} essentially reduces to \eqref{eq:weightedopest}.

In Section \ref{s:chapredtotrace}, we 
show how the weighted $L^2$ estimates \eqref{eq:weightedopest} follow from `dual trace lemmas' as discussed above. The trace lemmas are finally proved in Section \ref{s:chaptrace}, thus completing the proof of Theorem \ref{thm:mainthm}. The proof the trace lemmas are based on a number of estimates for Jacobi polynomials that are discussed in Section \ref{s:estforjacobipoly}.

\subsection*{Notation}

We write $\NN_0$ and $\NN$ for the sets of nonnegative and positive integers, respectively; $\RR^+$ denotes the positive half-line $(0,\infty)$. For two quantities $A$ and $B$, the expression $A \lesssim B$ indicates that there exists a constant $C>0$ such that $A \leq C B$; we also write $A \lesssim_p B$ to indicate that the implicit constant $C$ may depend on the parameter $p$. Moreover $A \simeq B$ is the conjunction of $A \lesssim B$ and $B \lesssim A$

\section{Analysis on stratified and H-type groups}\label{s:anaongroups}
\subsection{Stratified groups and sub-Riemannian structure}
We briefly recall a number of standard definitions and results. For details, we refer the reader to \cite{8,Goodman,19}.

A stratified group $G$ is a connected, simply connected nilpotent Lie group whose Lie algebra $\lie{g}$ is \emph{stratified}, i.e.,
\begin{equation}\label{eq:gradation}
\lie{g}=\bigoplus_{j=1}^{k}\lie{g}_j
\end{equation}
for certain subspaces $\lie{g}_1,\ldots,\lie{g}_k$ of $\lie{g}$, called \emph{layers}, such that
\[
[\lie{g}_a,\lie{g}_b]\subseteq\lie{g}_{a+b}
\]
for all $a,b=1,\dots,k$ (here $\lie{g}_a=\{0\}$ for $a>k$) and the first layer $\lie{g}_1$ generates $\lie{g}$ as a Lie algebra; if $\lie{g}_k \neq \{0\}$, we say that $\lie{g}$ and $G$ have \emph{step} $k$. Via the exponential map we may and shall normally identify a stratified Lie group $G$ with its Lie algebra $\lie{g}$. Group multiplication on $G$ is then given by the Baker--Campbell--Hausdorff formula,
\[
xy=x+y+\frac{1}{2}[x,y]+\dots,
\]
which due to nilpotency is a finite sum, while group inversion is simply given by
\[
x^{-1} = -x,
\]
and any Lebesgue measure on $\lie{g}$ is a (left and right) Haar measure on $G$.

The choice of a Haar measure on a stratified group $G$ allows us to define Lebesgue spaces $L^p(G)$ for $1 \leq p \leq \infty$. As it is known, $L^1(G)$ is a Banach $*$-algebra with respect to convolution and involution, given by
\[
f*g(x) \defeq \int_G f(y) \, g(y^{-1} x) \,dy, \qquad f^*(x) \defeq \overline{f(x^{-1})}
\]
for almost all $x \in G$ and $f,g \in L^1(G)$. We record here the useful identity
\begin{equation}\label{eq:conv_adjoint}
\langle f,g * h\rangle = \langle f * h^*, g\rangle,
\end{equation}
where $\langle \cdot, \cdot \rangle$ denotes the $L^2$ inner product, i.e.,
\[
\langle f,g \rangle = \int_G f(x) \, \overline{g(x)} \,dx.
\]
We will also consider weighted $L^p$ spaces on $G$; for a locally integrable nonnegative function $w : G \to \RR$, we will normally write $L^p(w)$ in place of $L^p(G,w(x)\,dx)$.

If we write $x \in G \cong \lie{g}$ as $(x_1,\dots,x_k)$ according to the decomposition \eqref{eq:gradation}, automorphic dilations $\delta_r$ ($r \in \RR^+$) on $\lie{g}$ and $G$ are defined by setting
\begin{equation}\label{eq:aut_dil}
\delta_r(x_1,\dots,x_k) = (r^1 x_1,\dots,r^k x_k).
\end{equation}
With respect to these dilations, the Haar measure scales according to the dimensional parameter $\hdim$ given by
\[
\hdim \defeq \sum_{j=1}^k j \dim(\lie{g}_j),
\]
called the \emph{homogeneous dimension} of $G$. We also define $\tdim$ to be the \emph{topological dimension} of $G$ given by
\[
\tdim \defeq \sum_{j=1}^k \dim(\lie{g}_j).
\]

Since the first layer $\lie{g}_1$ generates $\lie{g}$ as a Lie algebra, the choice of an inner product on $\lie{g}_1$ determines a left-invariant sub-Riemannian structure on $G$ and a corresponding Carnot--Carath\'eodory distance $d$. By left-invariance we have actually
\[
d(x,y) = |y^{-1} x|
\]
for a nonnegative proper continuous function $|\cdot| : G \to \RR$, which is but one example of a (subadditive) \emph{homogeneous norm} on $G$, since it satisfies
\[
|xy| \leq |x|+|y|, \qquad |\delta_r x| = r|x|,
\]
for all $x,y \in G$ and $r \in \RR^+$, and in particular
\[
|x| \simeq \sum_{j=1}^k |x_j|^{1/j}.
\]
In what follows we will write $B(x,r)$ and $\overline{B}(x,r)$ to denote the open and closed balls associated with the Carnot--Carath\'eodory distance.

Due to left-invariance and homogeneity,
\[
|B(x,r)| = r^Q |B(0,1)|
\]
for all $r \in \RR^+$ and $x \in G$, where $|B(x,r)|$ denotes the Haar measure of $B(x,r)$. In particular, $G$ with the Carnot--Carath\'eodory distance $d$ and the Haar measure is a doubling metric measure space, with `doubling dimension' $\hdim$, and the theory of singular integrals and weights on spaces of homogeneous type can be applied to $G$. In particular, the
Hardy--Littlewood maximal operator $M$ on $G$, given by 
\[
Mf(x) \defeq \sup_{r>0} r^{-\hdim} \int_{|x| \leq r}|f(xy)| \,dy,
\]
is of weak type $(1,1)$ and bounded on $L^p(G)$ for $p \in (1,\infty]$. This implies the following boundedness result for maximal operators \cite[Corollary 2.5]{8}, where, for any function $f : G \to \CC$ and $r>0$, we denote by $\fDil_r f$ the function given by
\begin{equation}\label{eq:dilationdef}
\fDil_r f(x)= r^{-\hdim/2} f(\delta_{r^{-1/2}}(x)).
\end{equation}

\begin{lem}\label{lem:stein}
Let $G$ be a stratified Lie group with a sub-Riemannian structure, and $Q$ be its homogeneous dimension. Let $K : G \to \CC$ be a measurable function satisfying the estimate
\[
|K(x)|\leq\frac{C}{(1+|x|)^{\hdim+\epsilon}}
\]
for some $C,\epsilon > 0$. Let $T^\maxf$ be the operator defined by
\[
T^\maxf f(x) \defeq \sup_{r>0}|f * (\fDil_r K)(x)|,
\]
Then, for all $x \in G$,
\[
T^\maxf f(x) \lesssim_\epsilon C Mf(x).
\]
In particular,
\[
\|T^\maxf f \|_p \lesssim_{\epsilon,p} C \|f\|_p
\]
for all $p \in (1,\infty]$.
\end{lem}

Recall that a weight on $G$ is a nonnegative locally integrable function $w : G \to \RR$. The Muckenhoupt class $A_2(G)$ is the set of weights on $G$ for which the Hardy--Littlewood maximal function on $G$ is bounded on $L^2(w)$; an equivalent characterisation is that $w\in A_2(G)$ if and only if
\begin{equation}\label{eq:a2weight}
   \sup_{\substack{x\in G \\ r>0}} r^{-2\hdim} \int_{B(x,r)} w(y) \,dy \int_{B(x,r)} w(y)^{-1} \,dy < \infty
\end{equation}
\cite{15,ST}. Then we have the following result
(cf.\ \cite[Chapter V]{15}).

\begin{lem}\label{lem:A2weights}
Let $G$ be any stratified group and let $|\cdot|$ be a homogeneous norm on $G$. Then the weights $|\cdot|^a$ and $(1+|\cdot|)^{a}$ are in $A_2(G)$ for $|a|<\hdim$. In addition, if $\fstW : G \to \RR$ is defined by $\fstW(x_1,\dots,x_k) = |x_1|$ for any norm on $\lie{g}_1$, then the weights $\fstW^a$ and $(1+\fstW)^a$ are in $A_2(G)$ for $|a|<\dim \lie{g}_1$.
\end{lem}

\subsection{Sub-Laplacians and their functional calculus}

Let $G$ be a stratified group with a sub-Riemannian structure as before.
Recall that the Lie algebra of $G$ may also be thought of as the space of left-invariant vector fields on $G$. If we take an orthonormal basis $X_1,\ldots,X_d$ of $\lie{g}_1$, then we define the sub-Laplacian $\opL$ on $G$ as
\[
\opL \defeq -\sum_{j=1}^d X_j^2.
\]
It can be shown that $\opL$ does not depend on the choice of the orthonormal basis.

We may also consider the sub-Laplacian $\opL$ via its spectral decomposition. One can show that $\opL$ is positive and (essentially) self-adjoint on $L^2(G)$, with core the Schwartz class $\Sch(G)$ of rapidly decaying functions on $G$. Hence, $\opL$ has a spectral decomposition 
\begin{equation}\label{eq:sp_dec}
\opL=\int_0^\infty \lambda \,dE(\lambda).
\end{equation}
We can then define a functional calculus for $\opL$ by defining operators 
\[
F(\opL) \defeq \int_0^\infty F(\lambda) \,dE(\lambda)
\]
for all bounded Borel functions $F : \RR \to \CC$. Such operators $F(\opL)$ are left-invariant and so are convolution operators; that is, there exists $K \in\Sch'(G)$ such that $F(\opL)f = f * K$. By homogeneity
then we have the following result \cite[Lemma 6.29]{8}.

\begin{lem}\label{lem:kerneldilation}
Let $G$ be a stratified Lie group and $\opL$ be a sub-Laplacian. Let $F : \RR \to \CC$ be a bounded Borel function and let $K$ denote the convolution kernel of $F(\opL)$. Then, for all $r>0$,
\begin{equation}\label{eq:kerneldilationeq}
 F(r\opL) f = f * (\fDil_r K) = \fDil_r F(\opL) \fDil_{r^{-1}} f.
\end{equation}
\end{lem}

Here we briefly recall a number of results concerning the functional calculus of sub-Laplacians $\opL$ on stratified groups.
A property of the sub-Laplacian $\opL$ which we will use is the `finite propagation speed' of solutions of the associated wave equation (see, e.g., \cite{14,20}).

\begin{lem}\label{lem:finitepropspeed}
Let $G$ be a stratified group and $\opL$ be a sub-Laplacian. For $t\in\RR$, let $K_t$ denote the convolution kernel of the operator $\cos(t\sqrt{\opL})$. Then 
\[
\supp(K_t)\subseteq \overline{B}(0,|t|).
\]
\end{lem}

Another important property is that, if $F : \RR \to \CC$ is in the Schwartz class, then the convolution kernel of the operator $F(\opL)$ is in the Schwartz class on $G$ \cite{17}. A particular instance of this result is stated below in a quantitative form.

\begin{lem}\label{lem:kernelestimate}
Let $\opL$ be a sub-Laplacian on a stratified group $G$. Then there exists $k \in \NN$ such that, if an even function $F : \RR \to \RR$ satisfies
\[
\| F \|^*_{k} \defeq \sup_{\substack{\lambda\in\RR^+\\j=0,\ldots,k}} (1+\lambda)^k |F^{(j)}(\lambda)| < \infty
\]
then the convolution kernel $K$ of the operator $F(\sqrt{\opL})$ satisfies the estimate
\begin{equation}\label{eq:decayestimate}
|K(x)|\lesssim \frac{\|F\|^*_{k}}{(1+|x|)^{\hdim+1}}
\end{equation} 
for all $x \in G$, where the implicit constant does not depend on $F$.
\end{lem}
\begin{proof}
This is an immediate consequence of \cite[Lemmas 1.2 and 2.4]{17} and \cite{18}.
\end{proof}

Note that the estimate \eqref{eq:decayestimate} implies integrability of $K$. A number of works have been devoted to determining the minimal smoothness requirement on a compactly supported $F : \RR \to \CC$ so that the convolution kernel of the operator $F(\opL)$ is integrable (see \cite{MM_gafa,MMG_lowerbound} and references therein). We recall here the definition of $\thr(\opL)$ from \cite{MM_gafa} as the infimum of all the $s_0 \in \RR^+$ such that, for all $s>s_0$ and all $F : \RR \to \CC$ supported in $[-1,1]$,
\begin{equation}\label{eq:sharpthreshold}
\|K\|_{L^1(G)} \lesssim_s \|F\|_{L^2_s(\RR)}
\end{equation}
where $K$ is the convolution kernel of $F(\opL)$, and $L^2_s(\RR)$ is the $L^2$ Sobolev space on $\RR$ of (fractional) order $s$. As mentioned in the introduction, $D/2 \leq \thr(\opL) \leq Q/2$ for all stratified groups $G$ and sub-Laplacians $\opL$, and the equality $\thr(\opL)=D/2$ is known to hold for a number of $2$-step stratified groups, including the H-type groups.

The next lemma regards weighted $L^2$ boundedness of a square function associated to a Littlewood--Paley decomposition for a sub-Laplacian.
The result is analogous to Euclidean results found in, for example, \cite{15}; a proof in our setting can be derived, e.g., from the results of \cite{KW_paleylittlewood} and \cite{stempak}.

\begin{lem}\label{lem:weightslem3} 
Let $\opL$ be a sub-Laplacian on a stratified group $G$. Let $\varphi \in C_c^\infty(\RR^+)$ be such that 
\[
\sum_{l\in\ZZ} \varphi(2^{-l} \lambda) = 1 \quad \text{ for all } \lambda>0,
\]
and let $\omega\in A_2(G)$. Then
\begin{equation}\label{eq:weights1}
    \sum_{l\in\ZZ} \|\varphi(2^{-l} \opL)f\|^2_{L^2(\omega)}\simeq\|f\|^2_{L^2(\omega)}
\end{equation}
for all $f \in L^2(\omega)$, where the implicit constants may depend on $\varphi$ and $\omega$.
\end{lem}

\subsection{H-type groups}

An H-type group is a $2$-step stratified Lie group whose Lie algebra $\lie{g}$ is endowed with an inner product $\langle \cdot,\cdot \rangle$ satisfying the following conditions. First, 
the layers $\lie{g}_1$ and $\lie{g}_2$ are orthogonal. Secondly, if we define, for each $\mu\in\lie{g}_2^*$, the skew-symmetric endomorphism $\skwJ_\mu$ of $\lie{g}_1$ by 
\[
\langle \skwJ_\mu(z),z'\rangle=\mu([z,z']) \qquad \forall z,z' \in \lie{g}_1,
\]
then we require that, for all $\mu\in\lie{g}_2^*$,
\[
\skwJ_\mu^2 = -|\mu|^2 \id.
\]
Note that, under these assumptions, $\mu([\cdot,\cdot])$ is a symplectic form on $\lie{g}_1$ for all $\mu \in \lie{g}_2^* \setminus \{0\}$, hence the dimension of $\lie{g}_1$ is even. Moreover, the restriction of the inner product to $\lie{g}_1$ determines a sub-Riemannian structure on $G$ and a distinguished sub-Laplacian $\opL$, which we will use throughout.

We refer to \cite{CDKR,DR,Kaplan,KaplanRicci} for additional information on H-type groups.

\subsubsection{The Fourier transform on H-type groups}\label{s:fouriertheory}

We now recall some facts regarding Fourier analysis on H-type groups.
Let $G$ be an H-type group with $\dim \lie{g}_1 = 2m$ and $\dim \lie{g}_2 = n$. Following \cite{ACDS}, for each
$\mu \in \lie{g}_2^* \setminus \{0\} \simeq \RR^n \setminus\{0\}$ there exists an orthonormal basis $E_1(\mu),\dots,E_{2m}(\mu)$ 
of $\lie{g}_1$ such that 
\[
\skwJ_\mu E_j(\mu) = \begin{cases}
|\mu| E_{m+j}(\mu) &\text{if $j \leq m$,} \\
-|\mu|E_{m-j}(\mu) &\text{otherwise.}
\end{cases}
\]
Here we do not assume that the $E_j(\mu)$ depend continuously on $\mu$, however we may and shall assume that $E_j(\lambda\mu) = E_j(\mu)$ for all $\lambda>0$.
For all $\mu \in \lie{g}_2^* \setminus \{0\}$, this choice of an orthonormal basis induces an isometric identification of $\lie{g}_1$ with $\CC^m$: to all $z \in \lie{g}_1$ we associate the element $z(\mu) = (z_1(\mu),\dots,z_m(\mu)) \in \CC^m$ such that
\[
z = \sum_{j=1}^m [ (\Re z_j(\mu)) E_j(\mu) + (\Im z_j(\mu)) E_{j+m}(\mu)].
\]
It is easily checked that, for all $\mu \in \lie{g}_2^*$, the map $(z,u) \mapsto (z(\mu), |\mu|^{-1} \mu \cdot u)$ defines a Lie group epimorphism from $G$ to the Heisenberg group $H_m$.
In particular, if we define  (cf.\ \cite{Folland_PhaseSpace,GM})
the Schr\"odinger representation $\varpi_s^0$ of $H_m$ with parameter $s \in \RR \setminus \{0\}$ on $L^2(\RR^m)$ by
\[
[\varpi^0_s(z,t)\varphi](x)
= e^{2\pi i s (t + \Im(z) \cdot x + \Re(z) \cdot \Im(z)/2)} \varphi(x+\Re(z))
\]
for all $\varphi \in L^2(\RR^m)$ and $(z,t) \in H_m$,
then a family $\varpi_\mu$ ($\mu \in \lie{g}_2^* \setminus \{0\}$) of pairwise inequivalent irreducible unitary representations of $G$ on $L^2(\RR^m)$ is given by
\begin{equation}\label{eq:schroedinger_quotient}
\varpi_\mu(z,u) = \pi^0_{|\mu|}(z(\mu),|\mu|^{-1} \mu \cdot u)
\end{equation}
for all $(z,u) \in G$.

This family of representations is enough to write a Plancherel formula for the group Fourier transform. Namely, if we define the group Fourier transform of $f \in L^1(G)$ as the operator-valued function given by 
\[
\widehat{f}(\mu) \defeq \int_G f(g) \,\varpi_\mu(g) \,dg
\]
for all $\mu \in \lie{g}_2^* \setminus \{0\}$, then the following properties hold (see, for instance, \cite{ACDS,Folland_Abstract}),
where $T^\adj$ denotes the adjoint operator to $T$.

\begin{lem}
For all $f,g \in L^1(G)$ and $\mu \in \lie{g}_2^* \setminus \{0\}$,
\begin{equation}\label{eq:fourierconv1}
   \widehat{f * g}(\mu)=\widehat{f}(\mu) \,\widehat{g}(\mu),
\end{equation}
\begin{equation}\label{eq:fourierinv}
    \widehat{f^*}(\mu)=\widehat{f}(\mu)^\adj.
\end{equation}
Moreover,
for all $f,g \in L^1 \cap L^2(G)$,
\begin{align}
\label{eq:plancherel2}
  \langle f,g\rangle 
	&= 
	\int_{\RR^n}\tr(\widehat{f}(\mu) \, \widehat{g}(\mu)^\adj) \,|\mu|^m \,d\mu , \\
\label{eq:plancherel}
\|f\|^2_2 &= \int_{\RR^n}\|\widehat{f}(\mu)\|^2_{\HS} \,|\mu|^m \,d\mu.
\end{align}
\end{lem}

Note that from \eqref{eq:schroedinger_quotient} it follows that, for all $f \in L^1(G)$ and $\mu \in \lie{g}_2^* \setminus \{0\}$,
\begin{equation}\label{eq:gft_quotient}
\widehat{f}(\mu) = \int_{H_m} P_\mu f(g) \, \varpi_{|\mu|}^0 (g) \,dg,
\end{equation}
where $P_\mu f \in L^1(H_m)$ is defined by
\[
P_\mu f(z(\mu),t) = \int_{\mu^\perp} f(z,t+v) \,dv
\]
for all $z \in \lie{g}_1$ and $t \in \RR$.
In other words, the group Fourier transform of $f \in L^1(G)$ at $\mu \in \lie{g}_2^* \setminus \{0\}$ is the same as the group Fourier transform of $P_\mu f \in L^1(H_m)$ at $|\mu|$.

It is convenient to express the `matrix components' of the group Fourier transform $\widehat f(\mu)$ of a function $f \in L^1(G)$ in terms of 
suitably rescaled Hermite functions.
We start by defining Hermite functions on the real line by 
\[
h_k(x) \defeq (2^kk!\sqrt{\pi})^{-1/2}
(-1)^ke^{x^2/2}\frac{d^k}{dx^k}e^{-x^2},\quad x\in\RR, \,k\in\NN_0,
\]
and their $m$-dimensional versions as
\[
h_\alpha(x) \defeq \prod_{j=1}^m h_{\alpha_j}(x_j),\quad x\in\RR^m, \,\alpha\in\NN_0^m.
\]
We then renormalise these Hermite functions by defining, for all $s>0$,
\begin{equation}
    h_\alpha^s(x) \defeq (2\pi s)^{m/4} h_\alpha((2\pi s)^{1/2}x), \quad x\in\RR^m, \,\alpha\in\NN_0^m.
\end{equation} 
For each $s>0$, the family $(h_\alpha^s)_{\alpha\in\NN_0^m}$ forms an orthonormal basis of $L^2(\RR^m)$.
Now, for all $f\in L^1(G)$, $\mu\in\lie{g}_2^*\setminus\{0\}$ and $\alpha,\beta\in \NN_0^m$, 
we define 
\begin{equation}\label{eq:matrixcomponentdef}
   \widehat{f}(\mu,\alpha,\beta) \defeq \langle \widehat f(\mu) h_\alpha^{|\mu|}, h_\beta^{|\mu|}\rangle = \int_{G} f(g) \, \langle \pi_\mu(g) h_\alpha^{|\mu|}, h_\beta^{|\mu|}\rangle \,dg.
\end{equation}
For later convenience we extend the definition of $\widehat{f}(\mu,\alpha,\beta)$ to all $\alpha,\beta \in \ZZ^m$ by
\[
\widehat{f}(\mu,\alpha,\beta) \defeq 0 \quad \text{ for all } (\alpha,\beta) \notin \NN_0^m \times \NN_0^m.
\]
From \eqref{eq:fourierconv1} and \eqref{eq:fourierinv} we immediately derive the following identities:
\begin{align}
\label{eq:matrixstar}
   \widehat{f^*}(\mu,\alpha,\beta)
	&=\overline{\widehat{f}(\mu,\beta,\alpha)}, \\
\label{eq:matrixcompconv}
  \widehat{f * g}(\mu,\alpha,\beta) 
	&= \sum_{\gamma\in\NN_0^m}  \widehat{g}(\mu,\alpha,\gamma) \widehat{f}(\mu,\gamma,\beta).
\end{align}
for all $f,g\in L^1(G)$, $\mu \in \lie{g}_2^* \setminus \{0\}$ and $\alpha,\beta \in \NN_0^m$.

One can show that the Hermite functions $h_\alpha^{|\mu|}$ are eigenfunctions for $\widehat{\opL}(\mu) \defeq d\varpi_\mu(\opL)$, the group Fourier transform of the sub-Laplacian; namely,
\[
\widehat{\opL}(\mu) h_\alpha^{|\mu|} = c(|\alpha|) |\mu| h_\alpha^{|\mu|},
\]
where $|\alpha| = \alpha_1+\dots+\alpha_m$ for all $\alpha \in \NN^m$ and, for all $k \in \NN_0$,
\begin{equation}\label{eq:eig}
  c(k) \defeq 2\pi(2k+m).
\end{equation}
In addition, if $\ctDer$ is the `central pseudodifferential operator' defined in \eqref{eq:difflambda}, then $\widehat\ctDer(\mu) \defeq d\varpi_\mu(\ctDer) = 2\pi|\mu| \id$.
The group Fourier transform is compatible with the joint spectral decomposition and functional calculus of $\opL$ and $\ctDer$, and so
\begin{equation}\label{eq:fourierL}
  (F(\opL,\ctDer)f)\widehat{\;}(\mu,\alpha,\beta)
	= F(c(|\alpha|)|\mu|,2\pi|\mu|) \, \widehat{f}(\mu,\alpha,\beta).
\end{equation}

It will often be convenient to consider functions on $G$ that depend only on $|z|$ and $u$; we shall call such functions \emph{radial}.
In the 
case that $f$ is radial, 
the off-diagonal matrix coefficients of $\widehat f(\mu)$, are zero, and furthermore the diagonal coefficients depend only on the magnitude $|\alpha| = \sum_j \alpha_j$ of the index $\alpha \in \NN_0^m$:
\[
      \widehat f(\mu,\alpha,\beta)=\delta_{\alpha,\beta} \widehat f(\mu,|\alpha|e_1,|\alpha|e_1)
\]
for all $\mu \in \lie{g}_2^* \setminus \{0\}$ and $\alpha,\beta \in \NN_0^m$, where $e_1 = (1,0,\dots,0)$ (this is true for Heisenberg groups \cite[Theorem 1.4.3]{ThHeis}, hence for general H-type groups by \eqref{eq:gft_quotient}).
In this case, we adopt the notation 
\begin{equation}\label{eq:gelfand}
\widehat f(\mu,k) \defeq \widehat f(\mu,ke_1,ke_1)
\end{equation}
for all $(\mu,k) \in (\lie{g}_2^* \setminus \{0\}) \times\NN_0$.
These simplifications correspond to the fact that the Banach $*$-algebra $L^1_\rad(G)$ of integrable radial functions on $G$ is commutative \cite{DR}; indeed \eqref{eq:gelfand} expresses a relation between the Gelfand transform and the group Fourier transform of $f \in L^1_\rad(G)$, where $(\lie{g}_2^* \setminus \{0\}) \times\NN_0$ parametrises a subset of full measure of the Gelfand spectrum of $L^1_\rad(G)$.

For radial functions we have simpler expressions for the Fourier coefficients and the Plancherel formula.
We recall that the Laguerre polynomial $L^a_k$ of type $a>-1$ and degree $k$ is defined by
\[
L_k^a(x) \defeq \frac{1}{k!}e^xx^{-a}\frac{d^k}{dx^k}(e^{-x}x^{k+a}),\quad x\in\RR.
\]
Then, for all radial $f\in\Sch(G)$,
\begin{equation}\label{eq:ftol}
     \widehat{f}(\mu,k) = \binom{k+m-1}{k}^{-1} \int_G e^{2\pi i\mu \cdot u}f(z,u) \, L_k^{m-1}(\pi|\mu||z|^2) \, e^{-\frac{\pi|\mu||z|^2}{2}} \,dz \,du.
 \end{equation}
This now gives us an alternative Plancherel theorem and inversion formula for radial functions, which may also be found in \cite{2}. Specifically, if $f\in L^2(G)$ is radial, then 
\begin{equation}\label{eq:radialplancherel}
    \|f\|^2_2 = \int_{\RR^n\setminus\{0\}} \sum_{k\in\NN_0} \binom{k+m-1}{k} \, |\widehat{f}(\mu,k)|^2 \,|\mu|^m \,d\mu
\end{equation}
and, if $f\in\Sch(G)$ is radial, then
\begin{equation}\label{eq:ftolinv}
    f(z,u)= \int_{\RR^n\setminus\{0\}}\sum_{k\in\NN_0} \widehat{f}(\mu,k) \, e^{-2\pi i\mu \cdot u} \, L_k^{m-1}(\pi|\mu||z|^2) \,e^{-\frac{\pi|\mu||z|^2}{2}}|\mu|^m \,d\mu.
\end{equation}

We recall from \eqref{eq:plancherel2} that
\begin{equation}\label{eq:convl2es0}
\begin{split}
\langle f,g\rangle 
& =\int_{\RR^n}\sum_{\alpha,\beta\in\NN_0^m}\widehat{f}(\mu,\beta,\alpha) \, \overline{\widehat{g}(\mu,\beta,\alpha)} \,|\mu|^m \,d\mu.
\end{split}
\end{equation}
Observe that, if one of $f,g$ is a radial function, then the only non-zero terms would be the diagonal ones, where $\alpha=\beta$. Thus, if one of $f,g$ is radial, then
\begin{equation}\label{eq:convl2es1}
\langle f, g\rangle 
= \int_{\RR^n} \sum_{\alpha\in\NN_0^m} \widehat{f}(\mu,\alpha,\alpha) \, \overline{\widehat{g}(\mu,\alpha,\alpha)} \,|\mu|^m \,d\mu.
\end{equation}
Furthermore, from \eqref{eq:matrixcompconv} and the Cauchy--Schwarz inequality, we deduce that
\begin{equation}
\label{eq:convl2es}
    \int_{\RR^n}\sum_{\alpha\in\NN_0^m}|\widehat{f * g}(\mu,\alpha,\alpha)| \,|\mu|^m \,d\mu
		\leq \|f\|_2\|g\|_2.
\end{equation}

\subsubsection{Differentiation on the Fourier dual}\label{s:diffdual}

A complex valued function on an H-type group $G$ will be called a \emph{polynomial} if it is a polynomial in exponential coordinates. In the Euclidean case, the Fourier transform intertwines operators of multiplication by a polynomial with constant coefficient differential operators. In analogy with this, it is natural to interpret the effect on the Fourier side of multiplication by a polynomial on $G$ as a sort of `differentiation' on the group Fourier dual.

This idea makes sense also in more general stratified groups (see, e.g., \cite{58}). However, on H-type groups, explicit formulas for these `differential operators' on the Fourier side can be written in many cases. In the case of the Heisenberg groups, a number of these formulas are listed in \cite[p.\ 151]{GM} (see also \cite{DeMicheleMauceri,MS_heisenberg} and \cite[Lemma 6.4]{Ricci_dispense}). In view of \eqref{eq:gft_quotient}, these formulas admit straightforward extensions to H-type groups, which we list below. We need some notation: for all $\mu \in \lie{g}_2^* \setminus \{0\}$, $j=1,\dots,m$ and $l=1,\dots,n$, define
\begin{equation}\label{eq:weightdef}
\begin{gathered}
\fstZ_{\mu,j}(z,u)=z_j(\mu), \qquad \overline{\fstZ_{\mu,j}}(z,u) =\overline{z_j(\mu)}, \qquad \fstW(z,u)=|z|,\\
	\sndW_l(z,u)=u_l, \qquad \sndW(z,u)=|u|.
\end{gathered}
\end{equation}
Here we are identifying $\lie{g}_2$ with $\RR^n$ by the choice of an orthonormal basis, so that the $u_l$ are the components of $u$.
Note that $\fstW$ and $\sndW$ are not polynomials, but their squares are.

Let us first consider \emph{first-layer} polynomials, i.e., those depending only on the first-layer variable $z$.
For all $f\in\Sch(G)$, $\mu \in \lie{g}_2^*\setminus\{0\}$, $j=1,\dots,m$, $\alpha,\beta \in \NN_0^m$,
\begin{gather}
\label{eq:partzeta1} (\pi|\mu|)^{1/2} \widehat{\fstZ_{\mu,j} f}(\mu,\alpha,\beta) = ({\alpha_j+1})^{1/2} \widehat f(\mu,\alpha+e_j,\beta) - \beta_j^{1/2} \widehat f(\mu,\alpha,\beta-e_j) \\
\label{eq:partzeta2} (\pi|\mu|)^{1/2} \widehat{\overline{\fstZ_{\mu,j}} f}(\mu,\alpha,\beta) = \alpha_j^{1/2} \widehat f(\mu,\alpha-e_j,\beta) - (\beta_j+1)^{1/2}\widehat f(\mu,\alpha,\beta+e_j),
\end{gather}
where $e_j \in \NN_0^m$ is the $j$th standard basis element.
By combining these operators and summing over $j$ we thus obtain a formula for the operator of multiplication by $\fstW^2$,
which is particularly simple in the case of radial functions $f$. Specifically, for all $\mu \in \lie{g}_2^* \setminus \{0\}$ and $k \in \NN_0$,
\begin{equation}\label{eq:partzsqrule}
    \widehat{\fstW^2 f}(\mu,k) = \frac{1}{\pi|\mu|}[(2k+m) \widehat{f}(\mu,k) - k \widehat{f}(\mu,k-1) - (k+m) \widehat{f}(\mu,k+1)].
\end{equation}

We now pass to \emph{second-layer} polynomials, i.e., those only depending on $u$. For all radial functions $f \in \Sch(G)$, $l=1,\dots,n$, $\mu \in \lie{g}_2^* \setminus \{0\}$ and $k \in \NN_0$,
\begin{multline}\label{eq:parturule}
   4\pi i \widehat{\sndW_l f}(\mu,k) \\
	= 2\frac{\partial}{\partial{\mu_l}} \widehat f(\mu,k) + \frac{\mu_l}{|\mu|^2}[m \widehat f(\mu,k) + k \widehat f(\mu,k-1) - (k+m) \widehat f(\mu,k+1)].
\end{multline}

Note that, in the formulas \eqref{eq:partzeta1} and \eqref{eq:partzeta2}, the 
$\mu$ in the multiplier $\fstZ_{\mu,j}$ must match the $\mu$ in the argument of the Fourier transform $\widehat{f}$.
In applications we will also need to consider the case of mismatch. This is discussed in the following lemma, where we use the notation
\begin{equation}\label{eq:conj_notation}
    \fstZ_{\mu,p,0} \defeq \fstZ_{\mu,p} \quad\text{and}\quad \fstZ_{\mu,p,1} \defeq \overline{\fstZ_{\mu,p}}.
\end{equation}

\begin{lem}\label{lem:matrixcomponentcoordchange}
Let $\mu_1,\mu\in\RR^n\setminus\{0\}$.
Then there exist $C_{\alpha,\beta,j,k}(\mu,\mu_1) \in \CC$ (where $j,k\in\{1,\dots,m\}$ and $\alpha,\beta\in\{0,1\}$) such that 
\begin{gather}\label{eq:mixedmu}
\fstZ_{\mu_1,j,\alpha} = \sum_{k=1}^m \sum_{\beta\in\{0,1\}} C_{\alpha,\beta,j,k}(\mu,\mu_1) \, \fstZ_{\mu,k,\beta},
\end{gather}
and $|C_{\alpha,\beta,j,k}(\mu,\mu_1)|$ is bounded uniformly in $\alpha,\beta,j,k,\mu,\mu_1$.
\end{lem}
\begin{proof}
For all $\mu,\mu_1 \in \lie{g}_2^* \setminus \{0\}$, the change of variables $z(\mu) \mapsto z(\mu_1)$ is an $\RR$-linear isometry on $\CC^m$, whose matrix coefficients are therefore uniformly bounded as well as those of its inverse.
\end{proof}

\subsubsection{Dual Leibniz rules}\label{s:secleibniz}
We now proceed to calculate `Leibniz rules' for the polynomials in \eqref{eq:weightdef}, describing the effect of multiplying by such polynomials a convolution product on $G$ (see also \cite[Proposition 5.2.10]{58}).

Note first that each $\fstZ_{\mu,j} : G \to \CC$ ($\mu \in \lie{g}_2^* \setminus \{0\}$, $j=1,\dots,m$) is a group homomorphism, whence
\[
\fstZ_{\mu,j} (f * g) 
= (\fstZ_{\mu,j} f) * g + f * (\fstZ_{\mu,j} g).
\]
for all $f,g\in L^1(G)$. An analogous rule holds for $\overline{\fstZ_{\mu,j}}$. Iterating and combining the above formula yields
\begin{equation}\label{eq:rhorule}
\fstW^2 (f * g)
= (\fstW^2 f) * g + f * (\fstW^2 g) + \sum_{j=1}^m (\fstZ_{\mu,j} f) * (\overline{\fstZ_{\mu,j}} g) + \sum_{j=1}^m (\overline{\fstZ_{\mu,j}} f) * (\fstZ_{\mu,j} g).
\end{equation}

The rule for $\sndW_l$ ($l=1,\dots,n$) is more involved, due to the fact that $\sndW_l$ is not a homomorphism. Indeed
\[
\sndW_l((z,u) \cdot (z',u'))
= \sndW_l((z,u)) + \sndW_l(z',v') + \frac{1}{2} ([z,z'])_l.
\]
By explicitly writing $([z,z'])_l$ in terms of the coordinates $z(\mu)$, $z'(\mu)$ (for any choice of $\mu \in \lie{g}_2^* \setminus \{0\}$) and the structure constants of $\lie{g}$, we easily derive
\begin{equation}\label{eq:psiruleexact}
\begin{split}
\sndW_l(f * g) 
&= (\sndW_l f) * g + f * (\sndW_l g)  \\
&+\sum_{k,j=1}^m \sum_{\alpha,\beta \in \{0,1\}} c^{(l)}_{\mu,k,j,\alpha,\beta} \, (\fstZ_{\mu,k,\alpha} f) * (\fstZ_{\mu,j,\beta} g),
\end{split}
\end{equation}
for some constants $c^{(l)}_{\mu,k,j,\alpha,\beta} \in \CC$ which are uniformly bounded in $\mu,k,j,\alpha,\beta$; here again we are using the notation \eqref{eq:conj_notation}.

\subsubsection{Dual fractional integration for radial functions}\label{s:fracint}
Recall from Section \ref{s:fouriertheory} that the Fourier transformation determines a unitary isomorphism between the space $L^2_{\rad}(G)$ of square-integrable radial functions on $G$ and the space $\mathcal{H} = L^2((\RR^n\setminus\{0\})\times \NN_0,|\mu|^m \,d\mu \, \binom{k+m-1}{k} \,d\#(k))$, where $\#$ denotes the counting measure on $\NN_0$. If $\omega = \omega(|z|,u)$ is a radial function on $G$, then this unitary isomorphism intertwines the operator of multiplication by $\omega$ with a (possibly unbounded) operator $\partial_\omega$ on $\mathcal{H}$.

If $\omega$ is a radial polynomial, then $\partial_\omega$ corresponds to one of the `differential operators on the dual' discussed in Section \ref{s:diffdual}. If instead $\omega$ is a negative fractional power of a radial polynomial, then we can think of $\partial_\omega$ as a `fractional integration operator on the dual'. The formulas below allow us to give a more explicit description of such operators $\partial_\omega$ in terms of $\omega$.

\begin{lem}\label{lem:integralkernel}
Let $\omega$ be a radial function on $G$, so that $\omega(z,u) = \omega_0(|z|^2,u)$.
Then, in the sense of distributions, we can write
$\partial_{\omega}$
as a generalised integral operator,
\begin{equation}\label{eq:integralkernel}
\partial_{\omega} H(\mu,k)=\int_{\RR^n}\sum_{l\in\NN_0} H(\nu,l)K_\omega(\nu,l;\mu,k)\binom{l+m-1}{l} \,|\nu|^m \,d\nu,
\end{equation}
with Schwartz kernel
\begin{multline}\label{eq:radialpartialweight}
K_\omega(\nu,l;\mu,k) \defeq \frac{C(m)}{\binom{k+m-1}{k} \binom{l+m-1}{l}} \int_0^\infty \Four_2\omega_0(t,\nu-\mu)\\
\times L_l^{m-1}(\pi|\nu|t) \, L_k^{m-1}(\pi|\mu|t) \,e^{-\frac{\pi(|\nu|+|\mu|)t}{2}} \,t^{m-1} \,dt.
\end{multline}
Here $C(m) = \pi^m/(m-1)!$ is half the measure of the unit sphere in $\CC^m$, and $\Four_2 \omega_0$ is the partial Euclidean Fourier transform of $\omega_0$ in the second variable, that is,
\begin{equation}\label{eq:partft}
\Four_2\omega_0(t,\mu) \defeq \int_{\RR^n} \omega_0(t,u) \, e^{-2\pi i u \cdot \mu} \,du.
\end{equation}
\end{lem}
\begin{proof}
Let $f \in \Sch(G)$ by a radial function. Then, by \eqref{eq:ftol},
\begin{multline}
\partial_{\omega}\widehat{f}(\mu,k)\\
=\binom{k+m-1}{k}^{-1} \int_G \omega_0(|z|^2,u) \, f(z,u) \, e^{2\pi i\mu \cdot u} \, L_k^{m-1}(\pi|\mu||z|^2) \, e^{-\frac{\pi|\mu||z|^2}{2}} \,dz \,du.
\end{multline} 
By \eqref{eq:ftolinv} and our identification of $G$ as $\CC^m\times \RR^n$, we then obtain that
\[\begin{split}
\partial_{\omega}\widehat{f}(\mu,k) &= \binom{k+m-1}{k}^{-1} \int_{\RR^n} \int_{\CC^m} \omega_0(|z|^2,u)  \\
	&\times \int_{\RR^n\setminus\{0\}}\sum_{l\in\NN_0} \widehat{f}(\nu,l) \, e^{-2\pi i\nu \cdot u} \, L_l^{m-1}(\pi|\nu||z|^2) \, e^{-\frac{\pi|\nu||z|^2}{2}} \, |\nu|^m \,d\nu\\
	&\times e^{2\pi i\mu \cdot u} \, L_k^{m-1}(\pi|\mu||z|^2) \, e^{-\frac{\pi|\mu||z|^2}{2}} \,dz \,du.
\end{split}\]
By \eqref{eq:partft} and using polar coordinates in $z$, this gives that
\begin{multline*}
\partial_{\omega}\widehat{f}(\mu,k) =  \frac{C(m)}{\binom{k+m-1}{k}} \int_{\RR^n} \sum_{l\in\NN_0}
	\int_0^\infty \Four_2\omega_0(t,\nu-\mu) \, \widehat{f}(\nu,l)  \\
	\times  L_l^{m-1}(\pi|\nu|t) \, L_k^{m-1}(\pi|\mu|t) \, e^{-\frac{\pi(|\nu|+|\mu|)t}{2}} \,t^{m-1} \,dt \,|\nu|^m \,d\nu,
\end{multline*}
as required.
\end{proof}

If we assume that $\omega(z,u)$ is a function of only $|z|$ or $u$, then simplifications occur in the formula for the Schwartz kernel $K_\omega$.

\begin{lem}\label{lem:radialkernelz}
With the notation of Lemma \ref{lem:integralkernel}, if $\omega_0(t,u)=w(t)$, then 
\begin{multline}\label{eq:radialkernelz}
K_\omega(\nu,l;\mu,k)\\
=\frac{C(m) \,\delta(\nu-\mu)}{\binom{k+m-1}{k}\binom{l+m-1}{l}}
\int_0^\infty w(t) \, L_l^{m-1}(\pi|\nu|t) \, L_k^{m-1}(\pi|\nu|t) \, e^{-\pi|\nu|t} \,t^{m-1} \,dt,
\end{multline}
where $\delta$ is the Dirac delta on $\RR^n$.
\end{lem}
\begin{proof}
Observe that, in this case,
$\Four_2 \omega_0(t,\nu)= w(t) \, \delta(\nu)$.
The result is then immediate from \eqref{eq:radialpartialweight}.
\end{proof}

\begin{lem}\label{lem:heatkerneluformulae}
With the notation of Lemma \ref{lem:integralkernel}, if $\omega_0(t,u)=w(u)$, then
\begin{multline}\label{eq:heatkernelu1}
K_\omega(\nu,l;\mu,k)\\
= \begin{cases}
\binom{l+m-1}{l}^{-1} \frac{\Four w(\nu-\mu)}{(|\nu|+|\mu|)^m}
\left(\tfrac{|\nu|-|\mu|}{|\mu|+|\nu|}\right)^{k-l}
P_l^{(k-l,m-1)}\left(1-2\left(\tfrac{|\mu|-|\nu|}{|\mu|+|\nu|}\right)^2\right) &\text{ if } k \geq l, \\
\binom{k+m-1}{k}^{-1} \frac{\Four w(\nu-\mu)}{(|\nu|+|\mu|)^m} 
\left(\tfrac{|\mu|-|\nu|}{|\mu|+|\nu|}\right)^{l-k}
P_k^{(l-k,m-1)}\left(1-2\left(\tfrac{|\mu|-|\nu|}{|\mu|+|\nu|}\right)^2\right) &\text{ if } k \leq l,
\end{cases}
\end{multline}
where the $P_n^{(a,b)}$ are Jacobi polynomials, and $\Four w$ is the Euclidean Fourier transform of $w$.
\end{lem}
\begin{proof}
Since $\Four_2{\omega}(t,\nu-\mu) = \Four w(\nu-\mu)$, in this case \eqref{eq:radialpartialweight} becomes
\[
K_\omega(\nu,l;\mu,k)=\frac{C(m) \, \Four w(\nu-\mu)}{\binom{k+m-1}{k}\binom{l+m-1}{l}} \int_0^\infty L_l^{m-1}(\pi|\nu|t) \,L_k^{m-1}(\pi|\mu|t) \, e^{-\frac{\pi(|\nu|+|\mu|)t}{2}} \,t^{m-1} \,dt.
\]
If we set $u=\pi t(|\nu|+|\mu|)$, then 
\begin{multline*}
K_\omega(\nu,l;\mu,k)\\
= \frac{\pi^{-m} \, C(m) \, \Four w(\nu-\mu)}{\binom{k+m-1}{k}\binom{l+m-1}{l}(|\nu|+|\mu|)^m}\int_0^\infty L_l^{m-1}(\tfrac{|\nu|}{|\nu|+|\mu|}u) \, L_k^{m-1}(\tfrac{|\mu|}{|\nu|+|\mu|}u) \,e^{-{u}/{2}} \,u^{m-1} \,du.
\end{multline*}
The result is then immediate using Lemma \ref{lem:laguerreintegrallem}\ref{en:laguerreintegrallem2} below.
\end{proof}

\section{The basic estimates}\label{s:basicresult}

In this section we prove the estimates \eqref{eq:general_maximal_ests} for an arbitrary stratified group $G$ and sub-Laplacian $\opL$, which immediately imply Theorem \ref{thm:mainMM}.

Recall \eqref{eq:multiplierdecompose} and \eqref{eq:mdecomposemax}. Let $K_\delta$ be the convolution kernel of $\BRm_\delta(\opL)$. Note that, by Lemma \ref{lem:kerneldilation},
\begin{equation}\label{eq:dilation}
   \BRm_\delta(r \opL)f = f * (\fDil_r K_\delta) .
\end{equation}
Moreover, for each $\delta\in\Dyad_0$, the operator ${M}^\maxf_{\delta}$ is bounded on $L^p$ for all $p \in (1,\infty]$ by Lemmas \ref{lem:stein} and \ref{lem:kernelestimate}, so it suffices to consider $\delta \in \Dyad = \Dyad_0 \setminus \{1\}$.

Now, from the conditions \eqref{eq:mdeltabnd} and interpolation it immediately follows that
\[
\|\BRm_\delta\|_{L^2_{s}(\RR)} \lesssim \delta^{1/2-s}
\]
for all $s \in \RR^+$ and $\delta \in \Dyad$. Combined with \eqref{eq:sharpthreshold}, this gives
\[
\|\fDil_r K_\delta\|_1 = \|K_\delta\|_1 \lessapprox \delta^{1/2-\thr(\opL)}
\]
for all $r \in \RR^+$, whence
\[
\sup_{r \in \RR^+} \|\BRm_\delta(r \opL)f\|_\infty \lessapprox \delta^{1/2-\thr(\opL)} \|f\|_\infty,
\]
which implies the first estimate in \eqref{eq:general_maximal_ests}.

Let us now recall a simple consequence of the Fundamental Theorem of Calculus (cf.\ \cite[Section 3]{GM}).

\begin{lem}\label{lem:fundthmcalc}
For all $f \in \Sch(G)$, all $\delta \in \Dyad$, and all $x \in G$,
\begin{equation}\label{eq:fundthmcalc_global}
|M_\delta^\maxf f(x)|^2 \leq 2 \delta^{-1} \int_0^\infty |\BRm_\delta(t\opL)f(x)| \, |\tBRm_\delta(t\opL)f(x)| \,\frac{dt}{t}
\end{equation}
and
\begin{equation}\label{eq:fundthmcalc_local}
|M_\delta^\lmaxf f(x)|^2 \leq  2\delta^{-1}\int_0^1|\BRm_\delta(t\opL)f(x)| \,|\tBRm_\delta(t\opL)f(x)| \,\frac{dt}{t},
\end{equation}
where
\begin{equation}\label{eq:tildem}
\tBRm_\delta(\zeta) \defeq \delta\zeta \BRm_\delta'(\zeta) \quad\text{for all }\zeta \in \RR^+.
\end{equation}
\end{lem}

It is worth noting that the functions $\tBRm_\delta$ defined in \eqref{eq:tildem} satisfy the same conditions \eqref{eq:mdeltabnd} as the $\BRm_\delta$. The second estimate in \eqref{eq:general_maximal_ests} is then contained in the following result.

\begin{prop}\label{prp:sobolevb2}
For all $\delta \in \Dyad$,
\[
\|M_\delta^\lmaxf\|_{L^{2}\rightarrow L^2} \leq \|M_\delta^\maxf\|_{L^{2}\rightarrow L^2}\lesssim 1.
\]
\end{prop}
\begin{proof}
The first inequality is obvious. As for the second one, from the spectral decomposition \eqref{eq:sp_dec} it is easily seen that, for all $f \in \Sch(G)$,
\[
      \left\| \int_0^\infty|\BRm_\delta(t\opL)f|^2 \,\frac{dt}{t}\right\|_{L^1}
			=\|f\|_2^2 \int_0^\infty |\BRm_\delta(t)|^2 \,\frac{dt}{t} 
\]
(see, e.g., \cite[p.\ 101]{ADM}), and moreover from \eqref{eq:mdeltabnd} it follows that $\int_{0}^\infty |\BRm_\delta(t)|^2 \,\frac{dt}{t} \lesssim \delta$. Clearly analogous estimates hold if $\BRm_\delta$ is replaced by $\tBRm_\delta$ defined in \eqref{eq:tildem}, so we conclude that
\begin{equation}\label{eq:mdeltaintdelta}
\left\| \int_0^\infty|\BRm_\delta(t\opL)f|^2 \,\frac{dt}{t}\right\|_{L^1} \lesssim \delta \|f\|_2^2, \qquad \left\| \int_0^\infty|\tBRm_\delta(t\opL)f|^2 \,\frac{dt}{t}\right\|_{L^1} \lesssim \delta \|f\|_2^2.
\end{equation}
The desired estimate $\|M_\delta^\maxf\|_{L^{2}\rightarrow L^2}\lesssim 1$ then follows by integrating \eqref{eq:fundthmcalc_global} over $G$, applying the Cauchy--Schwarz inequality to the right-hand side and majorizing each factor with one of the estimates \eqref{eq:mdeltaintdelta}.
\end{proof}

\section{The `maximal-to-nonmaximal' reduction}\label{s:chapsqfunarg}

For the rest of the paper, we restrict to the case of an H-type group $G$ and the distinguished sub-Laplacian $\opL$ thereon. Our aim is proving the estimates \eqref{eq:weightedmaxopest} for the maximal function $M_\delta^\lmaxf$, that is, the estimate
\begin{equation}\label{eq:maximal_w}
\|M_\delta^\lmaxf \|_{L^{2}(1/w) \to L^2(1/w)} \lessapprox 1
\end{equation}
for the weights $w=(1+|\cdot|)^a$ and $w=(1+\fstW)^b$ and suitable values of $a,b \geq 0$. Clearly a necessary condition for this to hold is the uniform estimate
\begin{equation}\label{eq:nonmaximal_w_sup}
\sup_{0 < s < 1} \|\BRm_\delta(s\opL) \|_{L^{2}(1/w) \to L^2(1/w)} \lessapprox 1
\end{equation}
for the norm of the nonmaximal operators, which by \eqref{eq:kerneldilationeq} actually reduces to
\begin{equation}\label{eq:nonmaximal_w}
\|\BRm_\delta(\opL) \|_{L^{2}(1/w) \to L^2(1/w)} \lessapprox 1
\end{equation}
for ``quasi-homogeneous'' weights $w$ such as $(1+|\cdot|)^a$ and $(1+\fstW)^b$.

In this section we will show that, for certain polynomial weights $w$, the implication can be essentially reversed and that, roughly speaking, it is enough to prove \eqref{eq:nonmaximal_w_sup} to obtain \eqref{eq:maximal_w}. This ``maximal-to-nonmaximal' reduction result unfortunately does not directly apply to the weights $(1+|\cdot|)^a$ and $(1+\fstW)^b$, unless $a \in 4\NN_0$ and $b\in 2\NN_0$. Nevertheless the result will play an important role in the following sections in treating the case of ``fractional'' $a$ and $b$, leading to the proof of \eqref{eq:weightedmaxopest}.

Similarly to \cite{CRV,GM}, one of the main techniques in the proof is the reduction of the desired estimate to a square function estimate.
Namely, from \eqref{eq:fundthmcalc_local} and the Cauchy--Schwarz inequality we immediately deduce that
\begin{equation}\label{eq:max_sqf}
\| M_\delta^\lmaxf f \|_{L^{2}(1/w)}^2 \leq 2\delta^{-1} \| \SqT_\delta f \|_{L^2((0,1),ds/s) \otimes L^{2}(1/w)} \| \tilde\SqT_\delta f \|_{L^2((0,1),ds/s) \otimes L^{2}(1/w)} ,
\end{equation}
where $\SqT_\delta : f \mapsto (\BRm_\delta(t\opL)f)_{t\in(0,1)}$ and $\tilde\SqT_\delta$ is the analogous operator with $\tBRm_\delta$ in place of $\BRm_\delta$.
Note now that
\begin{equation}\label{eq:adj_norm}
\| \SqT_\delta  \|_{L^2(1/w) \to L^2((0,1),ds/s) \otimes L^{2}(1/w) } = \| \SqT^\adj_\delta  \|_{L^2((0,1),ds/s) \otimes L^{2}(w) \to L^2(w)},
\end{equation}
where $\SqT_\delta^\adj$ denotes the adjoint operator to $\SqT_\delta$, which is given, for 
(measurable) 
families of functions $(\varphi_s)_{s\in(0,1)}$, by 
\begin{equation}\label{eq:SqTadj}
\SqT_\delta^\adj(\varphi_s)_s = \int_0^1 \BRm_\delta(s\opL)\varphi_s \,\frac{ds}{s}.
\end{equation}
We are then reduced to the study of norm estimates for operators of the form \eqref{eq:SqTadj}, featuring a decay as $\delta \to 0$ that is sufficient to compensate the power $\delta^{-1}$ in \eqref{eq:max_sqf}.

Let us first define the class of weights that we will be considering.

\begin{dfn}
A polynomial on $G$ is called \emph{sum-of-squares} if it can be written as a sum of squares of real-valued polynomials.
\end{dfn}

For a sum-of-squares polynomial weight $w$, norm estimates for operators of the form \eqref{eq:SqTadj} can be deduced from the following result.

\begin{prop}\label{prp:complicatedthm1_poly}
Let $w$ be a sum-of-squares polynomial on $G$. Let $I \subseteq \RR^+$ be an interval. Let $F \in C^\infty_c(\RR^+)$, and let $\SqT_F$ be defined by
\begin{equation}\label{eq:sqt_abs_def}
\SqT_F(\varphi_s)_s \defeq \int_I F(s\opL) \varphi_s \,\frac{ds}{s}
\end{equation}
for families $(\varphi_s)_{s\in I}$ of functions on $G$.
Then
\[
\|\SqT_F(\varphi_s)_s\|^2_{L^2(w)} 
\leq (1+\deg w) \, \kappa \int_I \|F(s\opL)\varphi_s\|^2_{L^2(w)} \,\frac{ds}{s},
\]
where $\kappa$ is the $\tfrac{ds}{s}$-measure of the support of $F$, and $\deg w$ is the degree of the polynomial $w$.
\end{prop}

The key ingredient of the proof of Proposition \ref{prp:complicatedthm1_poly} is the following lemma, which is based on the group Fourier transform discussed in Section \ref{s:fouriertheory}.

\begin{lem}\label{lem:support_poly}
Let $F \in C^\infty_c(\RR^+)$. Let $P$ be a polynomial on $G$.
\begin{enumerate}[label=(\roman*)]
\item\label{en:support_poly_mu} For all $\alpha,\beta \in \NN_0^m$, and all $f \in \Sch(G)$,
\[
\supp \left( (P F(\opL) f)\widehat{\;}(\cdot,\alpha,\beta) \right) \subseteq \bigcup_{k \in \NN_0 \tc |k-|\alpha|| \leq \deg P} \supp \left(F(|\cdot| c(k)) \right),
\]
where the supports in both left- and right-hand side are meant to be of functions with domain $\RR^n \setminus \{0\}$.
\item\label{en:support_poly_t} For all $\mu \in \RR^n \setminus \{0\}$, $\alpha,\beta \in \NN_0^m$, and all $f \in \Sch(G)$,
\[
\supp \left( (P F(\cdot \opL) f)\widehat{\;}(\mu,\alpha,\beta) \right) \subseteq \bigcup_{k \in \NN_0 \tc |k-|\alpha|| \leq \deg P} \supp \left(F(\cdot |\mu| c(k)) \right),
\]
where the supports in both left- and right-hand side are meant to be of functions with domain $\RR^+$.
\end{enumerate}
\end{lem}
\begin{proof}
Let $K_t$ denote the convolution kernel of $F(t \opL)$ for all $t \in \RR^+$. Then, by iteratively applying the Leibniz rules from Section \ref{s:secleibniz},
\[
P F(t\opL) f = P (f * K_t) = \sum_j (Q_j f) * (R_j K_t)
\]
for suitable polynomials $Q_j,R_j$ (depending only on $P$) with $\deg Q_j, \deg R_j \leq \deg P$, and therefore, by \eqref{eq:matrixcompconv},
\[
(P F(t\opL) f)\widehat{\;}(\mu,\alpha,\beta) = \sum_j \sum_{\gamma \in \NN_0^m} \widehat{R_j K_t}(\mu,\alpha,\gamma) \, \widehat{Q_j f}(\mu,\gamma,\beta).
\]
Taking any of the $R_j$ in place of $P$, we are then reduced to proving that
\begin{equation}\label{eq:supp_reduced}
\begin{aligned}
\supp \left( \widehat{P K_1}(\cdot,\alpha,\beta) \right) &\subseteq \bigcup_{k \in \NN \tc |k-|\alpha|| \leq \deg P} \supp \left(F(|\cdot| c(k)) \right), \\
\supp \left( t \mapsto \widehat{P K_t}(\mu,\alpha,\beta) \right) &\subseteq \bigcup_{k \in \NN \tc |k-|\alpha|| \leq \deg P} \supp \left(F(\cdot |\mu| c(k)) \right).
\end{aligned}
\end{equation}
By linearity, we may assume that $P$ factorises as $P(z,u) = Q(z) R(u)$.

Let $H_t(z,u) = R(u) K_t(z,u)$. By iteratively applying Lemma \ref{lem:matrixcomponentcoordchange} and the identities \eqref{eq:partzeta1} and \eqref{eq:partzeta2}, we can write
\[\begin{split}
\widehat{P K_t}(\mu,\alpha,\beta) &=  \sum_{\substack{\alpha',\beta' \in \NN_0^m \\ |\alpha'-\alpha| + |\beta'-\beta| \leq \deg Q}} c_{\mu,\alpha,\beta,\alpha',\beta'} \widehat{H_t}(\mu,\alpha',\beta') \\
&=  \sum_{\substack{k \in \NN_0 \\ |k-|\alpha|| + |k-|\beta|| \leq \deg Q}} c_{\mu,\alpha,\beta,k} \widehat{H_t}(\mu,k),
\end{split}\]
for suitable coefficients $c_{\mu,\alpha,\beta,\alpha',\beta'}$, $c_{\mu,\alpha,\beta,k} \in \CC$; the second identity is due to the fact that $H_t$ is radial. Similarly, by iteratively applying \eqref{eq:parturule} and \eqref{eq:fourierL},
\[\begin{split}
\widehat{H_t}(\mu,k) &= \sum_{\substack{ k' \in \NN_0,\, \gamma \in \NN_0^n \\ |k'-k|+|\gamma| \leq \deg R}} c_{\mu,k,k',\gamma} \left(\frac{\partial}{\partial \mu}\right)^\gamma \widehat{K_t}(\mu,k') \\
&= \sum_{\substack{ k',\ell \in \NN_0 \\ |k'-k|+\ell \leq \deg R}} c_{\mu,k,k',\ell} \, t^\ell F^{(\ell)}(t |\mu| c(k')).
\end{split}\]
Since $\supp F^{(\ell)} \subseteq \supp F$ for all $\ell \in \NN_0$, the containments \eqref{eq:supp_reduced} are easily deduced by combining the previous identities.
\end{proof}

\begin{proof}[Proof of Proposition \ref{prp:complicatedthm1_poly}]
Let $w = \sum_j P_j^2$ for some real-valued polynomials $P_j$ on $G$, so $\deg w = 2 \max_j \deg P_j$. Note that
\[\begin{split}
\|\SqT_F(\varphi_s)_s\|^2_{L^2(w)} 
&= \sum_j \left\| \int_I P_j F(s\opL) \varphi_s \frac{ds}{s} \right\|^2_{L^2(G)} \\
&= \sum_j \sum_{\alpha,\beta \in \NN_0^m} \int_{\RR^n \setminus \{0\}} \left|\int_I (P_j F(s\opL) \varphi_s)\widehat{\;}(\mu,\alpha,\beta) \frac{ds}{s}\right|^2  \, |\mu|^m \,d\mu 
\end{split}\]
by \eqref{eq:plancherel}. By Lemma \ref{lem:support_poly}\ref{en:support_poly_t}, for each $\mu \in \RR^n \setminus \{0\}$ and $\alpha,\beta \in \NN_0^m$, the $\frac{ds}{s}$-measure of the support of $s \mapsto (P_j F(s\opL) \varphi_s)\widehat{\;}(\mu,\alpha,\beta)$ is controlled by $(1+2\deg P_j) \kappa$. Hence, by the Cauchy--Schwarz inequality,
\[
\left|\int_I (P_j F(s\opL) \varphi_s)\widehat{\;}(\mu,\alpha,\beta) \frac{ds}{s}\right|^2 \leq (1+2\deg P_j) \, \kappa \int_I |(P_j F(s\opL) \varphi_s)\widehat{\;}(\mu,\alpha,\beta)|^2 \frac{ds}{s}.
\]
Therefore, again by \eqref{eq:plancherel},
\[\begin{split}
\|\SqT_F(\varphi_s)_s\|^2_{L^2(w)} &\leq \sum_j (1+\deg 2P_j) \, \kappa \int_I \|P_j F(s\opL) \varphi_s\|_{L^2(G)}^2 \frac{ds}{s} \\
&\leq (1+\deg w) \, \kappa \int_I \|F(s\opL) \varphi_s\|_{L^2(w)}^2 \frac{ds}{s},
\end{split}\]
and we are done.
\end{proof}

If we apply Proposition \ref{prp:complicatedthm1_poly} with $F = \BRm_\delta$ and $I = (0,1)$, we immediately deduce that
\begin{equation}\label{eq:adj_est}
 \| \SqT^\adj_\delta  \|_{L^2((0,1),ds/s) \otimes L^{2}(w) \to L^2(w)} \lesssim \delta^{1/2} \sup_{0 < s < 1} \| \BRm_\delta(s\opL) \|_{L^2(w) \to L^2(w)}.
\end{equation}
In view of \eqref{eq:max_sqf} and \eqref{eq:adj_norm}, this implies the estimate
\begin{equation}\label{eq:max_to_nonmax_red}
 \| M_\delta^\lmaxf f \|_{L^{2}(1/w) \to L^2(1/w)}^2 \lesssim \sup_{0 < s < 1} \| \BRm_\delta(s\opL) \|_{L^2(w) \to L^2(w)} \sup_{0 < s < 1} \| \tBRm_\delta(s\opL) \|_{L^2(w) \to L^2(w)},
\end{equation}
which provides the desired `maximal-to-nonmaximal' reduction (note that the norm of $\BRm_\delta(s\opL)$ on $L^2(w)$ is the same as that on $L^2(1/w)$ by self-adjointness).

It would be interesting to know if estimates of the form \eqref{eq:adj_est} and \eqref{eq:max_to_nonmax_red} hold for wider classes of weights. The methods used in the proof seem to strongly depend on the polynomial nature of $w$. In the next section, however, we will see that a sort of interpolation technique can be used to work around this obstruction in the case of certain fractional powers of polynomials.

\section{The `maximal-to-nonmaximal' reduction, take two}\label{s:chapsqfunarg_two}

As discussed in the previous section, a necessary condition for the maximal estimate \eqref{eq:maximal_w} to hold is the validity of nonmaximal estimates such as \eqref{eq:nonmaximal_w_sup} and \eqref{eq:nonmaximal_w}. In this section, instead of trying to revert the implication, we will show that certain two-weight estimates, from which \eqref{eq:nonmaximal_w} readily follows by interpolation, are also enough to obtain \eqref{eq:nonmaximal_w_sup} under suitable assumptions on the weight $w$. These two-weight estimates involve powers of the weight $w$, thus allowing us to apply the results of the previous section even when the weight $w$ is not a polynomial, provided that some power of $w$ is.
As we will see in the next section, in turn these two-weight estimates may be reduced to a `trace lemma', the proof of which will eventually be our main objective.

The aforementioned two-weight estimates are expressed in terms of a decomposition of the operators $\BRm_\delta(\opL)$.
For all $\delta \in \Dyad$, we define $J_\delta \in\NN$ so that
\begin{equation}\label{eq:jdefn}
   2^{J_\delta-1}\leq 20 \delta^{-1} \leq 2^{J_\delta}
\end{equation} 
and define
operators $R_{\delta,j}$, $j=1,\dots,J_\delta$, on $L^2(G)$ by 
\begin{equation}\label{eq:rjdefn}
   \widehat{R_{\delta,j} f}(\mu,\alpha,\beta) \defeq \begin{cases}
	\chr_{[2^j,2^{j+1})}(c(|\alpha|)) \,\widehat{f}(\mu,\alpha,\beta) &\text{ for } \, j=1,\ldots,J_\delta-1,\\
 \chr_{[2^{J_\delta},\infty)}(c(|\alpha|)) \,\widehat{f}(\mu,\alpha,\beta) &\text{ for } j=J_\delta,
\end{cases}
 \end{equation} 
where $c(k)$ is defined as in \eqref{eq:eig}.
In order to motivate the subsequent developments, let us first present the simple interpolation argument yielding the `nonmaximal' estimate \eqref{eq:nonmaximal_w}.

\begin{prop}\label{prp:tracetomdelta}
Let $w$ be a weight on $G$ and $N>1$. Suppose that, for all $\delta \in \Dyad$ and for all $1\leq j\leq J_\delta$, the estimates 
\begin{equation}\label{eq:nl1}
\|R_{\delta,j} \BRm_\delta(\opL)f\|^2_2\lessapprox C(\delta,j) \|f\|^2_{L^2(w)}
\end{equation} 
and
\begin{equation}\label{eq:nl2}
\|R_{\delta,j} \BRm_\delta(\opL)f\|^2_{L^2(w^N)}\lessapprox C(\delta,j)^{1-N} \|f\|^2_{L^2(w)}+ \|f\|^2_{L^2(w^N)}
\end{equation}
hold, where $C(\delta,j)>0$. Then, for all $\delta \in \Dyad$,
\begin{equation}\label{eq:n13}
\|\BRm_\delta(\opL)\|_{L^2(1/w) \to L^2(1/w)} \lessapprox 1.
\end{equation}
\end{prop}
\begin{proof}
Define $S \defeq \{(z,u)\in G \tc C(\delta,j) \, w(z,u)<1\}$. Note that, if $f$ is supported in either $S$ or its complement, then one of the two summands in the right-hand side of \eqref{eq:nl2} dominates the other one.
Consequently from \eqref{eq:nl2} we deduce
\begin{gather*}
\|R_{\delta,j} \BRm_\delta(\opL) (\chr_S f)\|^2_{L^2(w^N)}\lessapprox C(\delta,j)^{1-N} \|\chr_S f\|^2_{L^2(w)}, \\
\|R_{\delta,j} \BRm_\delta(\opL) (\chr_{G\setminus S}) f\|^2_{L^2(w^N)}\lessapprox \|\chr_{G \setminus S} f\|^2_{L^2(w^N)}.
\end{gather*}
If we interpolate the first estimate with \eqref{eq:nl1}, and the second estimate with the trivial $L^2$ estimate $\| R_{\delta,j} \BRm_\delta(\opL) \|_{L^2 \to L^2} \lesssim 1$, then we obtain
\begin{gather*}
\|R_{\delta,j} \BRm_\delta(\opL) (\chr_S f)\|^2_{L^2(w)}\lessapprox \|\chr_S f\|^2_{L^2(w)}, \\
\|R_{\delta,j} \BRm_\delta(\opL) (\chr_{G\setminus S}f)\|^2_{L^2(w)}\lessapprox \|\chr_{G\setminus S} f\|^2_{L^2(w)},
\end{gather*}
and consequently
\begin{equation*}
\|R_{\delta,j} \BRm_\delta(\opL)f\|^2_{L^2(w)}\lessapprox \|f\|^2_{L^2(w)}.
\end{equation*}
Since $J_\delta \simeq |\log(\delta)| \lessapprox 1$, this estimate holds if $R_{\delta,j} \BRm_\delta$ is replaced by just $\BRm_\delta$, and the desired result follows by self-adjointness of $\BRm_\delta(\opL)$.
\end{proof}

In this section we show that, for a certain class of weights $w$, the assumptions \eqref{eq:nl1} and \eqref{eq:nl2} are essentially enough to deduce the maximal estimate \eqref{eq:maximal_w} too. 

\begin{dfn}
A weight $w$ on $G$ will be called:
\begin{itemize}
\item \emph{quasi-homogeneous}, if $w \simeq 1+w_0^a$ for some $a \geq 0$ and some nonnegative function $w_0$ on $G$, which is $1$-homogeneous with respect to the group dilations;
\item \emph{temperate}, if there exists $\alpha \geq 0$ such that, for all $x,y \in G$,
\[
w(x) \lesssim w(y) \, (1+d(x,y))^\alpha;
\]
\item \emph{admissible}, if $w \in A_2(G)$, $w$ is quasi-homogeneous and temperate.
\end{itemize}
We denote by $\Adm(G)$ the collection of admissible weights on $G$.
\end{dfn}

As discussed in the previous section, estimates for the maximal function $M^\lmaxf_\delta$ can be reduced to estimates for the operator $\SqT_\delta^\adj$ defined in \eqref{eq:SqTadj}. These are contained in the following statement.

\begin{prop}\label{prp:max_nonmax_sq}
Let $w$ be a weight on $G$ and $N > 1$ such that:
\begin{itemize}
\item $w \in \Adm(G)$;
\item $w^N$ is a sum-of-squares polynomial on $G$;
\item the estimates \eqref{eq:nl1} and \eqref{eq:nl2} hold for all $\delta \in \Dyad$ and $j = 1,\dots,J_\delta$.
\end{itemize}
Then, for all $\delta \in \Dyad$,
\begin{equation}\label{eq:sqf_adj_w_est}
\|\SqT_\delta^\adj \|_{L^2((0,1),ds/s) \otimes L^{2}(w) \to L^2(w)} \lessapprox \delta^{1/2}.
\end{equation}
\end{prop}

As an immediate consequence, in view of \eqref{eq:max_sqf} and \eqref{eq:adj_norm}, we obtain the following estimate for $M_\delta^\lmaxf$.

\begin{cor}\label{cor:max_nonmax}
Let $w$ be a weight on $G$ and $N > 1$ such that:
\begin{itemize}
\item $w \in \Adm(G)$;
\item $w^N$ is a sum-of-squares polynomial on $G$;
\item the estimates \eqref{eq:nl1} and \eqref{eq:nl2} hold for all $\delta \in \Dyad$ and $j = 1,\dots,J_\delta$, as well as the corresponding estimates where $\BRm_\delta$ is replaced by $\tBRm_\delta$ defined as in \eqref{eq:tildem}.
\end{itemize}
Then, for all $\delta \in \Dyad$,
\[
\|M_\delta^\lmaxf \|_{L^{2}(1/w) \to L^2(1/w)} \lessapprox 1.
\]
\end{cor}

The proof of Proposition \ref{prp:max_nonmax_sq} will be given at the end of the section, after a number of preliminary lemmas.

First of all we show that, in place of $\SqT_\delta^\adj$, it is enough to consider a `portion' of it, where the integral in \eqref{eq:SqTadj} is restricted to $(1/8,1)$.

\begin{lem}\label{lem:sfa2}
Let $w \in A_2(G)$ be a quasi-homogeneous weight. Then, for all $\delta \in \Dyad$,
\[
     \|\SqT_\delta^\adj\|_{L^2((0,1),ds/s) \otimes L^2(w)\to L^2(w)} \\
		\lesssim
		\delta^{1/2} +
		\| \Psi_\delta\|_{L^2((1/8,1),ds/s) \otimes  L^2(w)\to L^2(w)},
\]
where the implicit constant may depend on $w$, and
\[
\Psi_\delta(\varphi_s)_s \defeq \int_{1/8}^{1} \BRm_\delta(s \opL)\varphi_s \,\frac{ds}{s}.
\]
\end{lem}
\begin{proof}
We first choose
$\vartheta \in C^\infty_c(\RR)$ with $\supp(\vartheta) \subseteq (1,4)$ and 
\[
1 = \sum_{k\in\ZZ} \vartheta(2^{-k}s), \quad s>0.
\]
Note that $\supp(\BRm_\delta) \subseteq [1/2,1]$ (here we use that $\delta\leq\tfrac{1}{2}$), so, for all $k \in \ZZ$,
\[
\BRm_\delta(t\opL) \,\vartheta(2^{-k}\opL)=0 \text{ for } t \notin I_k \defeq (2^{-k-3},2^{-k}),
\]
and moreover $I_k \cap [0,1] = \emptyset$ for $k \leq -4$. Hence, 
from Lemma \ref{lem:weightslem3} (note that $1/w\in A_2(G)$)
we readily deduce
\[\begin{split}
      \|\SqT_\delta^\adj(\varphi_s)_s\|^2_{L^2(w)}
			&
			\simeq 
			\sum_{k\in\ZZ} \|\vartheta(2^{-k}\opL) \SqT_\delta^\adj(\varphi_s)_s\|^2_{L^2(w)} \\
			&= \sum_{k=-3}^\infty\|\tilde\Psi_{\delta,k}(\varphi_s)_{s \in I_k}\|_{L^2(w)}^2,
\end{split}\]
where
\[
\tilde\Psi_{\delta,k} (g_s)_{s} \defeq \int_{I_k} \BRm_\delta(s\opL) \vartheta(2^{-k}\opL) g_s \,\frac{ds}{s},
\]
and in particular
\[
\|\SqT_\delta^\adj\|_{L^2((0,1),ds/s) \otimes L^2(w) \to L^2(w)} \lesssim \sup_{k \geq -3} \|\tilde\Psi_{\delta,k}\|_{L^2(I_k,ds/s) \otimes L^2(w) \to L^2(w)}.
\]
Since $w$ is quasi-homogeneous, $w \simeq 1+w_0^a$ for some nonnegative $1$-homogeneous function $w_0$ on $G$ and some $a \geq 0$. In order to conclude, it will be sufficient to prove that,
for all $k \geq -3$,
\begin{multline}\label{eq:inhomog_estimate}
\|\tilde\Psi_{\delta,k}\|_{L^2(I_k,ds/s) \otimes L^2(1+w_0^a) \to L^2(1+w_0^a)} \\
\lesssim \max_{b \in \{0,a\} }\|\Psi_\delta\|_{L^2(I_0,ds/s) \otimes L^2(1+w_0^b) \to L^2(1+w_0^b)},
\end{multline}
where the implicit constant is independent of $k$; indeed, the term with $b=0$ in the right-hand side is controlled by a multiple of $\delta^{1/2}$, by Proposition \ref{prp:complicatedthm1_poly}.

To prove \eqref{eq:inhomog_estimate}, note that, by Lemma \ref{lem:kerneldilation},
\[
\tilde\Psi_{\delta,k}(g_s)_{s \in I_k} = \fDil_{2^{-k}} \tilde\Psi_{\delta,0}(\fDil_{2^k} g_{2^{-k}s})_{s \in I_0},
\]
whence
\begin{equation}\label{eq:homog_dec_norm}
\begin{split}
\|\tilde\Psi_{\delta,k}(g_s)_s\|_{L^2(1+w_0^a)} 
&\simeq \max_{b \in \{0,a\} }\|\tilde\Psi_{\delta,k}(g_s)_s\|_{L^2(w_0^b)}  \\
&= \max_{b \in \{0,a\} } 2^{(Q-b)k/4} \|\tilde\Psi_{\delta,0}(\fDil_{2^{k}} g_{2^{-k}s})_s\|_{L^2(w_0^b)}\\
&\lesssim \max_{b \in \{0,a\} } 2^{(Q-b)k/4} \|\tilde\Psi_{\delta,0}(\fDil_{2^{k}} g_{2^{-k}s})_s\|_{L^2(1+w_0^b)}\\
\end{split}
\end{equation}
On the other hand,
\[
\|(\fDil_{2^{k}} g_{2^{-k}s})_s\|_{L^2(I_0,ds/s) \otimes L^2(w_0^b)} = 2^{-(Q-b)k/4} \|(g_{s})_s\|_{L^2(I_k,ds/s) \otimes L^2(w_0^b)},
\]
whence also
\begin{equation}\label{eq:homog_dyad_norm}
\|(\fDil_{2^{k}} g_{2^{-k}s})_s\|_{L^2(I_0,ds/s) \otimes L^2(1+w_0^b)} \lesssim 2^{-(Q-b)k/4} \|(g_{s})_s\|_{L^2(I_k,ds/s) \otimes L^2(1+w_0^b)}
\end{equation}
(here we are using that $2^{bk/4} \gtrsim 1$, since $b \in \{0,a\}$, $a \geq 0$, $k \geq -3$). A comparison of \eqref{eq:homog_dec_norm} and \eqref{eq:homog_dyad_norm} immediately yields 
\begin{multline*}
\|\tilde\Psi_{\delta,k}\|_{L^2(I_k,ds/s) \otimes L^2(1+w_0^a) \to L^2(1+w_0^a)} \\
\lesssim \max_{b \in \{0,a\} }\|\tilde\Psi_{\delta,0}\|_{L^2(I_0,ds/s) \otimes L^2(1+w_0^b) \to L^2(1+w_0^b)}.
\end{multline*}
On the other hand, since $\vartheta \in C^\infty_c(\RR^+)$ and $w \in A_2(G)$, $\vartheta(\opL)$ is bounded on $L^2(1+w_0^b)$ for $b\in\{0,a\}$ (see Lemmas \ref{lem:stein} and \ref{lem:kernelestimate}), whence
\[
\|\tilde\Psi_{\delta,0}\|_{L^2(I_0,ds/s) \otimes L^2(1+w_0^b) \to L^2(1+w_0^b)} \lesssim \|\Psi_{\delta}\|_{L^2(I_0,ds/s) \otimes L^2(1+w_0^b) \to L^2(1+w_0^b)}
\]
and \eqref{eq:inhomog_estimate} follows.
\end{proof}

Let $\chi \in C^\infty(\RR)$ be even, real-valued and such that
\[
\supp(\chi)\subseteq (-2,2),\qquad \chi(\lambda)=1 \text{ for } \lambda\in(-1,1).
\]
Define, for $\lambda\in\RR$, $\BRn_\delta(\lambda) = \BRm_\delta(\lambda^2)$. We now decompose $\BRn_\delta = \BRn_\delta^\rI + \BRn_\delta^\rII$, where
\[
\Four\BRn_\delta^\rI(\lambda) = \chi(\delta^2\lambda) \,\Four\BRn_\delta(\lambda), \qquad 
\Four\BRn_\delta^\rII(\lambda) = (1-\chi(\delta^2\lambda)) \, \Four\BRn_\delta(\lambda)
\]
and $\Four$ denotes the Euclidean Fourier transform.
Then
$\BRm_\delta(t\opL)= \BRn_\delta^\rI(\sqrt{t\opL}) + \BRn_\delta^\rII(\sqrt{t\opL})$,
and correspondingly $\Psi_\delta = \Psi^\rI_\delta + \Psi^\rII_\delta$, where
\[
\Psi_\delta^\rI(\varphi_s)_s \defeq \int_{1/8}^1 \BRn_\delta^\rI(\sqrt{s\opL}) \varphi_s \,\frac{ds}{s}, \qquad \Psi_\delta^\rII(\varphi_s)_s \defeq \int_{1/8}^1 \BRn_\delta^\rII(\sqrt{s\opL}) \varphi_s \,\frac{ds}{s}.
\]

We now show that $\Psi^\rII_\delta$ is effectively negligible in our analysis.

\begin{lem}\label{lem:psiiismall}
For all $w \in A_2(G)$ and $k \in \NN_0$,
\begin{equation}\label{eq:supnorm_II}
\sup_{s\in\RR^+} \|\BRn_\delta^\rII(\sqrt{s\opL}) f\|_{L^2(w)} \lesssim_k \delta^k \|f\|_{L^2(w)}
\end{equation}
and
\begin{equation}\label{eq:intnorm_II}
\|\Psi_\delta^\rII\|_{L^2((1/8,1),ds/s) \otimes L^2(w) \to L^2(w)} \lesssim_k \delta^k,
\end{equation}
where the implict constants may depend on $w$.
\end{lem}
\begin{proof}
Since $\BRn_\delta$ is even and vanishes at the origin ($\delta \leq 1/2$), we can write $\BRn_\delta(\lambda) = \BRn_\delta^+(\lambda) + \BRn_\delta^+(-\lambda)$, where $\supp(\BRn_\delta^+)\subseteq(0,\infty)$. Correspondingly
\begin{equation}\label{eq:ft_identity}
\begin{split}
 \delta^{-1} \Four\BRn_\delta^\rII(\delta^{-1} \lambda) 
&= 2 \delta^{-1} (1-\chi(\delta \lambda)) \Re \Four\BRn_\delta^+(\delta^{-1} \lambda) \\
&= 2 (1-\chi(\delta \lambda)) \Re \left( e^{-2\pi i\lambda\delta^{-1}} \Four N_\delta(\lambda) \right),
\end{split}
\end{equation}
where $N_\delta(\lambda) \defeq \BRn_\delta^+(\delta\lambda+1)$ and we have used that $\BRn_\delta^+$ is real-valued.

From \eqref{eq:mdeltabnd} it easily follows that $\supp(N_\delta) \subseteq [-1/2,0]$ and $\|N_\delta^{(j)}\|_\infty \lesssim_j 1$ for all $j \in \NN_0$ (uniformly in $\delta$). Hence each Schwartz seminorm of $\Four N_\delta$ is bounded uniformly in $\delta$. Since $1-\chi(\delta \lambda)$ vanishes unless $|\lambda| \geq \delta^{-1}$, it is readily seen that each Schwartz seminorm of $\lambda \mapsto (1-\chi(\delta \lambda)) \, e^{-2\pi i\lambda\delta^{-1}} \Four N_\delta(\lambda)$ is majorized by $\delta^k$ uniformly in $\delta$ for arbitrarily large $k$. By \eqref{eq:ft_identity}, this implies that each Schwartz seminorm of $\BRn_\delta^\rII(\delta \cdot)$ is majorized by $\delta^k$ uniformly in $\delta$ for arbitrarily large $k$.

Since $\BRn_\delta^\rII$ is even, from Lemmas \ref{lem:kernelestimate} and \ref{lem:stein} we deduce, for all $s>0$ and for all $k\in\NN$, the estimate 
\begin{equation}\label{eq:maxfcnt_II}
|\BRn_\delta^\rII(\sqrt{s\opL}) f(x)| \lesssim_k \delta^k M f(x) \quad\text{a.e.}
\end{equation}
where $M$ denotes the Littlewood-Hardy maximal operator on $G$ and the implicit constant in $\lesssim$ does not depend on $s$ or $\delta$. Since $w \in A_2(G)$, $M$ is bounded on $L^2(w)$ and \eqref{eq:supnorm_II} follows immediately from \eqref{eq:maxfcnt_II}. Moreover, by the Cauchy--Schwarz inequality,
\[
\|\Psi_\delta^\rII(\varphi_s)_s\|^2_{L^2(w)}
\lesssim \int_{1/8}^1 \|\BRn_\delta^\rII(\sqrt{s\opL}) \varphi_s\|_{L^2(w)}^2 \,\frac{ds}{s}
\]
and \eqref{eq:intnorm_II} follows by applying \eqref{eq:supnorm_II} to the inner norm.
\end{proof}

The analysis of $\Psi_\delta$ is then essentially reduced to that of the `main term' $\Psi_\delta^\rI$, for which we can exploit the support condition on $\Four \BRn_\delta^I$ and finite propagation speed for $\opL$. This leads to the following result.

\begin{lem}\label{lem:fps_dec}
Let $w \in \Adm(G)$, and assume that $\inf w = 1$. Let $A_l = \{ x \in G \tc 2^{l-1} \leq w(x) < 2^l \}$ for all $l \in \NN$. Then
\[
\|\Psi_\delta (\varphi_s)_s\|_{L^2(w)}^2 \lessapprox \delta \, \|(\varphi_s)_s\|_{L^2((1/8,1),ds/s)\otimes L^2(w)}^2 + \sum_{l \in \NN} \|\Psi_\delta (\chr_{A_l} \varphi_s)_s\|_{L^2(w)}^2.
\]
\end{lem}
\begin{proof}
Note that $G = \bigcup_{l \in \NN} A_l$, since $w \geq 1$.
In view of the decomposition $\Psi_\delta = \Psi_\delta^\rI + \Psi_\delta^\rII$ and Lemma \ref{lem:psiiismall}, it is enough to prove that
\begin{equation}\label{eq:PsiI_fps}
\|\Psi_\delta^\rI (\varphi_s)_s\|_{L^2(w)}^2 \lessapprox \sum_{l \in \NN} \|\Psi_\delta^\rI (\chr_{A_l} \varphi_s)_s\|_{L^2(w)}^2.
\end{equation}

Let $K_{\delta,t}$ be the convolution kernel of $\BRn_\delta^\rI(\sqrt{t\opL})$. Since $\supp \Four\BRn_\delta^\rI \subseteq [-2\delta^{-1},2\delta^{-1}]$, by finite propagation speed (Lemma \ref{lem:finitepropspeed}) we deduce that, for $|t| \leq 1$,
\begin{equation}\label{eq:fps33}
\supp(K_{\delta,t}) \subseteq \overline{B}(0,4\pi\delta^{-2}).
\end{equation}
Since $w$ is temperate, there exists $\alpha \geq 0$ such that, for all $x,z \in G$,
\begin{equation}
    \frac{w(x)}{w(z)} \lesssim (1+d(x,z))^\alpha.
\end{equation}
From this it immediately follows that, for a suitable constant $\kappa \geq 0$ and all $l \in \NN$,
\[
\overline B(A_l,8\pi\delta^{-2}) \subseteq \{x \in G \tc 2^{l-\kappa |\log(\delta)|-1} \leq 1 + w_0(x) \leq 2^{l+\kappa|\log(\delta)|} \},
\]
which implies that $\overline B(A_l,4\pi\delta^{-2}) \cap \overline B(A_{l'},4\pi\delta^{-2}) \neq \emptyset$ only if $|l'-l| \leq \kappa|\log(\delta)|$.

Observe now that, by \eqref{eq:fps33},
$\supp \Psi^\rI_\delta (\chr_{A_l} \varphi_s)_s \subseteq \overline B(A_l,4\pi\delta^{-2})$. This means that, in the decomposition
\[
\Psi^\rI_\delta (\varphi_s)_s = \sum_{l \in \NN} \Psi^\rI_\delta (\chr_{A_l} \varphi_s)_s,
\]
the number of nonvanishing summands at each point of $G$ is $\lessapprox 1$, and \eqref{eq:PsiI_fps} immediately follows.
\end{proof}

\begin{proof}[Proof of Proposition \ref{prp:max_nonmax_sq}]
Observe first that \eqref{eq:nl1} and \eqref{eq:nl2} hold in a slightly enhanced form:
\begin{equation}\label{eq:nl1_enh}
\|R_{\delta,j} \BRm_\delta(s\opL)f\|^2_2\lessapprox C(\delta,j) \|f\|^2_{L^2(w)}
\end{equation} 
and
\begin{equation}\label{eq:nl2_enh}
\|R_{\delta,j} \BRm_\delta(s\opL)f\|^2_{L^2(w^N)}\lessapprox C(\delta,j)^{1-N} \|f\|^2_{L^2(w)}+ \|f\|^2_{L^2(w^N)}
\end{equation}
uniformly in $s \in (1/8,1)$. This is an immediate consequence of the observation
that
\[
R_{\delta,j} \BRm_\delta(s\opL) = \fDil_s R_{\delta,j} \BRm_\delta(\opL) \fDil_{s^{-1}},
\]
(cf.\ Lemma \ref{lem:kerneldilation}) and that moreover, by quasi-homogeneity, $w \circ \delta_s \simeq w$ uniformly in $s \in (1/8,8)$.

Without loss of generality, we may assume that $\inf w = 1$. Then, by Lemmas \ref{lem:sfa2} and \ref{lem:fps_dec}, we are reduced to proving that
\begin{equation}\label{eq:reduced_estimate}
\sum_{l \in \NN} \|\Psi_\delta (\chr_{A_l} \varphi_s)_s\|_{L^2(w)}^2 \lessapprox \delta \, \|(\varphi_s)_s\|_{L^2((1/8,1),ds/s)\otimes L^2(w)}^2.
\end{equation}

Note now that, for $j = 1,\dots,J_\delta$,
\[
\|R_{\delta,j} \Psi_\delta (\chr_{A_l} \varphi_s)_s\|_{L^2(w)} 
\leq \|R_{\delta,j} \Psi_\delta (\chr_{A_l} \varphi_s)_s\|_{L^2(G)}^{(N-1)/N} \|R_{\delta,j} \Psi_\delta (\chr_{A_l} \varphi_s)_s\|_{L^2(w^N)}^{1/N}.
\]
Moreover, by Proposition \ref{prp:complicatedthm1_poly} and \eqref{eq:nl2_enh},
\begin{multline*}
\|R_{\delta,j} \Psi_\delta (\chr_{A_l} \varphi_s)_s\|_{L^2(w^N)}^2 
\lesssim \delta \int_{1/8}^1 \|R_{\delta,j} \BRm_\delta(s\opL) \chr_{A_l} \varphi_s\|_{L^2(w^N)}^2 \frac{ds}{s} \\
\lessapprox \delta \, \max\{ 2^{(N-1)l},  C(\delta,j)^{1-N}\} \| (\chr_{A_l} \varphi_s)_s \|^2_{L^2((1/8,1),ds/s) \otimes L^2(w)};
\end{multline*}
similarly, by Proposition \ref{prp:complicatedthm1_poly}, \eqref{eq:nl1_enh} and the trivial $L^2$ estimate for $R_{\delta,j} \BRm_\delta(s\opL)$,
\begin{multline*}
\|R_{\delta,j} \Psi_\delta (\chr_{A_l} \varphi_s)_s\|_{L^2(G)}^2
\lesssim \delta \int_{1/8}^1 \|R_{\delta,j} \BRm_\delta(s\opL) \chr_{A_l} \varphi_s \|_{L^2(G)}^2 \frac{ds}{s} \\
\lessapprox \delta \, \min\{ 2^{-l}, C(\delta,j)\} \| (\chr_{A_l} \varphi_s)_s \|^2_{L^2((1/8,1),ds/s) \otimes L^2(w)}.
\end{multline*}
Hence
\[
\|R_{\delta,j} \Psi_\delta (\chr_{A_l} \varphi_s)_s\|_{L^2(w)}^2 
\lessapprox
\delta \, \| (\chr_{A_l} \varphi_s)_s \|^2_{L^2((1/8,1),ds/s) \otimes L^2(w)},
\]
and
\[
\sum_{l \in \NN} \|R_{\delta,j} \Psi_\delta (\chr_{A_l} \varphi_s)_s\|_{L^2(w)}^2  
\lessapprox \delta \, \| (\varphi_s)_s \|^2_{L^2((1/8,1),ds/s) \otimes L^2(w)}.
\]
Since $J_\delta \lessapprox 1$, summing in $j=1,\dots,J_\delta$ gives \eqref{eq:reduced_estimate}.
\end{proof}

\section{Reduction to dual trace lemmas}\label{s:chapredtotrace}

The aim of this section is to reduce proving the estimates we need, that is \eqref{eq:nl1} and \eqref{eq:nl2}, in the case of the weights $w = (1+|\cdot|)^a$ and $w = (1+\fstW)^b$, to proving suitable `trace lemmas'. It is easily checked (see Lemma \ref{lem:A2weights}) that such weights $w$ are admissible. Moreover $(1+|(z,u)|)^4 \simeq 1+|z|^4+|u|^2$; hence, in the case $w = (1+|\cdot|)^a$, if we set $N=4/a$, then $w^N$ is equivalent to a sum-of-squares polynomial, so Proposition \ref{prp:max_nonmax_sq} and Corollary \ref{cor:max_nonmax} apply to $w$. Since $(1+\fstW(z,u))^4 \simeq 1+|z|^4$, a similar remark applies in the case $w=(1+\fstW)^b$.

Recall the definition of $c(k)$ in \eqref{eq:eig}. We set 
\begin{equation}
c_\gamma(k) \defeq c(k+\gamma)  \text{ for } \gamma\in\{-1,0,1\},
\end{equation}
 and define operators $M^\gamma_{\delta,j}$, for $\gamma \in \{-1,0,1\}$,  $j=1,\ldots,J_\delta$ and $f\in\Sch(G)$, by
\begin{equation}\label{eq:mgamma}
\widehat{M^\gamma_{\delta,j}f}(\mu,\alpha,\beta) \defeq \begin{cases}
\chr_{[1-\delta,1]}(c_\gamma(|\alpha|)|\mu|) \, \chr_{[2^j,2^{j+1})}(c_\gamma(|\alpha|))\widehat{f}(\mu,\alpha,\beta) &\text{if } j<J_\delta,\\
\chr_{[1-\delta,1]}(c_\gamma(|\alpha|)|\mu|) \, \chr_{[2^{J_\delta},\infty)}(c_\gamma(|\alpha|))\widehat{f}(\mu,\alpha,\beta) &\text{if } j=J_\delta.
\end{cases}
\end{equation}
Note that $M^0_{\delta,j} = R_{\delta,j} \BRm_\delta(\opL)$.

\begin{prop}\label{prp:fullmdeltaest}
Let $a \in (0,2]$. Suppose that the estimate
\begin{equation}\label{eq:redtraceful}
    \|M^\gamma_{\delta,j}f\|^2_2 \lessapprox C(\delta,j) \|f\|^2_{L^2(w)},
\end{equation}
holds for all $\delta \in \Dyad$, $1\leq j\leq J_\delta$, $\gamma\in\{-1,0,1\}$, in one of the following cases:
\begin{enumerate}[label=(\roman*)]
\item $w=(1+|\cdot|)^a$ and $C(\delta,j) = (2^{-j}\delta)^{a/2}$;
\item $w=(1+\fstW)^a$ and $C(\delta,j) = 2^{-aj}$.
\end{enumerate}
Then the estimates \eqref{eq:nl1} and \eqref{eq:nl2} hold with $N=4/a$.
\end{prop}

The proof will be given at the end of the section, after a number of auxiliary results.

Let $K_{\delta,j}$ to be the convolution kernel of $R_{\delta,j} \BRm_\delta(\opL)$.
Recall that, by \eqref{eq:fourierL},
\begin{equation}\label{eq:kernelformula}
\widehat{K_{\delta,j}}(\mu,k)=\begin{cases} 
      \chr_{[2^j,2^{j+1})}(c(k)) \, \BRm_\delta(|\mu|c(k)) &\text{for } j=1,\ldots,J_\delta-1,\\
       \chr_{[2^{J_\delta},\infty)}(c(k)) \, \BRm_\delta(|\mu|c(k)) &\text{for } j=J_\delta.
   \end{cases} 
\end{equation}

\begin{lem}\label{lem:ptwest76}
Let
\[
H_{\delta,j}(\mu,k) \defeq \begin{cases} \chr_{[2^j,2^j+1)}(c(k)) \, \chr_{[1-\delta,1]}(c(k)|\mu|) &\text{if $k \geq 0$, $j<J_\delta$,}\\
\chr_{[2^{J_\delta},\infty)}(c(k)) \, \chr_{[1-\delta,1]}(c(k)|\mu|) &\text{if $k\geq 0$, $j=J_\delta$,}\\
0 &\text{if $k< 0$.}
\end{cases}
\]
Then, for all $\delta \in \Dyad$, $1\leq j \leq J_\delta$, and for all $\mu \in \RR^n \setminus \{0\}$, $k \in \NN_0$,
\begin{align}
\label{eq:easy_maj}
|\widehat{K_{\delta,j}}(\mu,k)| &\lesssim H_{\delta,j}(\mu,k), \\
\label{eq:ptwest761}
   |\widehat{\fstW^2 K_{\delta,j}}(\mu,k)| &\lesssim 2^{2j} \sum_{\gamma\in\{-1,0,1\}} H_{\delta,j}(\mu,k+\gamma) ,\\
\label{eq:ptwest762}
   |\widehat{\sndW_l K_{\delta,j}}(\mu,k)| &\lesssim \delta^{-1} 2^j \sum_{\gamma\in\{-1,0,1\}} H_{\delta,j}(\mu,k+\gamma) ,
\end{align}
and, if $P$ is any homogeneous first-layer polynomial of degree $1$,
then
\begin{equation}\label{eq:Pmatrixest}
   |\widehat{P K_{\delta,j}}(\mu,\alpha,\beta)| \lesssim_P \begin{cases} 2^j [ H_{\delta,j}(\mu,|\alpha|) + H_{\delta,j}(\mu,|\beta|) ]  &\text{if } |\alpha-\beta|=1, \\
	0 &\text{otherwise.}
	\end{cases}
\end{equation}
\end{lem}
\begin{proof}
The estimate \eqref{eq:easy_maj} is an immediate consequence of \eqref{eq:kernelformula} and \eqref{eq:mdeltabnd}.

As for \eqref{eq:ptwest761}, note that, by \eqref{eq:partzsqrule} and \eqref{eq:easy_maj},
\[
|\widehat{\fstW^2 K_{\delta,j}}(\mu,k)| 
 \lesssim \frac{1+k}{|\mu|} \sum_{\gamma \in \{-1,0,1\}} H_{\delta,j}(\mu,k+\gamma).
\]
In the case $j<J_\delta$, the latter sum vanishes unless $1+k\simeq c(k)\simeq|\mu|^{-1}\simeq 2^j$, and \eqref{eq:ptwest761} follows. In the case $j=J_\delta$, a similar argument works provided $c(k-1) < 2^{J_\delta}$: indeed the sum vanishes unless $c(k+1) \geq 2^{J_\delta}$, so again we deduce $1+k\simeq c(k)\simeq|\mu|^{-1}\simeq 2^{J_\delta}$.

Suppose now that $j=J_\delta$ and $c(k-1) \geq 2^{J_\delta}$. Then, by \eqref{eq:kernelformula} and \eqref{eq:partzsqrule},
\begin{multline}\label{eq:kerneltaylor1}
\widehat{\fstW^2 K_{\delta,J_\delta}}(\mu,k) \\
=\frac{1}{\pi|\mu|} \Bigl[(2k+m) \BRm_\delta(|\mu|c(k)) 
 -k \BRm_\delta(|\mu|c(k-1)) 
 -(k+m) \BRm_\delta(|\mu|c(k+1)) \Bigr].
\end{multline}
On the other hand, by Taylor's Theorem, there exist $\theta_+,\theta_- \in (0,4\pi)$ such that
\begin{equation}\label{eq:taylor}
    \BRm_\delta(|\mu|c(k \pm 1)) 
		= \BRm_\delta(|\mu|c(k)) \pm 4\pi|\mu| \BRm_\delta'(|\mu|c(k)) + (4\pi|\mu|)^2 \BRm_\delta''(|\mu|(c(k)\pm\theta_\pm)).
\end{equation}
Substituting into \eqref{eq:kerneltaylor1} and exploiting cancellations, one easily obtains that
\[
|\widehat{\fstW^2 K_{\delta,J_\delta}}(\mu,k)|
 \lesssim   |\BRm_\delta'(|\mu|c(k))| + (1+k)|\mu| \sum_{\pm} |\BRm_\delta''(|\mu|(c(k)\pm\theta_\pm))|.
\]
By \eqref{eq:mdeltabnd}, the right-hand side vanishes unless $1+k \simeq c(k) \simeq |\mu|^{-1}$, and moreover the first summand is clearly controlled by a multiple of $\delta^{-1} H_{\delta,J_\delta}(\mu,k)$. Moreover, $|\BRm_\delta''(|\mu|(c(k)\pm\theta_\pm))|$ vanishes unless $|\mu|(c(k)\pm\theta_\pm) \in [1-\delta,1]$, which implies
\begin{equation}\label{eq:displ_est}
4\pi|\mu| \leq 4\pi(c(k)\pm\theta_\pm)^{-1} \leq 4\pi(c(k-1))^{-1} \leq 4\pi 2^{-J_\delta} \leq \delta
\end{equation}
by \eqref{eq:jdefn}, and
\[
|\mu| c(k_\pm) \leq 1, \qquad |\mu|c(k_\pm+1) \geq 1-\delta, \qquad |\mu|c(k_\pm+1) - |\mu| c(k_\pm) \leq \delta
\]
(where $k_+ = k$ and $k_- = k-1$), whence $|\mu| c(k) \in [1-\delta,1]$ or $|\mu| c(k\pm1) \in [1-\delta,1]$.
This, together with \eqref{eq:mdeltabnd}, shows that $(1+k) |\mu| |\BRm_\delta''(|\mu|(c(k)\pm\theta_\pm))|$ is controlled by a multiple of $\delta^{-2} [H_{\delta,J_\delta}(\mu,k) + H_{\delta,J_\delta}(\mu,k\pm1)]$. Putting all together, and recalling that $\delta^{-1} \simeq 2^{J_\delta}$, yields \eqref{eq:ptwest761} in this case too.

Let us finally consider \eqref{eq:ptwest762}. From \eqref{eq:parturule}, \eqref{eq:kernelformula} and \eqref{eq:mdeltabnd}, we immediately deduce that
\[\begin{split}
|\widehat{\sndW_l K_{\delta,j}}(\mu,k)|
&\lesssim \left|\frac{\partial}{\partial \mu_l} \widehat{K_{\delta,j}}(\mu,k)\right| + \frac{1+k}{|\mu|} \sum_{\gamma\in\{-1,0,1\}} |\widehat{K_{\delta,j}}(\mu,k+\gamma)| \\
&\lesssim (1+k) \delta^{-1} H_{\delta,j}(\mu,k) + \frac{1+k}{|\mu|} \sum_{\gamma\in\{-1,0,1\}} H_{\delta,j}(\mu,k+\gamma).
\end{split}\]
So, arguing as in the proof of \eqref{eq:ptwest761}, we easily deduce \eqref{eq:ptwest762} in the case $j<J_\delta$, and also in the case $j=J_\delta$ and $c(k-1) < 2^{J_\delta}$.

Suppose now that $j=J_\delta$ and $c(k-1) \geq 2^{J_\delta}$. Then, by \eqref{eq:parturule},
\begin{multline*}
    \widehat{\sndW_l K_{\delta,j}}(\mu,k) 
		= 4\pi(2k+m) \frac{\mu_l}{|\mu|} \BRm_\delta'(|\mu|c(k)) \\
		+\frac{\mu_l}{|\mu|^2} \Bigl[m  \BRm_\delta(|\mu|c(k)) + k  \BRm_\delta(|\mu|c(k-1))- (k+m) \BRm_\delta(|\mu|c(k+1)) \Bigr].
\end{multline*}
By substituting the Taylor expansions \eqref{eq:taylor} and exploiting cancellations, we obtain
\[
| \widehat{\sndW_l K_{\delta,j}}(\mu,k) | \lesssim (1+k)|\mu| \sum_{\pm} |\BRm_\delta''(|\mu|(c(k)\pm\theta_\pm))|,
\]
which, analogously as before, leads to the desired estimate \eqref{eq:ptwest762}.

It remains to consider \eqref{eq:Pmatrixest}. First, note that, by \eqref{eq:partzeta1} and the radiality of $K_{\delta,j}$, we immediately deduce that $\widehat{\fstZ_{\mu,j} K_{\delta,j}}(\mu,\alpha,\beta)$ vanishes unless $\alpha+e_j = \beta$, in which case $\alpha_j+1=\beta_j$ and
\[
\widehat{\fstZ_{\mu,j} K_{\delta,j}}(\mu,\alpha,\beta) \lesssim \left(\frac{1+|\alpha|}{|\mu|}\right)^{1/2} \left| \widehat{K_{\delta,j}}(\mu,|\beta|) - \widehat{K_{\delta,j}}(\mu,|\alpha|) \right|.
\]
Similarly, by \eqref{eq:partzeta2}, $\widehat{\overline{\fstZ_{\mu,j}} K_{\delta,j}}(\mu,\alpha,\beta)$
 vanishes unless $\alpha = \beta+e_j$, in which case an analogous estimate holds.
Hence, by Lemma \ref{lem:matrixcomponentcoordchange}, we deduce that $\widehat{P K_{\delta,j}}(\mu,\alpha,\beta)$ vanishes unless $|\alpha-\beta|=1$, in which case
\[
|\widehat{P K_{\delta,j}}(\mu,\alpha,\beta)| \lesssim_P  \left(\frac{1+|\alpha|}{|\mu|}\right)^{1/2} \left| \widehat{K_{\delta,j}}(\mu,|\beta|) - \widehat{K_{\delta,j}}(\mu,|\alpha|) \right|.
\]
Noting that $|\beta|=|\alpha|\pm 1$ when $|\alpha-\beta|=1$, the right-hand side can be estimated analogously as in the proof of \eqref{eq:ptwest761}, by exploiting, in the case where $j=J_\delta$ and $c(|\alpha|-1) \geq 2^{J_\delta}$, a first-order Taylor expansion in place of \eqref{eq:taylor}.
\end{proof}

\begin{lem}\label{lem:leibniz_cutoff}
For all $\delta \in \Dyad$ and $1 \leq j \leq J_\delta$, the estimate
\begin{multline}\label{eq:leibniz_cutoff}
\|R_{\delta,j} \BRm_\delta(\opL)f\|_{L^2(\omega^4)} \\
 \lesssim \|f\|_{L^2(\omega^4)} + D(\delta,j)^2 \sum_{\gamma \in \{-1,0,1\}} \|M_{\delta,j}^\gamma f\|_{2} + D(\delta,j) \| f \|_{L^2(\omega^2)}
\end{multline}
holds in the following cases:
\begin{enumerate}[label=(\roman*)]
\item $\omega = 1+|\cdot|$ and $D(\delta,j) = (2^j\delta^{-1})^{1/2}$;
\item $\omega = 1+\fstW$ and $D(\delta,j) = 2^j$.
\end{enumerate}
\end{lem}
\begin{proof}
Note that
\begin{align*}
\|R_{\delta,j} \BRm_\delta(\opL)f\|_{L^2((1+\fstW)^4)} 
&\lesssim \|f\|_{2} + \|R_{\delta,j} \BRm_\delta(\opL)f\|_{L^2(\fstW^4)},\\
\|R_{\delta,j} \BRm_\delta(\opL)f\|_{L^2((1+|\cdot|)^4)} 
&\lesssim \|f\|_{2} + \|R_{\delta,j} \BRm_\delta(\opL)f\|_{L^2(\fstW^4)} + \sum_l \|R_{\delta,j} \BRm_\delta(\opL)f\|_{L^2(\sndW_l^2)},
\end{align*}
where we used the $L^2$ boundedness of $R_{\delta,j} \BRm_\delta(\opL)$.

Clearly the term $\|f\|_{2}$ is bounded by $\|f\|_{L^2(\omega^4)}$ in any case.

Now,
\[
\|R_{\delta,j} \BRm_\delta(\opL) f\|_{L^2(\fstW^4)}
= \|\fstW^2 (f * K_{\delta,j})\|_2.
\]
By \eqref{eq:rhorule},
\[
\fstW^2 (f * K_{\delta,j}) 
= f * (\fstW^2 K_{\delta,j}) + (\fstW^2 f) * K_{\delta,j} + \sum_s (P_s f) * (Q_s K_{\delta,j})
\]
for some homogeneous first-layer polynomials $P_s,Q_s$ of degree $1$.
Note that we trivially have
\[
\|(\fstW^2 f) * K_{\delta,j}\|^2_2 \lesssim \|f\|^2_{L^2(\fstW^4)}
\]
since $R_{\delta,j} \BRm_\delta(\opL)$ is bounded on $L^2$.
Next, by \eqref{eq:matrixcompconv} and \eqref{eq:ptwest761},
\[\begin{split}
|(f * (\fstW^2 K_{\delta,j}))\widehat{\;}(\mu,\alpha,\beta)|^2 &= |\widehat{\fstW^2 K_{\delta,j}}(\mu,|\alpha|)|^2 \, |\widehat{f}(\mu,\alpha,\beta)|^2 \\
&\lesssim 2^{4j} \sum_{\gamma \in \{-1,0,1\}} H_{\delta,j}(\mu,|\alpha|+\gamma) \, |\widehat{f}(\mu,\alpha,\beta)|^2 \\
&= 2^{4j} \sum_{\gamma \in \{-1,0,1\}}  |\widehat{M_{\delta,j}^\gamma f}(\mu,\alpha,\beta)|^2,
\end{split}\]
which implies, by \eqref{eq:plancherel}, that
\[
\|f * (\fstW^2 K_{\delta,j})\|_{2} \lesssim 2^{2j} \sum_{\gamma \in \{-1,0,1\}} \|M_{\delta,j}^\gamma f\|_{2}.
\]
Further, 
by \eqref{eq:matrixcompconv} and \eqref{eq:Pmatrixest},
\begin{multline*}
|((P_s f) * (Q_s K_{\delta,j}))\widehat{\;}(\mu,\alpha,\beta)|^2 \\
\lesssim 2^{2j} \sum_{\alpha' \tc |\alpha-\alpha'| = 1}  \Bigl[H_{\delta,h}(\mu,|\alpha|)+H_{\delta,h}(\mu,|\alpha'|) \Bigr] \, |\widehat{P_s f}(\mu,\alpha',\beta)|^2,
\end{multline*}
whence
\[
\sum_\alpha |((P_s f) * (Q_s K_{\delta,j}))\widehat{\;}(\mu,\alpha,\beta)|^2 \lesssim 2^{2j} \sum_{\gamma \in \{-1,0,1\}} \sum_\alpha H_{\delta,h}(\mu,|\alpha|+\gamma) |\widehat{P_s f}(\mu,\alpha,\beta)|^2
\]
and again, by \eqref{eq:plancherel}, we deduce that
\[
   \|(P_s f) * (Q_s K_{\delta,j}) \|_2 \lesssim 2^j \sum_{\gamma \in \{-1,0,1\}}   \| M_{\delta,j}^\gamma (P_s f) \|_{2}.
\]
Combining the above estimates yields
\begin{multline*}
\|R_{\delta,j} \BRm_\delta(\opL)f\|_{L^2(\fstW^4)} \\
 \lesssim \|f\|^2_{L^2(\fstW^4)} + 2^{2j} \sum_{\gamma \in \{-1,0,1\}} \|M_{\delta,j}^\gamma f\|_{2} + 2^j \sum_{\gamma \in \{-1,0,1\}} \sum_s \| M_{\delta,j}^\gamma (P_s f) \|_{2},
\end{multline*}
whence the estimate \eqref{eq:leibniz_cutoff} in the case $\omega = 1+\fstW$ follows, since the $M_{\delta,j}^\gamma$ are uniformly $L^2$-bounded.

Similarly,
\[
\|R_{\delta,j} \BRm_\delta(\opL)f\|_{L^2(\sndW_l^2)} = \|\sndW_l (f * K_{\delta,j}) \|_2
\]
and, by \eqref{eq:psiruleexact},
\[
\sndW_l (f * K_{\delta,j}) = (\sndW_l f) * K_{\delta,j} + f * (\sndW_l K_{\delta,j}) + \sum_s (P_{l,s} f) * (Q_{l,s} K_{\delta,j})
\]
for some homogeneous first-layer polynomials $P_{l,s},Q_{l,s}$ of degree $1$. Arguing as above, and using \eqref{eq:ptwest762} in place of \eqref{eq:ptwest761}, one deduces
\begin{multline*}
\|R_{\delta,j} \BRm_\delta(\opL)f\|_{L^2(\sndW_l^2)} \\
 \lesssim \|f\|_{L^2(\sndW_l^2)} + 2^{j} \delta^{-1} \sum_{\gamma \in \{-1,0,1\}} \|M_{\delta,j}^\gamma f\|_{2} + 2^j \sum_{\gamma \in \{-1,0,1\}} \sum_s \| M_{\delta,j}^\gamma (P_{l,s} f) \|_{2}.
\end{multline*}

Combining all the above estimates, and observing that $2^j \lesssim \delta^{-1}$, we obtain that
\begin{multline*}
\|R_{\delta,j} \BRm_\delta(\opL)f\|_{L^2((1+|\cdot|)^4)}  \lesssim \|f\|^2_{L^2((1+|\cdot|)^4)} \\
+ 2^{j} \delta^{-1} \sum_{\gamma \in \{-1,0,1\}} \|M_{\delta,j}^\gamma f\|_{2} + (2^j \delta^{-1})^{1/2} \sum_{\gamma \in \{-1,0,1\}} \sum_s \| M_{\delta,j}^\gamma (\tilde P_{s} f) \|_{2}
\end{multline*}
for some homogeneous first-layer polynomials $\tilde P_s$ of degree $1$. The estimate \eqref{eq:leibniz_cutoff} in the case $\omega=1+|\cdot|$ again follows since the $M_{\delta,j}^\gamma$ are uniformly $L^2$-bounded.
\end{proof}

\begin{proof}[Proof of Proposition \ref{prp:fullmdeltaest}.]
Clearly, \eqref{eq:redtraceful} with $\gamma=0$ implies \eqref{eq:nl1}.

As for \eqref{eq:nl2}, noting that $w^N \simeq \omega^4$ (where $\omega$ is $1+|\cdot|$ or $1+\fstW$ as appropriate), if we combine Lemma \ref{lem:leibniz_cutoff} and the estimate \eqref{eq:redtraceful}, then we deduce
\[
\|R_{\delta,j} \BRm_\delta(\opL)f\|_{L^2(w^N)}^2 
 \lesssim \|f\|_{L^2(w^N)}^2 + D(\delta,j)^{4-a} \|f\|_{L^2(w)}^2 + D(\delta,j)^2 \| f \|_{L^2(w^{N/2})}^2,
\]
where $D(\delta,j)$ is $(2^j \delta^{-1})^{1/2}$ or $2^j$ as appropriate, so that $C(\delta,j)=D(\delta,j)^{-a}$.
To complete the proof of \eqref{eq:nl2}, it is enough to show that the last summand in the right-hand side is controlled by the other two. However, this is clear in the case $a=2$, since $N/2=1$ and $4-a=2$ in that case. Otherwise, let 
\[
S \defeq \{(z,u)\in G: w(z,u)^{N/2} \leq D(\delta,j)^2\};
\]
since $N/2 = 2/a > 1$, it is then easy to check that
\[
D(\delta,j)^2 \| \chr_S f \|_{L^2(w^{N/2})}^2 \leq D(\delta,j)^{4-a} \|f\|_{L^2(w)}^2,
\]
while 
\[
D(\delta,j)^2 \| \chr_{G\setminus S} f \|_{L^2(w^{N/2})}^2 \leq  \|f\|_{L^2(w^N)}^2 ,
\]
and we are done.
\end{proof}

\section{The dual trace lemmas}\label{s:chaptrace}

Recall from \eqref{eq:mgamma} the definition of the operators $M^\gamma_{\delta,j}$.
The main results of this section are the following `dual trace lemmas'.

\begin{thm}\label{thm:tracelemfinal}
For all $\delta \in \Dyad$, $1\leq j \leq J_\delta$, $\gamma\in\{-1,0,1\}$ and $a\in [0,\tfrac{2}{3}]$,
\begin{equation}\label{eq:tracelemfinal}
\|M^\gamma_{\delta,j}f\|_2^2\lessapprox (2^{-j}\delta)^{a/2} \|f\|_{L^2((1+|\cdot|)^{a})}^2.
\end{equation}
\end{thm}

\begin{thm}\label{thm:trace5}
For all $\delta \in \Dyad$, $1\leq j\leq J_\delta$, $\gamma\in\{0,1,-1\}$ and $a\in[0,1]$, 
\begin{equation}\label{eq:trace5}
\|M^\gamma_{\delta,j}f\|^2_2\lessapprox (2^{-j})^a \|f\|^2_{L^2((1+\fstW)^a)}.
\end{equation}
\end{thm}

It should be observed that, in the case $j = J_\delta$, the constants $(2^{-j}\delta)^{a/2}$ and $(2^{-j})^a$ in the right-hand sides of \eqref{eq:tracelemfinal} and \eqref{eq:trace5} are comparable (since $2^{-J_\delta} \simeq \delta$), so Theorem \ref{thm:trace5} gives a stronger estimate in this case.
In the case $j < J_\delta$, instead, the two results are not comparable, and Theorem \ref{thm:tracelemfinal} requires an independent proof. In both cases, the proof strategy will be based on the following conditional result.

Here and henceforth, $K_{\delta,j}^\gamma$ denotes the convolution kernel of the operator $M_{\delta,j}^\gamma$. Moreover, a function on $G$ is said to be \emph{$G$-homogeneous} if it is homogeneous with respect to the automorphic dilations \eqref{eq:aut_dil}.

\begin{prop}\label{prp:new_nruest}
Let $a\in (0,\infty)$.
Let $\omega$ be a $G$-homogeneous weight of degree $1$, which is a fractional power of a nonnegative polynomial.
Suppose that, for all $\theta\in \RR$, all $\delta \in \Dyad$, all integers $1\leq j \leq J_\delta$, all $\gamma\in\{-1,0,1\}$ and all compact $I\subseteq \RR^+$,
\begin{equation}\label{eq:ourassumption2}
    \sup_{\substack{\mu,k\\c(k)|\mu|\in I}} |e^{-\theta^2} \partial_{\omega^{-a/2+i\theta}} \widehat{K_{\delta,j}^\gamma}(\mu,k)|
		\lessapprox_I C(\delta,j),
\end{equation}
where the implicit constant does not depend on $\theta$, and $C(\delta,j) \gtrsim \delta^\kappa$ for some $\kappa \geq 0$.
Then, for all $\delta \in \Dyad$, all integers $1\leq j \leq J_\delta$, and all $\gamma\in\{-1,0,1\}$,
\begin{equation}\label{eq:conditional_trace}
\|M^\gamma_{\delta,j}f\|_2^2
\lessapprox C(\delta,j) \, \|f\|_{L^2((1+|\cdot|)^a)}^2;
\end{equation}
moreover, if $\omega$ is first-layer (i.e., depends only on $z$), then
\begin{equation}\label{eq:conditional_trace_fst}
\|M^\gamma_{\delta,j}f\|_2^2
\lessapprox C(\delta,j) \, \|f\|_{L^2((1+\fstW)^a)}^2.
\end{equation}
\end{prop}

In light of this result, the proof of our `trace lemmas' boils down to establishing the estimate \eqref{eq:ourassumption2} for an appropriate choice of the weight $\omega$. 
More precisely, for Theorem \ref{thm:trace5} we take $\omega = \fstW$, while in the case of Theorem \ref{thm:tracelemfinal} (and $j<J_\delta$) we take $\omega = \sndW^{1/2}$.
In the proof of the various instances of \eqref{eq:ourassumption2}, a crucial role is played by the explicit formulas for `dual fractional integral operators' obtained in Section \ref{s:fracint}, as well as the estimates for Jacobi polynomials discussed in Section \ref{s:estforjacobipoly}.

It should be noticed that, starting from the kernel estimate \eqref{eq:ourassumption2} with $\theta = 0$, the `trace estimates' \eqref{eq:conditional_trace} and \eqref{eq:conditional_trace_fst} could be directly derived using duality and Schur's Test (cf.\ \cite[proof of Lemma 3]{CRV}), provided one restricted to the class of radial functions $f$. Indeed, from the estimate in Lemma \ref{lem:nonradialtraceu1} one could derive the following sharpened version of Theorem \ref{thm:tracelemfinal}, that only involves second layer weights, but applies to radial functions only.

\begin{prop}
Assume that $f$ is radial.
For all $\delta \in \Dyad$, $1\leq j < J_\delta$, $\gamma\in\{-1,0,1\}$ and $a\in [0,\tfrac{2}{3}]$,
\begin{equation}
\|M^\gamma_{\delta,j}f\|_2^2
\lessapprox (2^{-j}\delta)^{a/2}\|f\|_{L^2((1+\sndW)^{a/2})}^2.
\end{equation}
\end{prop}

For general functions $f$, however, the direct approach through Schur's Test appears not to be enough; similarly to \cite[proof of Lemma 7]{GM}, the proof of Proposition \ref{prp:new_nruest} exploits a more delicate complex interpolation argument (requiring the estimate \eqref{eq:ourassumption2} for arbitrary $\theta \in \RR$), combined with the Leibniz rules of Section \ref{s:secleibniz}. The fact that the Leibniz rule \eqref{eq:psiruleexact} for second-layer polynomials produces first-layer polynomials as well explains why the final trace estimate \eqref{eq:conditional_trace} contains the `full weight' $|\cdot|$, despite being derived from a kernel estimate (Lemma \ref{lem:nonradialtraceu1}) involving a second-layer weight only.

Before discussing the proofs of the `trace lemmas', Theorems \ref{thm:tracelemfinal} and \ref{thm:trace5},  and the conditional result Proposition \ref{prp:new_nruest}, we shall prove a small lemma that will be of use in what follows.

\begin{lem}\label{lem:ckratio}
Let $k\in\ZZ$ and $x\in\NN_0$. If $c(k)> 0$, then $\tfrac{c(k+x)}{c(k)}\in[1,1+2x]$. If additionally $c(k-x)>0$, then $\tfrac{c(k-x)}{c(k)}\in[\tfrac{1}{1+2x},1]$.
\end{lem}
\begin{proof}
Recall the definition of $c(k)$ in \eqref{eq:eig}. For the first inclusion, since $c(k)>0$ then $c(k)\geq 2\pi$, so
\begin{equation}
    1\leq \frac{c(k+x)}{c(k)}= 1+\frac{4\pi x}{c(k)}=1+\frac{4\pi x}{2\pi}\leq 1+2x.
\end{equation}
If $c(k-x)>0$, then let $l \defeq k-x$. Then $c(l)>0$, so applying the first result of this Lemma gives
\begin{equation}
    \frac{c(k)}{c(k-x)}=\frac{c(l+x)}{c(l)}\in[1,1+2x],
\end{equation}
which gives the second result.
\end{proof}

\subsection{The conditional result}\label{s:trace_conditional}
In this section we prove Proposition \ref{prp:new_nruest}.

Let $\varphi,\varphi_0\in C_c^\infty(\RR)$ be such that $\supp(\varphi)\subseteq(1,3)$ and 
\begin{equation}\label{eq:9.41}
1 = \sum_{k\in\NN_0}\varphi_k^2(t) \text{ for } t>0, \text{ where } \varphi_k(t) \defeq \varphi(2^{-k} t) \text{ for } k\geq 1.  
\end{equation}
For all $r \in \NN_0$, define the cut-off operators $\Lambda_r$ and $\tilde\Lambda_r$ by
\begin{equation}\label{eq:Lambda_cutoff}
\Lambda_r f(z,u) \defeq \varphi_r(|(z,u)|) \, f(z,u), \qquad \tilde\Lambda_r f(z,u) \defeq \varphi_r(|z|)\,  f(z,u)
\end{equation}
for all functions $f : G \to \CC$. We first prove an auxiliary estimate.

\begin{lem}\label{lem:nruest}
Under the assumptions of Proposition \ref{prp:new_nruest},
for all $\Psi \in C^\infty_c(\RR^+)$, all $R,\theta\in \RR$ with $R\geq 0$, all $\delta \in \Dyad$, all integers $1\leq j \leq J_\delta$, all $\gamma\in\{-1,0,1\}$, and all $r \in \NN_0$,
\begin{equation}\label{eq:radialtraceuest20}
    |e^{-\theta^2} \langle \omega^{R-a/2+i\theta} K_{\delta,j}^\gamma , \Psi(\opL)[(\Lambda_r f)^* * (\Lambda_r f)] \rangle|
		\lessapprox_{\Psi,R} C(\delta,j) \, \|2^{Rr} f\|^2_2,
\end{equation}
where
the implicit constant does not depend on $\theta$. In addition, if $\omega$ is first-layer, then the estimate \eqref{eq:radialtraceuest20} also holds with $\Lambda_r$ replaced by $\tilde \Lambda_r$.
\end{lem}

In the proof, for a $G$-homogeneous polynomial $P$, we denote by $\hdeg P$ its homogeneity degree with respect to the dilations \eqref{eq:aut_dil}.

\begin{proof}
Let $d \in \NN$ be such that $\omega^d$ is a polynomial.
By complex interpolation (i.e., Hadamard's three-line theorem) it is enough to consider the case where $R = dN$ for some $N \in \NN_0$. Let $K$ denote the convolution kernel of $\Psi(\opL)$. Now, by \eqref{eq:convl2es1},
\[\begin{split}
&|\langle \omega^{dN-a/2+i\theta} K_{\delta,j}^\gamma , \Psi(\opL)[(\Lambda_r f)^* * (\Lambda_r f) ] \rangle| \\
&= \left|\int_{\RR^n}\sum_{\alpha\in\NN_0^m} \partial_{\omega^{-a/2+i\theta}} \widehat{K_{\delta,j}^\gamma}(\mu,|\alpha|) \, \overline{[\omega^{dN} \Psi(\opL)[(\Lambda_r f)^* * (\Lambda_r f)]]\widehat{\phantom{;}}(\mu,\alpha,\alpha)} \,|\mu|^m \,d\mu\right| .
\end{split}\]
By Lemma \ref{lem:support_poly}\ref{en:support_poly_mu}, $[\omega^{dN} \Psi(\opL)[(\Lambda_r f)^* * (\Lambda_r f)]]\widehat{\phantom{;}}(\mu,\alpha,\alpha) \neq 0$ only if $|\mu|c(k) \in \supp\Psi$ for some $k \in \NN_0$ such that $|k-|\alpha|| \leq 2N$, which implies by Lemma \ref{lem:ckratio} that $|\mu| c(|\alpha|) \in I \defeq [(1+4N)^{-1} \min \supp \Psi, (1+4N) \max \supp\Psi]$. We can then apply \eqref{eq:ourassumption2} to deduce that
\begin{equation}\label{eq:w_est}
\begin{split}
&|e^{-\theta^2} \langle \omega^{dN-a/2+i\theta} K_{\delta,j}^\gamma , \Psi(\opL)[(\Lambda_r f)^* * (\Lambda_r f) ] \rangle| \\
&\lessapprox_N C(\delta,j) \int_{\RR^n} \sum_{\alpha\in\NN_0^m} |[\omega^{dN} \Psi(\opL)[(\Lambda_r f)^* * (\Lambda_r f)]] \widehat{\phantom{;}} (\mu,\alpha,\alpha)| \,|\mu|^m \,d\mu .
\end{split}
\end{equation}
Let $K_\Psi$ be the convolution kernel of $\Psi(\opL)$. Then, by iteratively applying the Leibniz rules from Section \ref{s:secleibniz},
\[\begin{split}
\omega^{dN} \Psi(\opL)[(\Lambda_r f)^* * (\Lambda_r f)] 
&= \omega^{dN} [(\Lambda_r f)^* * (\Lambda_r f) * K_\Psi] \\
&= \sum_l [(P_{l,1} \Lambda_r f)^* * (P_{l,2} \Lambda_r f) * (P_{l,3} K_\Psi)],
\end{split}\]
where $P_{l,1},P_{l,2},P_{l,3}$ are $G$-homogeneous polynomials on $G$ with $\sum_{s=1}^3 \hdeg P_{l,s} = dN$.
From \eqref{eq:w_est}, \eqref{eq:convl2es} and Young's convolution inequality we then deduce that
\[
\begin{split}
&|e^{-\theta^2} \langle \omega^{dN-a/2+i\theta} K_{\delta,j}^\gamma , \Psi(\opL)[(\Lambda_r f)^* * (\Lambda_r f) ] \rangle| \\
&\lessapprox_N C(\delta,j) \sum_l \| P_{l,1} \Lambda_r f\|_2 \| P_{l,2} \Lambda_r f\|_2 \| P_{l,3} K_\Psi \|_1 \\
&\lesssim_{\Psi,N} C(\delta,j) \sum_l 2^{r(\hdeg P_{l,1} + \hdeg P_{l,2})} \| f\|_2^2 \\
&\lesssim_{N} C(\delta,j) \|2^{2Nr} f\|_2^2,
\end{split}
\]
where we used that $K_\Psi \in \Sch(G)$ \cite{17},
that $|P_{l,s}| \lesssim |\cdot|^{\hdeg P_{l,s}} \lesssim 2^{r\hdeg P_{l,s}}$ on the support of $\Lambda_r f$, and that $\hdeg P_{l,1} + \hdeg P_{l,2} \leq dN$.

If $\omega$ is first-layer, then essentially the same proof works with $\Lambda_r$ replaced by $\tilde \Lambda_r$. In this case, the polynomials $P_{l,s}$ given by the Leibniz rules are first-layer as well, whence $|P_{l,s}| \lesssim \fstW^{\hdeg P_{l,s}} \lesssim 2^{r\hdeg P_{l,s}}$ on the support of $\tilde\Lambda_r f$.
\end{proof}

\begin{proof}[Proof of Proposition \ref{prp:new_nruest}]
Choose $\Psi\in C_c^\infty(\RR^+)$ such that $\Psi(x)=1$ for $x\in[\tfrac{1}{6},3]$.
From \eqref{eq:mgamma} and Lemma \ref{lem:ckratio} it is clear that $M_{\delta,j}^\gamma$ is an orthogonal projection and $\Psi(\opL) M_{\delta,j}^\gamma = M_{\delta,j}^\gamma$. 
Hence, by \eqref{eq:9.41},
\[\begin{split}
    \|M_{\delta,j}^\gamma f\|_2
		&\leq \sum_{r=0}^\infty \|M_{\delta,j}^\gamma \Lambda_r^2 f \|_2\\
		&= \sum_{r=0}^\infty \langle \Psi(\opL) M_{\delta,j}^\gamma \Lambda_r^2 f, \Lambda_r^2 f \rangle^{1/2} \\
		&= \sum_{r=0}^\infty \langle K_{\delta,j}^\gamma, \Psi(\opL) [(\Lambda_r^2 f)^* * \Lambda_r^2 f] \rangle^{1/2}.
\end{split}\]
We now apply \eqref{eq:radialtraceuest20} with $R = a/2$ and $\theta=0$ and the Cauchy--Schwarz inequality to obtain that
\[\begin{split}
    \|M_{\delta,j}^\gamma f\|_2
		&\lessapprox C(\delta,j)^{1/2} \sum_{r=0}^\infty2^{ra/2}\| \Lambda_r f\|_2\\
		&\leq C(\delta,j)^{1/2} \left(\sum_{r=0}^\infty 2^{(a+\epsilon)r/2}\| \Lambda_r f\|_2^2\right)^{1/2}\left(\sum_{r=0}^\infty 2^{-\epsilon r}\right)^{1/2}\\
		&\simeq_\epsilon C(\delta,j)^{1/2} \, \|f\|_{L^2((1+|\cdot|)^{a+\epsilon})}
\end{split}\]
for all $\epsilon > 0$. Since $C(\delta,j) \gtrsim \delta^\kappa$ for some $\kappa \geq 0$, interpolation with the trivial $L^2$-estimate for $M_{\delta,j}^\gamma$ completes the proof.

In the case $\omega$ is first-layer, a similar argument works with $\tilde\Lambda_r$ in place of $\Lambda_r$. In this case, one exploits the fact that $\sum_{r=0}^\infty 2^{(a+\epsilon)r/2}\| \tilde\Lambda_r f\|_2^2 \simeq \|f\|_{L^2((1+\fstW)^{a+\epsilon})}^2$.
\end{proof}

\subsection{The first-layer trace lemma}\label{s:trace3proof}

In this section we prove Theorem \ref{thm:trace5}, which also implies the case $j=J_\delta$ of Theorem \ref{thm:tracelemfinal}.

Recall that $K_{\delta,j}^\gamma$ denotes the convolution kernel of $M_{\delta,j}^\gamma$. From \eqref{eq:mgamma} it is clear that, if $j<J_\delta$, then
\begin{equation}\label{eq:kgamma}
\widehat{K_{\delta,j}^\gamma}(\mu,k) 
= \chr_{[1-\delta,1]}(c_\gamma(k)|\mu|) \, \chr_{[2^j,2^{j+1})}(c_\gamma(k)).
\end{equation}

First, the following estimate will be useful.

\begin{lem}\label{lem:nrstirlest}
Let $m\in\NN$ and let $a\in [1,2m]$. Then, for all $x\in\NN_0$,
\[
    \sum_{p=0}^x (1+x-p)^{a-2} (1+p)^{m-a/2-1}
		\lesssim_{m,a}
		\begin{cases}
		(1+x)^{m+a/2-2} &\text{if } a\neq 1,2m,\\
		(1+x)^{m+a/2-2}\log(2+x)  &\text{otherwise}.
		\end{cases}
\]
\end{lem}
\begin{proof}
The case of $x=0$ is trivial, so in what follows we assume that $x>0$, and consequently $x+1\simeq x$.

Set $f(p) = (1+x-p)^{a-2} (1+p)^{m-a/2-1}$. Then
\[
f'(p)/f(p) = (2-a) (1+x-p)^{-1} + (m-a/2-1) (1+p)^{-1},
\]
whence
\[
|f'(p)/f(p)| \lesssim_{a,m} 1
\]
for $p \in [0,x]$, uniformly in $x$. Hence, by \cite[Lemma 4.1]{51},
\[
\begin{split}
&\sum_{p=0}^x (1+x-p)^{a-2} (1+p)^{m-a/2-1}\\
&\lesssim_{a,m} \int_0^x (1+x-p)^{a-2} (1+p)^{m-a/2-1} \,dp \\
&\lesssim_{a,m} x^{m+a/2-2} \left[ \int_0^{1/2} (1/x+s)^{a-2} \,ds + \int_0^{1/2} (1/x+s)^{m-a/2-1} \,ds \right] \\
&\leq x^{m+a/2-2} \left[ \int_{1/x}^{3/2} s^{a-2} \,ds + \int_{1/x}^{3/2} s^{m-a/2-1} \,ds \right],
\end{split}
\]
since $0 < 1/x \leq 1$. In the case $a \in (1,2m)$, both $a-2>-1$ and $m-a/2-1>-1$, so both integrals in the last line are bounded uniformly in $x$, and we are done. In the case $a=1$ or $a=2m$,
one of the two exponents is equal to $-1$, so the corresponding integral is bounded by a multiple of $1 + \log(x)$, and again we are done.
\end{proof}

As before, let $K_{\delta,j}^\gamma$ denote the convolution kernel of $M_{\delta,j}^\gamma$, given by \eqref{eq:kgamma}. We now establish the estimate \eqref{eq:ourassumption2} in the case $\omega = \fstW$.

\begin{lem}\label{nonradialtracez1}
Let $I\subseteq \RR^+$ be compact.
Let $a\in\CC$ with $1 < \Re(a) < 2m$. Then, for all $\delta \in \Dyad$ and $1 \leq j \leq J_\delta$,
\[
    \sup_{\substack{k,\mu\\c(k)|\mu|\in I}} |e^{a^2} \partial_{\fstW^{-a}} \widehat{K_{\delta,j}^\gamma}(\mu,k)| 
		\lesssim_{I,\Re(a)} 
 2^{-j} .
\]
The estimate
also holds for $\Re(a)=1$ and $j<J_\delta$ if we replace $\lesssim$ with $\lessapprox$.
\end{lem}
\begin{proof}
From Lemmas \ref{lem:radialkernelz}
and \ref{lem:laguerreintegrallem}\ref{en:laguerreintegrallem1} we easily deduce that
\begin{multline}
K_{\fstW^{-a}}(\nu,l;\mu,k) 
= \frac{C_{a,m}}{(\Gamma(a/2))^2} \frac{\delta(\nu-\mu) \, |\nu|^{a/2-m}}{\binom{k+m-1}{k}\binom{l+m-1}{l}} \\
\times \sum_{p=0}^{\min\{k,l\}} \frac{\Gamma(a/2+k-p)}{(k-p)!} \frac{\Gamma(a/2+l-p)}{(l-p)!} \frac{\Gamma(p+m-a/2)}{p!},
\end{multline}
where $C_{a,m} = \frac{\pi^{a/2}}{(m-1)!}$.

Note that $|C_{a,m}| \lesssim_{\Re(a)} 1$. Moreover, by \cite[eq.\ 5.11.9]{27},
\begin{equation}\label{gam1}
    \left|\frac{e^{a^2}}{(\Gamma(a/2))^2}\right| \lesssim_{\Re(a)} e^{-(\Im(a))^2} e^{\pi \Im(a)} \lesssim 1.
\end{equation}
In particular, in view of \eqref{eq:integralkernel} and \eqref{eq:kgamma}, 
\[
|e^{a^2} \partial_{\fstW^{-a}} \widehat{K_{\delta,j}^\gamma}(\mu,k)| \lesssim_{\Re(a)}
A^{a,\gamma}_{\delta,j}(\mu,k),
\]
where, for $j<J_\delta$,
\begin{multline}\label{nonradialest1.5}
A^{a,\gamma}_{\delta,j}(\mu,k) \defeq 
\Biggl| \frac{|\mu|^{a/2}}{\binom{k+m-1}{k}}
\sum_{\substack{l \in \NN_0 \\ c_\gamma(l)\in [2^j,2^{j+1})}}  \chr_{[1-\delta,1]}(c_\gamma(l)|\mu|)  \\
\times \sum_{p=0}^{\min\{k,l\}} \frac{\Gamma(a/2+k-p)}{(k-p)!} \frac{\Gamma(a/2+l-p)}{(l-p)!} \frac{\Gamma(p+m-a/2)}{p!} \Biggr| ,
\end{multline}
while, if $j=J_\delta$, then the sum in $l$ is to be restricted to $c_\gamma(l) \in [2^{J_\delta},\infty)$ instead. In any case, the condition $c_\gamma(l) \geq 2^j$ is required.

Note that the conditions $c(k)|\mu|\in I$, $c_\gamma(l) |\mu| \in [1-\delta,1]$ and $c_\gamma(l) \geq 2^j$ imply that
\[
1+k \simeq_I |\mu|^{-1} \simeq 1+l \gtrsim 2^j.
\]
Hence, if $l,l' \in \NN_0$ satisfy $c_\gamma(l)|\mu|, c_\gamma(l') |\mu| \in [1-\delta,1]$,
then
\[
4\pi |l - l'| = |c_\gamma(l)-c_\gamma(l')| \leq \delta/|\mu|;
\]
in other words, for every fixed $\mu$, the number of the $l \in \NN_0$ satisfying $c_\gamma(l) |\mu| \in [1-\delta,1]$ and $c_\gamma(l) \geq 2^j$ is $\lesssim_I 1+ \delta/|\mu| \lesssim 2^{-j}/|\mu|$ (here we use that $\delta,|\mu| \lesssim 2^{-j}$).
In addition,
\[
\left|\frac{|\mu|^{a/2}}{\binom{k+m-1}{k}}\right|
\simeq |\mu|^{\Re(a/2)}(1+k)^{1-m}
\simeq_{I,\Re(a)} |\mu|^{\Re(a/2)+m-1}.
\]
Furthermore,
by \cite[eqs.\ 5.6.6 and 5.11.12]{27}, for all $h,p \in \NN_0$ with $p \leq h$,
\begin{align*}
    \left|\frac{\Gamma(a/2+h-p)}{(h-p)!}\right| 
		&\lesssim_{\Re(a)} (1+h-p)^{\Re(a/2)-1},\\
    \left|\frac{\Gamma(p+m-a/2)}{p!}\right| 
		&\lesssim_{\Re(a)} (1+p)^{m-\Re(a/2)-1}.
\end{align*}
Hence
\begin{equation*}
\begin{split}
&\sup_{\substack{k \\ c(k)|\mu|\in I}} A^{a,\gamma}_{\delta,j}(\mu,k) \\
&\lesssim_{I,\Re(a)} 2^{-j}\, |\mu|^{\Re(a/2)+m-2} 
\sup_{\substack{k,l \in \NN_0 \\ 1+k\simeq_I |\mu|^{-1} \simeq 1+l}} 
\sum_{p=0}^{\min\{k,l\}}
(1+p)^{m-\Re(a/2)-1}  \\
&\qquad\qquad\times (1+k-p)^{\Re(a/2)-1} (1+l-p)^{\Re(a/2)-1} \\
&\lesssim_{I,\Re(a)} 2^{-j} \, |\mu|^{\Re(a/2)+m-2} \sup_{\substack{h \in \NN_0 \\ 1+h\simeq_I |\mu|^{-1} }} \sum_{p=0}^{h} (1+h-p)^{\Re(a)-2} (1+p)^{m-\Re(a/2)-1},
\end{split}
\end{equation*}
since $\max\{l,k\} \simeq_I \min\{l,k\}$, and the desired estimate follows from Lemma \ref{lem:nrstirlest} (in the case $\Re(a)=1$ and $j<J_\delta$ one also uses that $|\mu| \simeq 2^{-j}$, which follows from the conditions $c_\gamma(l)|\mu| \in [1-\delta,1]$ and $c_\gamma(l) \in [2^j,2^{j+1})$ in \eqref{nonradialest1.5}).
\end{proof}

\begin{proof}[Proof of Theorem \ref{thm:trace5}]
By Lemma \ref{nonradialtracez1},
the assumptions of Proposition \ref{prp:new_nruest} are satisfied with $\omega=\fstW$, $C(\delta,j) = 2^{-j}$ (note that $2^{j} \lesssim \delta^{-1}$) and $a = 1+\epsilon$ for any $\epsilon>0$, so we get the estimate
\[
\|M^\gamma_{\delta,j} f\|^2_2 \lesssim_\epsilon 2^{-j} \|f\|^2_{L^2((1+\fstW)^{1+\epsilon})},
\]
and interpolation with the trivial $L^2$ bound for $M^\gamma_{\delta,j}$ gives the result.
\end{proof}

\subsection{The second-layer trace lemma}\label{s:trace1proof}

In this section we complete the proof of Theorem \ref{thm:tracelemfinal}, by treating the missing case $j<J_\delta$.

As already mentioned, our proof will be based on establishing the estimate \eqref{eq:ourassumption2} in the case where $\omega = \sndW$.
We first obtain a preliminary estimate, which should be compared with those obtained in the proof of \cite[Lemma 3]{CRV}.

\begin{lem}\label{lem:fsest}
Define, for all $\beta \in \RR$ and $s \in \RR^+$,
\begin{equation}
F_\beta(s) \defeq \int_{S^{n-1}}\frac{1}{|(s,0,\ldots,0)-\sigma|^{n-\beta}} \,d\sigma,
\end{equation}
where integration is with respect to the surface measure on $S^{n-1}$.
Then
\begin{equation}\label{eq:Fbetas_estimate}
F_\beta(s)\simeq\begin{cases}
    (1+s)^{1-n}|1-s|^{\beta-1} &\text{for } \beta<1,\\
    (1+(n-1)\log_+\frac{1}{|1-s|})(1+s)^{1-n} &\text{for } \beta=1,\\
    (1+s)^{\beta-n} &\text{for } \beta>1.
\end{cases}
\end{equation}
\end{lem}
\begin{proof}
If $n=1$, then $F_\beta(s) = |1+s|^{\beta-1} + |1-s|^{\beta-1}$ and the estimate is clear.

Assume now that $n \geq 2$.
By using polar coordinates, it is easily seen that
\[
F_\beta(s)
\simeq \int_0^\pi \frac{\sin^{n-2}\theta}{\big||1-s|+\min\{1,s\}\theta\big|^{n-\beta}} \,d\theta
\simeq \int_0^1 \frac{t^{n-2}}{\big||1-s|+\min\{1,s\}t\big|^{n-\beta}} \,dt.
\]

First, suppose 
$s \in \RR^+ \setminus (\tfrac{1}{2},\tfrac{3}{2})$.
In this case $|1-s| \gtrsim 1 \gtrsim \min\{1,s\}t$,
so
\begin{equation}\label{eq:betterc1}
F_\beta(s) \simeq |1-s|^{\beta-n} \int_0^1 t^{n-2} \,dt \simeq |1-s|^{\beta-n}.
\end{equation}
Since $|1-s| \simeq 1+s$, this proves \eqref{eq:Fbetas_estimate} in this case.

Now, we assume $\tfrac{1}{2}<s<\tfrac{3}{2}$. In this case, $|1-s|\leq\tfrac{1}{2}$ and $\min\{1,s\}\simeq 1$, whence
\begin{equation}\label{eq:betterc2}
\begin{split}
F_\beta(s)
&\simeq \int_0^1 \frac{t^{n-2}}{\big||1-s|+t\big|^{n-\beta}} \,dt\\
&\simeq |1-s|^{\beta-n} \int_0^{|1-s|} t^{n-2} \,dt + \int_{|1-s|}^{1} t^{\beta-2} \,dt\\
&\simeq 
\begin{cases}
    |1-s|^{\beta-1} & \text{for } \beta<1,\\
    \log\tfrac{1}{|1-s|} & \text{for } \beta=1,\\
    1 & \text{for } \beta>1.
\end{cases}
\end{split}
\end{equation}
Since $1+s \simeq 1$, this again matches \eqref{eq:Fbetas_estimate}.
\end{proof}

In the next result, we assume that $m>1$, 
 due to a technical constraint on one of the estimates for Jacobi polynomials we will use (Corollary \ref{cor:jacobirefinedest3}).
However, if $m=1$, then $G$ is isomorphic to the first Heisenberg group $H_1$, so 
this case is effectively already covered by \cite{GM}.

\begin{lem}\label{lem:nonradialtraceu1}
Assume that $m>1$.
Let $I \subseteq (0,\infty)$ be compact.
For all $a\in\CC$ with $\Re(a)\in (0,\min\{2,n\}) \setminus \{\tfrac{2}{3}\}$, all $\delta \in \Dyad$ and all $j<J_\delta$,
\begin{equation}\label{eq:nrtu1aeq}
  \sup_{\substack{\mu,k\\c(k)|\mu|\in I}} |\partial_{\sndW^{-a/2}} \widehat{K_{\delta,j}^\gamma}(\mu,k)|
		\lesssim_{I,\Re(a)} 
		\begin{cases}
		({2^{-j}\delta})^{\Re(a)/2} &\text{if } j\leq \frac{3J_\delta(2-\Re(a))}{4}, \\
		(2^{-j})^{\Re(a)/2-2/3}\delta &\text{if } j\geq \frac{3J_\delta(2-\Re(a))}{4}.
		\end{cases}
\end{equation}
The estimate
\eqref{eq:nrtu1aeq} 
also holds for $\Re(a)=\tfrac{2}{3}$ if we replace $\lesssim$ with $\lessapprox$.
\end{lem}
\begin{proof}
For $\omega(|z|,u)=w(u)=|u|^{-a/2}$, where $0<\Re(a)<n$, recall that $\widehat{w}(\mu) = C_{n,a} |\mu|^{a/2-n}$, where
\[
C_{n,a} = \pi^{(a-n)/2} \Gamma(n/2-a/4)/\Gamma(a/4)
\]
\cite[\S V.1, Lemma 2]{9}. Hence, by \eqref{eq:heatkernelu1},
\begin{multline}\label{eq:uheatker}
|K_{\sndW^{-a/2}}(\nu,l;\mu,k)| \\
= |C_{n,a}| \frac{|\nu-\mu|^{\Re(a)/2-n}}{\binom{\min\{k,l\}+m-1}{m-1}(|\nu|+|\mu|)^m}
\left|\tfrac{|\mu|-|\nu|}{|\mu|+|\nu|}\right|^{|k-l|}
\left| P_{\min\{k,l\}}^{(|k-l|,m-1)}\left(1-2\left(\tfrac{|\mu|-|\nu|}{|\mu|+|\nu|}\right)^2\right) \right|,
\end{multline}
where
\[
|C_{n,a}| = \pi^{(\Re(a)-n)/2} |\Gamma(a/4-n/2)/\Gamma(a/4)| \lesssim_{\Re(a)} 1
\]
\cite[eq.\ (5.11.12)]{27}.
Thus, in view of \eqref{eq:kgamma}, we are required to estimate
\[\begin{split}
\mathscr{I} \defeq &\sup_{\substack{\mu,k\\c(k)|\mu|\in I}}
\int_{\RR^n} \sum_{l\in\NN_0} \frac{|\nu-\mu|^{\Re(a)/2-n}}{\binom{\min\{k,l\}+m-1}{\min\{k,l\}}(|\nu|+|\mu|)^m}
\left|\tfrac{|\mu|-|\nu|}{|\mu|+|\nu|}\right|^{|k-l|} \\
&\times \left| P_{\min\{k,l\}}^{(|k-l|,m-1)}\left(1-2\left(\tfrac{|\mu|-|\nu|}{|\mu|+|\nu|}\right)^2\right) \right| \\ 
&\times \chr_{[1-\delta,1]}(c_\gamma(l)|\nu|) \, 	\chr_{[2^j,2^{j+1})}(c_\gamma(l))\binom{l+m-1}{l} \,|\nu|^m \,d\nu.
\end{split}\]
Since the above quantity only depends on $a$ through its real part, in what follows we may assume that $a$ is real, i.e., $a = \Re(a)$.

By changing to spherical coordinates (letting $\nu=r\varrho$ and $\mu=s\sigma$ for $r,s \in (0,\infty)$ and $\varrho,\sigma \in S^{n-1}$), rotating, rescaling and then applying Lemma \ref{lem:fsest}, we obtain that
\begin{multline*}
\mathscr{I} \leq\sup_{\substack{s,k\\c(k)s\in I}}\sum_{c_\gamma(l)\in[2^j,2^{j+1})} \int_{(1-\delta)/c_\gamma(l)}^{1/c_\gamma(l)}(1+\tfrac{s}{r})^{1-n-m}|r-s|^{a/2-1}\\
\times \left|\frac{r-s}{r+s}\right|^{|k-l|}|P_{\min\{k,l\}}^{(|k-l|,m-1)}(1-2(\tfrac{r-s}{r+s})^2)|\frac{\binom{l+m-1}{l}}{\binom{\min\{k,l\}+m-1}{\min\{k,l\}}} \,dr.
\end{multline*}
For brevity, we define
\begin{multline}\label{eq:definein}
\mathscr{K} \defeq
\mathscr{K}(k,l,s)
=\int_{(1-\delta)/c_\gamma(l)}^{1/c_\gamma(l)}(1+\tfrac{s}{r})^{1-n-m}
|r-s|^{a/2-1}\\
\times\left|\frac{r-s}{r+s}\right|^{|k-l|}|P_{\min\{k,l\}}^{(|k-l|,m-1)}(1-2(\tfrac{r-s}{r+s})^2)|\frac{\binom{l+m-1}{l}}{\binom{\min\{k,l\}+m-1}{\min\{k,l\}}} \,dr.
\end{multline}

Fix $k\in\NN_0$. First, note that, if $l=k$, then the conditions $c(k)s \in I$ and $c_\gamma(l) r \in [1-\delta,1]$ imply that $s\simeq r$ (recall that $\delta \leq 1/2$), whence
\[
1+\tfrac{s}{r} \simeq 1,\qquad{\binom{l+m-1}{l}}{\binom{\min\{k,l\}+m-1}{\min\{k,l\}}}^{-1}=1,
\]
and moreover
\[
\frac{1}{2}\left(1+(1-2(\tfrac{s-r}{r+s})^2)\right) = 1-(\tfrac{r-s}{r+s})^2 = \tfrac{4rs}{(r+s)^2} \simeq 1.
\]
Then, by Theorem \ref{thm:jacobi_exist}\ref{en:jacobiest2},
\begin{equation}\label{eq:in1}
\begin{split}
    \mathscr{K} 
		&\simeq\int_{(1-\delta)/c_\gamma(l)}^{1/c_\gamma(l)}|r-s|^{a/2-1} |P_l^{(0,m-1)}(1-2(\tfrac{s-r}{r+s})^2)| \,dr\\
		&\lesssim \int_{(1-\delta)/c_\gamma(l)}^{1/c_\gamma(l)}|r-s|^{a/2-1} \,dr \\
		&\lesssim (\delta/c_\gamma(l))^{a/2} \simeq (2^{-j} \delta)^{a/2}
\end{split}
\end{equation}
whenever $c_\gamma(l) \simeq 2^j$. In estimating the last integral we used that $a > 0$ and that the value of the integral for $s\notin[\tfrac{1-\delta}{c_\gamma(l)},\tfrac{1}{c_\gamma(l)}]$ is smaller than the one for $s\in[\tfrac{1-\delta}{c_\gamma(l)},\tfrac{1}{c_\gamma(l)}]$.

Now, assume that
$1 \leq |k-l|\leq c_1(\min\{k,l\}+\tfrac{m}{2})$, where $c_1>1$ is a constant to be specified later. Then
 the conditions $c(k)s \in I$, $c_\gamma(l) r \in [1-\delta,1]$ and $c_\gamma(l) \in [2^j,2^{j+1})$ imply that
\[
1+k\simeq 1+l \simeq 2^j, \qquad r\simeq s \simeq 2^{-j},
\]
and therefore
\[
1+\tfrac{s}{r} \simeq 1, \qquad {\binom{l+m-1}{l}}{\binom{\min\{k,l\}+m-1}{\min\{k,l\}}}^{-1}\simeq 1.
\]
Moreover, note that 
\begin{equation}\label{eq:7335}
    1-2\left(\frac{s-r}{r+s}\right)^2 = -1+\frac{8rs}{(r+s)^2}  \geq -1+\epsilon,
\end{equation}
for some $\epsilon\in(0,2)$ which is independent of $r,s$ since $r\simeq s$.
Thus,
\begin{equation}\label{eq:injacobi}
\begin{split}
\mathscr{K} 
&\simeq \int_{(1-\delta)/c_\gamma(l)}^{1/c_\gamma(l)} |r-s|^{a/2-1}\left|\frac{r-s}{r+s}\right|^{|k-l|} |P_{\min\{k,l\}}^{(|k-l|,m-1)}(1-2(\tfrac{r-s}{r+s})^2)| \,dr \\
&= \mathscr{K}_1 + \mathscr{K}_2 + \mathscr{K}_3,
\end{split}
\end{equation}
where the above splitting corresponds to whether $|\tfrac{r-s}{r+s}|\geq 4|\tfrac{k-l}{k+l+m}|$, $|\tfrac{r-s}{r+s}|\leq \tfrac{1}{4}|\tfrac{k-l}{k+l+m}|$, or $\tfrac{1}{4}|\tfrac{k-l}{k+l+m}|\leq |\tfrac{r-s}{r+s}|\leq 4|\tfrac{k-l}{k+l+m}|$. Due to \eqref{eq:7335}, we may apply Theorem \ref{thm:jacobirefinedest} to estimate the Jacobi polynomial in \eqref{eq:injacobi}.

Consider first the part
 where $|\tfrac{r-s}{r+s}|\geq 4|\tfrac{k-l}{k+l+m}|$, so that
\[
|r-s|\gtrsim 2^{-2j}|k-l|.
\]
Then, by
the first estimate in \eqref{eq:jacobieq1}, 
\[\begin{split}
&|r-s|^{a/2-1}\left|\frac{r-s}{r+s}\right|^{|k-l|} |P_{\min\{k,l\}}^{(|k-l|,m-1)}(1-2(\tfrac{r-s}{r+s})^2)| \\
&\lesssim |r-s|^{a/2-1}
((k+l+m)^2 (\tfrac{r-s}{r+s})^2)^{-1/4} 
\\
&\simeq
2^{-j}
|r-s|^{a/2-3/2}  \\
&\lesssim 2^{-j(a-2)}|k-l|^{a/2-3/2},
\end{split}\]
whence
\[
\mathscr{K}_1 \lesssim 2^{-j(a-1)}|k-l|^{a/2-3/2} \delta.
\]

Next, consider the part where $|\tfrac{r-s}{r+s}|\leq \tfrac{1}{4}|\tfrac{k-l}{k+l+m}|$.
In this region we can apply
the second estimate in \eqref{eq:jacobieq1},
which
 gives that
\[
   \mathscr{K}_2 \lesssim 2^{-|k-l|}
	\int_{(1-\delta)/c_\gamma(l)}^{1/c_\gamma(l)}|r-s|^{a/2-1} \,dr
	\lesssim 2^{-|k-l|}(2^{-j}\delta)^{a/2}.
\]

Finally, consider the part where $\tfrac{1}{4}|\tfrac{k-l}{k+l+m}|\leq |\tfrac{r-s}{r+s}|\leq 4|\tfrac{k-l}{k+l+m}|$, so that
\[
|r-s|\simeq 2^{-2j}|k-l|.
\]
Here we can again apply the first estimate in \eqref{eq:jacobieq1}
and obtain that
\[\begin{split}
&|r-s|^{a/2-1}\left|\frac{r-s}{r+s}\right|^{|k-l|} |P_{\min\{k,l\}}^{(|k-l|,m-1)}(1-2(\tfrac{r-s}{r+s})^2)| \\
&\lesssim |r-s|^{a/2-1}
|k-l|^{-1/3} 
\\
&\simeq
2^{-j(a-2)}|k-l|^{a/2-4/3},
\end{split}\]
whence
\[
\mathscr{K}_3 \lesssim 2^{-j(a-1)}|k-l|^{a/2-4/3}\delta.
\]

In conclusion, for $1 \leq |k-l|\leq c_1(\min\{k,l\}+\tfrac{m}{2})$,
\begin{equation}\label{eq:smalldifference}
    \mathscr{K}
		\lesssim  2^{-j(a-1)}\delta|k-l|^{a/2-4/3}+2^{-|k-l|}(2^{-j}\delta)^{a/2}.
\end{equation}

Now, we assume that $|k-l|>c_1(\min\{k,l\}+\tfrac{m}{2})$. We consider two cases. First, let $l<k$, so that $k>c_1(l+\tfrac{m}{2})+l$.
Then, 
\[
k+\frac{m}{2}>(c_1+1)\left(l+\frac{m}{2}\right),
\]
so by Lemma \ref{lem:ckratio}, since $c_\gamma(l)>0$,
\[
c(k)>(c_1+1)c(l)\geq \frac{1}{3}(c_1+1) c_\gamma(l).
\]
Hence, from the conditions $c(k)s \in I$, $c_\gamma(l)r \in [1-\delta,1]$ and $\delta \leq 1/2$ we deduce that
\[
s\leq \tfrac{\max I}{c(k)}< \tfrac{3 \max I}{(c_1+1)c_\gamma(l)}\leq \tfrac{3 \max I}{(c_1+1)(1-\delta)}r\leq \tfrac{6 \max I}{c_1+1} r,
\]
since $\delta \leq 1/2$.
For $c_1$ sufficiently large, this means that,
\[
s \lesssim |r-s|\simeq r\simeq (1+l)^{-1} \simeq 2^{-j}, \qquad 
\frac{\binom{l+m-1}{l}}{\binom{\min\{k,l\}+m-1}{\min\{k,l\}}}=1, \qquad
1+\frac{s}{r}\simeq 1.
\]
Hence, recalling \eqref{eq:definein},
\[
\mathscr{K} \simeq \int_{(1-\delta)/c_\gamma(l)}^{1/c_\gamma(l)} r^{a/2-1}\left|\frac{r-s}{r+s}\right|^{k-l}|P_{l}^{(k-l,m-1)}(1-2(\tfrac{r-s}{r+s})^2)| \,dr.
\]
We apply Corollary \ref{cor:jacobirefinedest3} (here we need $m>1$) and the fact that
\[
|r-s|\simeq |r+s|\simeq r, \quad
1-\left(\frac{r-s}{r+s}\right)^2=\frac{4rs}{(r+s)^2}\simeq \frac{s}{r}, \quad
(1+l) r \simeq (1+k) s \simeq 1
\]
to get
\begin{equation}\label{eq:inl1}
\begin{split}
\mathscr{K}&\lesssim \int_{(1-\delta)/c_\gamma(l)}^{1/c_\gamma(l)} r^{a/2-1}(1-(\tfrac{r-s}{r+s})^2)^{-m/2+1/4}\left|\frac{r-s}{r+s}\right|^{-1/2}\\
&\qquad \times (l+1)^{-1/3} \left(\frac{l+1}{k+1}\right)^{(m-1)/2+1/4} \,dr\\
&\simeq
(l+1)^{2/3-a/2} \int_{(1-\delta)/c_\gamma(l)}^{1/c_\gamma(l)} \left(\frac{l+1}{k+1}\right)^{m/2-1/4} (r/s)^{m/2-1/4} \,dr \\
&\simeq 2^{-j(1/3+a/2)} \delta.
\end{split}
\end{equation}
Now, let $l>k$, so that $l>c_1(k+\tfrac{m}{2})+k$ and thus
\[
l+\frac{m}{2}>(c_1+1) \left(k+\frac{m}{2}\right).
\]
Then, again using Lemma \ref{lem:ckratio},
\[
3c_\gamma(l)\geq c(l) > (c_1+1)c(k),
\]
whence, from the conditions $c(k)s \in I$ and $c_\gamma(l)r \in [1-\delta,1]$ we deduce that
\[
r\leq \tfrac{1}{c_\gamma(l)} < \tfrac{3}{(c_1+1)c(k)}\leq \tfrac{3}{(c_1+1) \min I}s.
\]
Thus, for $c_1$ sufficiently large,
\[
r \lesssim |r-s|\simeq s\simeq |r+s|\simeq (1+k)^{-1}, \qquad
1+\frac{s}{r}\simeq \frac{s}{r}
\]
and
\[
\frac{\binom{l+m-1}{l}}{\binom{\min\{k,l\}+m-1}{\min\{k,l\}}} (1+\tfrac{s}{r})^{1-m} \simeq \frac{(l+1)^{m-1}r^{m-1}}{(k+1)^{m-1}s^{m-1}}\simeq 1.
\]
So, by \eqref{eq:definein},
\[
\mathscr{K} \simeq \int_{(1-\delta)/c_\gamma(l)}^{1/c_\gamma(l)} \left(\frac{r}{s}\right)^{n}s^{a/2-1}
\left|\frac{r-s}{r+s}\right|^{l-k} |P_{k}^{(l-k,m-1)}(1-2(\tfrac{r-s}{r+s})^2)| \,dr.
\]
Note also that
\[
1-\left(\frac{r-s}{r+s}\right)^2=\frac{4rs}{(r+s)^2}\simeq \frac{r}{s}, \quad l-k \simeq 1+l \simeq 2^j, \quad (1+l) r \simeq (1+k) s \simeq 1.
\]
As before, we apply Corollary \ref{cor:jacobirefinedest3} to get
\begin{equation}\label{eq:inl2}
\begin{split}
\mathscr{K} 
&\lesssim (k+1)^{-1/3} s^{a/2-1} \int_{(1-\delta)/c_\gamma(l)}^{1/c_\gamma(l)} \left(\frac{k+1}{l+1}\right)^{m/2-1/4}  \left(\frac{s}{r}\right)^{m/2-1/4-n} \,dr \\
		&\simeq (k+1)^{-1/3} s^{a/2-1-n} \int_{(1-\delta)/c_\gamma(l)}^{1/c_\gamma(l)} \,r^n \,dr\\
		&\simeq (k+1)^{n+2/3-a/2}2^{-j(n+1)}\delta\\
		&\lesssim 2^{-j(a/2+1/3)}\delta,
\end{split}
\end{equation}
where we used that $n+2/3-a/2 \geq 0$ and $1+k \lesssim 1+l \simeq 2^j$.

From \eqref{eq:in1}, \eqref{eq:smalldifference}, \eqref{eq:inl1}, \eqref{eq:inl2} we obtain that
\[
\mathscr{K} \lesssim \begin{cases}
(2^{-j} \delta)^{a/2} &\text{if } l=k,\\
2^{-j(a-1)}\delta|k-l|^{a/2-4/3}+2^{-|k-l|}(2^{-j}\delta)^{a/2} &\text{if } 1 \leq |k-l|\leq c_1(\min\{k,l\}+\tfrac{m}{2}),\\
2^{-j(a/2+1/3)}\delta &\text{if }  |k-l|>c_1(\min\{k,l\}+\tfrac{m}{2}).
\end{cases}
\]
Hence
\[\begin{split}
\sum_{c_\gamma(l) \in [2^j,2^{j+1})} \mathscr{K} 
&\lesssim (2^{-j} \delta)^{a/2} + 2^{-j(a/2-2/3)}\delta + 2^{-j(a-1)} \delta \sum_{N \lesssim 2^j} (1+N)^{a/2-4/3} \\
&\lesssim \begin{cases}
(2^{-j} \delta)^{a/2} &\text{if } 0 < a < 2/3,\\
j (2^{-j} \delta)^{a/2} &\text{if } a = 2/3,\\
(2^{-j} \delta)^{a/2} + 2^{-j(a/2-2/3)}\delta &\text{if } 2/3 < a < \min\{2,n\},\\
\end{cases}
\end{split}\]
where we used that $2^{-j(a-1)} \delta \lesssim (2^{-j} \delta)^{a/2}$ for $a \leq 2$ (this follows from the fact that $2^j \lesssim \delta^{-1}$).

Finally, since $\delta \simeq 2^{-J_\delta}$, note that
\[
  2^{-j(a/2-2/3)}\delta \lesssim (2^{-j}\delta)^{a/2} \iff 2^j \lesssim 2^{3J_\delta(2-a)/4},
\]
completing the proof.
\end{proof}

These results lead to the following `trace lemma'.

\begin{cor}\label{cor:tracelemfinalcor}
For all $\delta \in \Dyad$, all $1\leq j < J_\delta$ and all $\gamma\in\{-1,0,1\}$,
\[
\|M^\gamma_{\delta,j}f\|_2^2\lessapprox (2^{-j}\delta)^{1/3}\|f\|_{L^2((1+|\cdot|)^{2/3})}^2.
\]
\end{cor}
\begin{proof}
As noted earlier, we must defer to \cite[Lemma 7]{GM} if $m=1$. Otherwise, by Lemma \ref{lem:nonradialtraceu1},
the assumptions of Proposition \ref{prp:new_nruest} are satisfied with $\omega = \sndW$, $a = 2/3$ and $C(\delta,j) = (2^{-j} \delta)^{1/3}$ for $j<J_\delta$ (note that $2^{j} \lesssim \delta^{-1}$, and that we can trivially set $C(\delta,J_\delta) = 1$), so the desired estimate is given by Proposition \ref{prp:new_nruest}.
\end{proof}

\begin{rem}\label{rem:trace1rem}
If we instead consider the result of Lemma \ref{lem:nonradialtraceu1} with $a=1$, then the results of this section combine to prove the `stronger' estimate 
\[
\|M^\gamma_{\delta,j}f\|_2^2\lessapprox (2^{-j}\delta)^{1/2}\|f\|_{L^2(1+|\cdot|)}^2,
\]
but on a reduced range of $j$, specifically $1\leq j\leq \tfrac{3}{4} J_\delta$. Note that the same estimate also holds at the `endpoint' $j=J_\delta$ by Theorem \ref{thm:trace5}. This leaves a `middle region' $\tfrac{3}{4} J_\delta < j < J_\delta$ where pure first- or second-layer weights do not appear to be sufficient to prove this estimate.
\end{rem}

\begin{proof}[Proof of Theorem \ref{thm:tracelemfinal}]
By interpolation, it suffices to prove Theorem \ref{thm:tracelemfinal} for $a=\tfrac{2}{3}$. 
For $j=J_\delta$, this follows from from Theorem \ref{thm:trace5}, while Corollary \ref{cor:tracelemfinalcor} gives the required estimate for $j<J_\delta$.
\end{proof}

\section{Jacobi polynomials}\label{s:estforjacobipoly}

As observed in Section \ref{s:fracint}, when calculating integral kernels for fractional integration operators on the dual of an H-type group, we encounter integrals over the positive half-line of a pair of Laguerre polynomials against an exponential and a polynomial weight.
The following lemma contains a few identities that allow us to rewrite these integrals in a more manageable form; in particular, the identity \eqref{eq:lagintlem} shows that some of these integrals can be expressed in terms of Jacobi polynomials.

\begin{lem}\label{lem:laguerreintegrallem}
The following hold.
\begin{enumerate}[label=(\roman*)]
\item\label{en:laguerreintegrallem1} Let $k,l \in \NN_0$, $m \in \NN$, $a \in \CC$ with $\Re(a) \in (0,m)$. Then
\begin{multline*}
\int_0^\infty L_k^{m-1}(t) \, L_l^{m-1}(t) \, e^{-t} \,t^{m-1-a} \,dt\\
=(\Gamma(a))^{-2} \sum_{p=0}^{\min\{k,l\}} \frac{\Gamma(a+k-p)}{(k-p)!} \frac{\Gamma(a+l-p)}{(l-p)!} \frac{\Gamma(p+m-a)}{p!}.
\end{multline*}
\item\label{en:laguerreintegrallem2} Let $a,b,c>0$, $\gamma>-1$ and $l,k\in\NN_0$ with $l \leq k$. Then
\begin{multline}\label{eq:lagintlem}
\int_0^\infty L_l^{\gamma}(at) \,L_k^{\gamma}(bt) \,e^{-ct} \,t^{\gamma} \,dt \\
= \begin{cases}
 \frac{\Gamma(k+l+\gamma+1)}{l!k!} \frac{b^la^k}{c^{k+l+\gamma+1}} &\text{if } a+b=c,\\
 \frac{\Gamma(k+\gamma+1)}{k!} \frac{(c-b)^{k-l}(a+b-c)^l}{c^{k+\gamma+1}} P_{l}^{(k-l,\gamma)}\left(1-2\frac{(c-a)(c-b)}{c(c-a-b)}\right) &\text{otherwise}.
\end{cases}
\end{multline}
\end{enumerate}
\end{lem}
\begin{proof}
\ref{en:laguerreintegrallem1}. The identity \cite[eq.\ 18.18.18]{27} allows us to turn a Laguerre polynomial of type $m-1$ into a linear combination of Laguerre polynomials of type $m-1-a$; the desired identity then follows from the orthogonality relations \cite[Lemma 1.1.4]{4}.

\ref{en:laguerreintegrallem2}. Assume that $c\neq a$ and $c \neq b$ (the cases $c=a$ and $c=b$ can be recovered a posteriori by continuity). By \cite[page 175, entry (35)]{26},
\begin{multline}\label{eq:laghypergeom1}
\int_0^\infty L_l^{\gamma}(at)L_k^{\gamma}(bt) \,e^{-ct} \,t^{\gamma} \,dt 
= \frac{\Gamma(k+l+\gamma+1)}{l!k!} \frac{(c-a)^l(c-b)^k}{c^{k+l+\gamma+1}} \\
\times \HypF\left[-l,-k;-l-k-\gamma;\frac{c(c-a-b)}{(c-a)(c-b)}\right],
\end{multline}
where $\HypF$ is the hypergeometric function \cite[Chapter 15]{27}.

If $a+b=c$, then \eqref{eq:lagintlem} immediately follows, because $\HypF[-l,-k;-l-k-\gamma;0]=1$.
Suppose instead that $a+b\neq c$.
By applying the transformation formula \cite[eq.\ 15.8.6]{27},
we easily obtain that
\begin{multline}\label{eq:laghypergeom2}
\int_0^\infty L_l^{\gamma}(at)L_k^{\gamma}(bt) \,e^{-ct} \,t^{\gamma} \,dt 
= \frac{(c-b)^{k-l}}{c^{k+\gamma+1}} \frac{\Gamma(k+\gamma+1)}{l!(k-l)!} (a+b-c)^l \\
\times \HypF\left[-l,k+\gamma+1;1+k-l;\frac{(c-a)(c-b)}{c(c-a-b)}\right],
\end{multline}
and \eqref{eq:lagintlem} follows by applying the formula expressing Jacobi polynomials in terms of the hypergeometric function \cite[eq.\ 18.5.7]{27}.
\end{proof}

The remaining of this section is devoted to the discussion of
estimates for the Jacobi polynomials that appear in our formulae.

We first note some uniform, weighted bounds that are available in the literature. 

\begin{thm}\label{thm:jacobi_exist}
The following estimates hold.
\begin{enumerate}[label=(\roman*)]
\item\label{en:jacobiest2}
For all $x\in[-1,1]$, for all $\beta\geq 0$ and $\alpha\geq\beta-\lfloor\beta\rfloor$ and for all $n\in\NN_0$,
\begin{equation}
\left(\frac{1+x}{2}\right)^{\beta/2}|P_n^{(\alpha,\beta)}(x)|\leq\binom{n+\alpha}{n}.
\end{equation}
In particular, this estimate holds whenever $\alpha,\beta\in\NN_0$.
\item\label{en:jacobihs}
For all $x\in[-1,1]$, for all $\alpha,\beta\geq 0$ and for all $n\in\NN_0$,
\begin{multline}
\left(\frac{1-x}{2}\right)^{\alpha/2+1/4} \left(\frac{1+x}{2}\right)^{\beta/2+1/4} |P_n^{(\alpha,\beta)}(x)|\\
\lesssim (2n+\alpha+\beta+1)^{-1/4}\left(\frac{\Gamma(n+\alpha+1)\Gamma(n+\beta+1)}{\Gamma(n+1)\Gamma(n+\alpha+\beta+1)}\right)^{1/2}.
\end{multline}
\item\label{en:jacobikrasikov}
For all $x\in[-1,1]$, for all $\alpha\geq\beta\geq \tfrac{1+\sqrt{2}}{4}$ and for all $n\in\NN_0$ with $n\geq 6$,
\begin{multline}
\left(\frac{1-x}{2}\right)^{\alpha/2+1/4}\left(\frac{1+x}{2}\right)^{\beta/2+1/4} |P_n^{(\alpha,\beta)}(x)|\\
\lesssim \alpha^{1/6}\left(1+\frac{\alpha}{n}\right)^{1/12} \left(\frac{\Gamma(n+\alpha+1)\Gamma(n+\beta+1)}{(2n+\alpha+\beta+1)\Gamma(n+1)\Gamma(n+\alpha+\beta+1)}\right)^{1/2}.
\end{multline}
\end{enumerate}
\end{thm}
\begin{proof}
\ref{en:jacobiest2}. This is Theorem 5.1 of \cite{KKT}.

\ref{en:jacobihs}. This may be found as equation (2) in \cite{Haagerup}.

\ref{en:jacobikrasikov}. This is Theorem 2 of \cite{Krasikov}.
\end{proof}

Here is an immediate consequence of the previous estimates.

\begin{cor}\label{cor:jacobirefinedest3}
Let $\beta \in \NN$ and $c>0$. 
Then, for all $x\in [-1,1]$ and all $\alpha,n\in\NN_0$ with $\alpha\geq c(1+n)$,
\begin{multline}\label{eq:jacobirefinedest3eq} 
\left(\frac{1-x}{2}\right)^{\alpha/2+1/4}\left(\frac{1+x}{2}\right)^{\beta/2+1/4} |P_n^{(\alpha,\beta)}(x)| \\
\lesssim_{\beta,c} (n+1)^{-1/3} \left(\frac{n+1}{\alpha+1}\right)^{\beta/2+1/4}.
\end{multline}
\end{cor}
\begin{proof}
First, if $\alpha\geq\beta$ and $n\geq 6$, then this is an easy corollary of Theorem \ref{thm:jacobi_exist}\ref{en:jacobikrasikov}. If $\alpha\geq\beta$ and $0\leq n\leq 5$, then $n+1\simeq 1$ and \eqref{eq:jacobirefinedest3eq} follows from Theorem \ref{thm:jacobi_exist}\ref{en:jacobihs}. The remaining case ($\alpha < \beta$) involves only finitely many pairs $(\alpha,n)$ (note that $\beta$ is fixed and $\alpha \gtrsim_c 1+n$), so the desired estimate is trivial in this case.
\end{proof}

Next, we prove some more specialised estimates, which give sharper bounds than the above, but only on a restricted range of indices $\alpha,\beta,n$. 
The above estimates are essentially weighted $L^\infty$ estimates for Jacobi polynomials, where the weight is independent of the degree $n$. The estimates below, instead, involve a `transition point' depending on $n$, away from which much better estimates may be obtained. We proceed similarly to Proposition 3.5 of \cite{51} in order to prove such estimates.

\begin{thm}\label{thm:jacobirefinedest}
Let $\beta \in \NN_0$, $\epsilon\in(0,2)$ and $c > 0$.
Then, for all $x\in [-1+\epsilon,1]$ and all $\alpha,n\in\NN_0$ with
 $1\leq\alpha \leq c(1+n)$,
\begin{equation}\label{eq:jacobieq1}
\left|\left(\frac{1-x}{2}\right)^{\alpha/2}P_{n}^{(\alpha,\beta)}(x)\right|
\lesssim_{\beta,c,\epsilon}
\begin{cases}
(u^2|x-x_{tr}|+\alpha^{4/3})^{-1/4} &\text{in any case},\\
2^{-\alpha} &\text{if } 1-x\leq \tfrac{1}{16}(1-x_{tr});
\end{cases}
\end{equation}
here
\begin{align}
\label{eq:jacobiu}
    u = u(\alpha,\beta,n) &\defeq n+\frac{\alpha+\beta+1}{2}, \\
\label{eq:jacobitr2}
    x_{tr} = x_{tr}(\alpha,\beta,n) &\defeq 1-\frac{\alpha^2}{2u^2}.
\end{align}
\end{thm}
\begin{proof}
By means of the well-known relation
\begin{equation}\label{eq:jacobiminus}
    P_n^{(\alpha,\beta)}(x) = 
		(-1)^n P_n^{(\beta,\alpha)}(-x)
\end{equation}
and the change of variables $y=-x$, we may equivalently restate the above estimate as follows:
\begin{equation}\label{eq:jacobieq1_y}
\left|
\left(\frac{1+y}{2}\right)^{\alpha/2}
P_{n}^{(\beta,\alpha)}(y)
\right|
\lesssim_{\beta,c,\epsilon}
\begin{cases}
(u^2|y-y_{tr}|+\alpha^{4/3})^{-1/4} &\text{in any case},\\
2^{-\alpha} &\text{if } 1+y\leq \tfrac{1}{16}(1+y_{tr});
\end{cases}
\end{equation}
here $y \in [-1,1-\epsilon]$ and
\begin{equation}\label{eq:jacobitr2_y}
y_{tr} = y_{tr}(\alpha,\beta,n) \defeq \frac{\alpha^2}{2u^2} - 1.
\end{equation}

We will derive the estimate \eqref{eq:jacobieq1_y} from the asymptotic approximation for Jacobi polynomials given in \cite[Section 3]{Dunster}, which in turn makes use of the theory of \cite{49}. Namely, under our assumptions on $n,\alpha,\beta,y$, from \cite[eq.\ (3.49)]{Dunster} (applied with $N=0$) and the error bound in \cite[eq.\ below (3.11)]{49} we deduce that
\begin{multline}\label{eq:J_D_approximation}
      \left|\left(\frac{1-y}{2}\right)^{\beta/2+1/4} \left(\frac{1+y}{2}\right)^{\alpha/2} P_{n}^{(\beta,\alpha)}(y)\right| \\
			= \varkappa_{\alpha,\beta,n} 
				\Bigg| \left(\frac{\zeta-\tilde\alpha^2}{y-y_{tr}}\right)^{1/4}
				\bigl[ J_\alpha(u\zeta^{1/2})
				+ E_{\alpha}^{-1}M_{\alpha}(u\zeta^{1/2}) \, O(u^{-1}) \bigr] \Biggr|,
\end{multline}
where $\tilde\alpha = \alpha/u$,
\begin{equation}\label{eq:Dunster_constant}
\varkappa_{\alpha,\beta,n} = 2^{-1/4} \left(\frac{\Gamma(n+\alpha+1)\Gamma(n+\beta+1)}{\Gamma(n+\alpha+\beta+1)\Gamma(n+1)}\right)^{1/2}(1+O(u^{-1})),
\end{equation}
\cite[eqs.\ (3.22) and (3.34)]{Dunster}, $J_\alpha$ is the Bessel function of the first kind and index $\alpha$,
\begin{equation}\label{eq:besseldef}
    J_\alpha(x)=\sum_{m=0}^\infty \frac{(-1)^m}{m! \,\Gamma(m+\alpha+1)}\left(\frac{x}{2}\right)^{2m+\alpha},
\end{equation}
$E^{-1}_\alpha M_\alpha$ is the pointwise ratio of the auxiliary functions $E_\alpha$ and $M_\alpha$ defined in \cite[Section 3]{49}, and the relation between $y$ and $\zeta$ is implicitly given by
\begin{gather}
\label{eq:cov1}
    \int_{\tilde\alpha^2}^\zeta \frac{(\tau-\tilde\alpha^2)^{1/2}}{2\tau} \,d\tau = \int_{y_{tr}}^y \frac{(t-y_{tr})^{1/2}}{(1-t)^{1/2}(1+t)} \,dt  \qquad (y_{tr}\leq y\leq 1), \\
\label{eq:cov2}
   \int^{\tilde\alpha^2}_\zeta \frac{(\tilde\alpha^2-\tau)^{1/2}}{2\tau} \,d\tau = \int^{y_{tr}}_y\frac{(y_{tr}-t)^{1/2}}{(1-t)^{1/2}(1+t)} \,dt  \qquad (-1<y\leq y_{tr}).
\end{gather}
\cite[eqs.\ (3.7) and (3.10)]{Dunster}. We remark that the asymptotic approximation of \cite[Section 3]{Dunster} is obtained by invoking \cite[Theorem 3]{49}, which is a generalisation of \cite[Theorem 1]{49} that allows one to consider complex values of the argument; since we are only interested in real values of $y$, the approximation given by \cite[Theorem 1]{49} is enough, which justifies the simpler form of the error bound that we are using.
Furthermore, according to \cite[Section 3]{49}, the error bound is uniform provided $\zeta$ remains in a bounded interval; now, by definition, $\tilde \alpha^2 \leq 4$ and $y_{tr} \geq -1$, hence,
 by \eqref{eq:cov1}, if $\zeta \geq 4$, then
\begin{multline*}
    \int_4^\zeta \frac{(\tau-4)^{1/2}}{2\tau} \,d\tau
		\leq \int_{\tilde\alpha^2}^\zeta \frac{(\tau-\tilde\alpha^2)^{1/2}}{2\tau} \,d\tau
		=
		\int_{y_{tr}}^y\frac{(t-y_{tr})^{1/2}}{(1-t)^{1/2}(1+t)} \,dt \\
		\leq \int_{-1}^{1}\frac{1}{(1-t^2)^{1/2}} \,dt
		=\pi
		=\int_4^{\zeta_1} \frac{(\tau-4)^{1/2}}{2\tau} \,d\tau
\end{multline*}
for some $\zeta_1\in(4,\infty)$ independent of all parameters, so that $\zeta \in [0,\zeta_1]$.

From our assumptions on $\alpha,\beta,n,y$ and \eqref{eq:Dunster_constant} it is easily derived that
\[
\kappa_{\alpha,\beta,n} \simeq_{\beta,c} 1, \qquad 1-y \simeq_\epsilon 1;
\]
hence \eqref{eq:J_D_approximation} immediately gives that
\begin{equation}\label{eq:J_D_approximation_s}
      \left|\left(\frac{1+y}{2}\right)^{\alpha/2} P_{n}^{(\beta,\alpha)}(y)\right| 
			\lesssim_{\beta,c,\epsilon} 
			\left|\frac{\zeta-\tilde\alpha^2}{y-y_{tr}}\right|^{1/4}
				\bigl[ |J_\alpha(u\zeta^{1/2})|
				+ E_{\alpha}^{-1}M_{\alpha}(u\zeta^{1/2}) \bigr] .
\end{equation}

Note that, by \cite[Section 12.1.3]{53}, the pointwise estimate
\begin{equation}\label{eq:JM_bd}
    |J_\alpha| \leq E^{-1}_\alpha M_\alpha \leq M_\alpha
\end{equation}
holds, and furthermore, by \cite[Appendix B, Lemma 2]{49}, the quantity
\[
u^{1/2} |\zeta - \tilde\alpha^2|^{1/4} M_\alpha(u\zeta^{1/2})
\]
is uniformly bounded. Thus, from \eqref{eq:J_D_approximation_s} we immediately deduce that
\begin{equation}\label{eq:jacobiunif}
\left|\left(\frac{1+y}{2}\right)^{\alpha/2}P_{n}^{(\beta,\alpha)}(y)\right|\lesssim_{\beta,c,\epsilon} (u^{2} |y-y_{tr}|)^{-1/4}.
\end{equation}
Hence, in order to conclude the proof of the first estimate of \eqref{eq:jacobieq1}, it is enough to prove the uniform bound
\begin{equation}\label{eq:jacobiunif_unif}
\left|\left(\frac{1+y}{2}\right)^{\alpha/2}P_{n}^{(\beta,\alpha)}(y)\right|\lesssim_{\beta,c,\epsilon} \alpha^{-1/3}.
\end{equation}

Now, define $I$ to be the interval of the $y \in [-1,1]$ satisfying
\begin{equation}\label{eq:transitioninterval}
    \frac{2}{3}(1+y_{tr})\leq 1+y\leq \frac{3}{2}(1+y_{tr}).
\end{equation}
We first observe that, for all $y\notin I$,
\begin{multline}\label{eq:jacobienough}
    (u^{2} |y-y_{tr}|)^{-1/4}
		= (u^{2} |(1+y)-(1+y_{tr})|)^{-1/4} \\
		\lesssim (u^{2}|1+y_{tr}|)^{-1/4}
		=(\alpha^2/2)^{-1/4}
		\simeq \alpha^{-1/2}\leq \alpha ^{-1/3}. 
\end{multline}
This shows that, if $y \notin I$, then \eqref{eq:jacobiunif} implies \eqref{eq:jacobiunif_unif}; so we only need to prove \eqref{eq:jacobiunif_unif} for $y \in I$.

We now claim that, for $y \in I \cap [-1,1-\epsilon]$,
\begin{equation}\label{eq:alphaytrratio}
    \frac{\zeta-\tilde\alpha^2}{y-y_{tr}} \simeq_{\beta,c,\epsilon} 1.
\end{equation}

If we assume this claim, then from \eqref{eq:J_D_approximation_s} and \eqref{eq:JM_bd} we deduce that, for $y \in I \cap [-1,1-\epsilon]$,
\begin{equation}\label{eq:simpler_est}
      \left|\left(\frac{1+y}{2}\right)^{\alpha/2} P_{n}^{(\beta,\alpha)}(y)\right| 
			\lesssim_{\beta,c,\epsilon} 
				|J_\alpha(u\zeta^{1/2})|
				+ E_{\alpha}^{-1}M_{\alpha}(u\zeta^{1/2}) 
			\lesssim
				 M_{\alpha}(u\zeta^{1/2}).				
\end{equation}
On the other hand, for each $\alpha \in \NN$,
$M_\alpha$ is a bounded continuous function on $\RR^+$ \cite[eqs.\ (1.23) and (1.24), p.\ 437]{53}, whence the bound \eqref{eq:jacobiunif_unif} trivially holds for each fixed $\alpha \in \NN$, and it is enough to prove \eqref{eq:jacobiunif_unif} for $\alpha \geq \alpha_0$ for some large $\alpha_0 \in \NN$.

Note that $E_\alpha^{-1}M_\alpha(x) = \sqrt{2} J_\alpha(x)$ for all $x \in [0,X_\alpha]$, where $X_\alpha$ is defined in \cite[Section 12.1.3]{53} and satisfies
\[
X_\alpha = \alpha + 2c\alpha^{1/3} + O(\alpha^{-1/3})
\]
for some $c \in (0,1)$ as $\alpha \to \infty$ \cite[Chapter 12, Ex.\ 1.1, p.\ 438]{53}.
is a fixed constant which may be inferred from \cite{53}. Thus, there exists $\alpha_0 \in \NN$ such that, for all $\alpha \geq \alpha_0$,
\[
X_\alpha \geq \alpha(1+c\alpha^{-2/3}).
\]
In particular, if $\alpha \geq \alpha_0$ and $u\zeta^{1/2} \leq \alpha(1+c\alpha^{-2/3})$, then from \eqref{eq:simpler_est} we deduce that
\begin{equation}\label{eq:jacobi_bessel}
      \left|\left(\frac{1+y}{2}\right)^{\alpha/2} P_{n}^{(\beta,\alpha)}(y)\right| 
			\lesssim_{\beta,c,\epsilon} 
				|J_\alpha(u\zeta^{1/2})| \lesssim \alpha^{-1/3},
\end{equation}
where we used the uniform bound for Bessel functions,
\begin{equation}\label{eq:besselbound}
    |J_\alpha(x)|\lesssim \alpha^{-1/3},
\end{equation}
for all $\alpha,x > 0$, discussed in \cite{52}; this proves \eqref{eq:jacobiunif_unif} in this case. If instead $u\zeta^{1/2} \geq \alpha(1+c\alpha^{-2/3})$, then
\[
\zeta \geq \tilde\alpha^2 (1+2c\alpha^{-2/3}),
\]
and therefore, by \eqref{eq:alphaytrratio},
\[
u^2 (y-y_{tr}) \simeq u^2 (\zeta-\tilde\alpha^2) \geq 2c \alpha^{4/3},
\]
which again implies $(u^2 |y-y_{tr}|)^{-1/4} \lesssim \alpha^{-1/3}$, so in this case \eqref{eq:jacobiunif_unif} follows from \eqref{eq:jacobiunif}. This concludes the proof of the first estimate in \eqref{eq:jacobieq1_y}, conditional to the validity of the claim \eqref{eq:alphaytrratio}.

We now prove the second estimate in \eqref{eq:jacobieq1_y}.
Due to the uniform bound given by the first estimate in \eqref{eq:jacobieq1_y}, it is clearly enough to prove the second estimate in \eqref{eq:jacobieq1_y} for $\alpha \geq \alpha_0$.
Note now that
\[
1+y_{tr} = \tilde\alpha^2/2.
\]
Hence, by \eqref{eq:cov2}, if $y \leq y_{tr}$ then
\[\begin{split}
   \int^{\tilde\alpha^2}_\zeta \frac{(\tilde\alpha^2-\tau)^{1/2}}{2\tau} \,d\tau
	&=\int^{1+y_{tr}}_{1+y} \frac{(1+y_{tr}-t)^{1/2}}{(2-t)^{1/2}} \frac{dt}{t}\\
	&=\int^{\tilde\alpha^2}_{2(1+y)} \frac{(\tilde\alpha^2-\tau)^{1/2}}{(4-\tau)^{1/2}} \frac{d\tau}{\tau}
	\geq \int^{\tilde\alpha^2}_{2(1+y)} \frac{(\tilde\alpha^2-\tau)^{1/2}}{2\tau} \,d\tau.
\end{split}\]
Since the integrand is non-negative, this is only possible if
\[
\zeta\leq 2(1+y).
\]
Under the assumption $1+y \leq \tfrac{1}{16} (1+y_{tr}) = \tfrac{1}{32} \tilde\alpha^2$, this implies that
\begin{equation}\label{eq:zeta_bound}
u\zeta^{1/2} \leq 
\frac{\alpha}{4},
\end{equation}
and therefore the bound \eqref{eq:jacobi_bessel} applies.
From \cite[eq.\ 10.14.4]{27}
we deduce that, for all $\alpha\geq -\tfrac{1}{2}$ and all $x\in\RR$,
\[
    |J_\alpha(x)|\leq \frac{|x/2|^\alpha}{\Gamma(\alpha+1)}.
\]
Thus, by \eqref{eq:jacobi_bessel}, \eqref{eq:zeta_bound} and Stirling's formula,
\[
     \left|\left(\frac{1+y}{2}\right)^{\alpha/2}P_{n}^{(\beta,\alpha)}(y)\right|
		\lesssim_{\beta,c,\epsilon}|J_\alpha(u\zeta^{1/2})|
		\leq \frac{|u\zeta^{1/2}/2|^\alpha}{\alpha!}\\
		\lesssim \alpha^{-1/2} \left(\frac{e}{4}\right)^\alpha 2^{-\alpha}
		\leq 2^{-\alpha},
\]
which proves the second estimate in \eqref{eq:jacobieq1_y}.

\bigskip

Finally, we prove the claim \eqref{eq:alphaytrratio}.
First, assume that
\begin{equation}\label{eq:positive_interval}
 1+y_{tr} \leq 1+y\leq \frac{3}{2}(1+y_{tr})
\end{equation}
so \eqref{eq:cov1} is applicable.
Now, recalling that $1+y_{tr}=\tilde\alpha^2/2$ and that $\epsilon\leq 1-y\leq 1-y_{tr}\leq 2$, from \eqref{eq:positive_interval} we deduce that
\begin{multline}\label{eq:worsefy}
		\int_{y_{tr}}^y \frac{(t-y_{tr})^{1/2}}{(1-t)^{1/2}(1+t)} \,dt
		\simeq_\epsilon \frac{1}{1+y_{tr}} \int_{y_{tr}}^y (t-y_{tr})^{1/2} \,dt \\
		\simeq \frac{(y-y_{tr})^{3/2}}{\tilde\alpha^2}
		\lesssim (y-y_{tr})^{1/2},
\end{multline}
where we used the fact that, by \eqref{eq:positive_interval},
\begin{equation}\label{eq:yalpha}
    y-y_{tr}=(1+y)-(1+y_{tr})\leq (1+y_{tr})/2 = \tilde\alpha^2/4
\end{equation}

We now claim that, under the assumption \eqref{eq:positive_interval},
\begin{equation}\label{eq:Calpha}
    \zeta \simeq_\epsilon \tilde\alpha^2.
\end{equation}
This is certainly true if $\zeta \leq 2\tilde\alpha^2$, since we already know that $\zeta \geq \tilde\alpha^2$. Suppose instead that $\zeta \geq 2\tilde\alpha^2$; 
then
\[
	\int_{\tilde\alpha^2}^\zeta \frac{(\tau-\tilde\alpha^2)^{1/2}}{2\tau} \,d\tau
	\geq \frac{1}{2\zeta} \int_{\tilde\alpha^2}^\zeta(\tau-\tilde\alpha^2)^{1/2} \,d\tau
	= \frac{1}{3}\frac{(\zeta-\tilde\alpha^2)^{3/2}}{\zeta} \simeq \zeta^{1/2}.
\]
Combining this with \eqref{eq:cov1}, \eqref{eq:worsefy} and \eqref{eq:yalpha} proves that
\[
\zeta \lesssim_\epsilon y-y_{tr} \leq \tilde\alpha^2/4,
\]
whence \eqref{eq:Calpha} follows.

Now, from \eqref{eq:Calpha} we deduce that
\[
	\int_{\tilde\alpha^2}^\zeta \frac{(\tau-\tilde\alpha^2)^{1/2}}{2\tau} \,d\tau
	\simeq_\epsilon \frac{1}{\tilde\alpha^2} \int_{\tilde\alpha^2}^\zeta(\tau-\tilde\alpha^2)^{1/2} \,d\tau
	\simeq
	\frac{(\zeta-\tilde\alpha^2)^{3/2}}{\tilde\alpha^2},
\]
which, combined with \eqref{eq:cov1} and \eqref{eq:worsefy}, gives that
\[
    \frac{(\zeta-\tilde\alpha^2)^{3/2}}{\tilde\alpha^2}
		\simeq_\epsilon \frac{(y-y_{tr})^{3/2}}{\tilde\alpha^2},
\]
that is, \eqref{eq:alphaytrratio}.

Assume now that
\begin{equation}\label{eq:negative_interval}
\frac{2}{3}(1+y_{tr})\leq 1+y\leq (1+y_{tr}),
\end{equation}
which makes \eqref{eq:cov2} applicable. Note that, under our assumptions on $n,\alpha,\beta$,
\[
\frac{1-y_{tr}}{2} =  \frac{2u-\alpha}{2u} \cdot \frac{2u+\alpha}{2u} \geq \frac{2n+\beta+1}{2n+\alpha+\beta+1} \geq \frac{2n+\beta+1}{(2+c)n+c+\beta+1} \gtrsim_{\beta,c} 1,
\]
which implies that $1-y_{tr} \simeq_{\beta,c} 1-y \simeq_{\beta,c} 1$, and moreover
\[
y_{tr}-y = (1+y_{tr}) - (1+y) \leq \frac{1}{3} (1+y_{tr}) =\tilde\alpha^2/6.
\]
From \eqref{eq:cov2} we then deduce that
\begin{multline}\label{eq:weaker_lower}
\frac{\tilde\alpha}{2\sqrt{2}} \log_+\left(\frac{\tilde\alpha^2}{2\zeta}\right)
\leq \int^{\tilde\alpha^2/2}_{\min\{\zeta,\tilde\alpha^2/2\}} \frac{(\tilde\alpha^2-\tau)^{1/2}}{2\tau} \,d\tau 
\leq \int^{\tilde\alpha^2}_\zeta \frac{(\tilde\alpha^2-\tau)^{1/2}}{2\tau} \,d\tau \\
= \int^{y_{tr}}_y\frac{(y_{tr}-t)^{1/2}}{(1-t)^{1/2}(1+t)} \,dt 
\simeq_{\beta,c} \frac{(y_{tr}-y)^{3/2}}{\tilde\alpha^2}
 \lesssim \tilde\alpha,
\end{multline}
which again implies that
\[
\zeta \simeq_{\beta,c} \tilde\alpha^2
\]
(note that we already know that $\zeta \leq \tilde\alpha^2$ in this case). Consequently
\[
\int^{\tilde\alpha^2}_\zeta \frac{(\tilde\alpha^2-\tau)^{1/2}}{2\tau} \,d\tau \simeq_{\beta,c} \frac{(\tilde\alpha^2-\zeta)^{3/2}}{\tilde\alpha^2},
\]
and again \eqref{eq:cov2} and \eqref{eq:weaker_lower} give that
\[
\frac{(\tilde\alpha^2-\zeta)^{3/2}}{\tilde\alpha^2} \simeq_{\beta,c} \frac{(y_{tr}-y)^{3/2}}{\tilde\alpha^2},
\]
that is, \eqref{eq:alphaytrratio}.
\end{proof}

\end{document}